\documentclass{article}

\usepackage{arxiv}

\usepackage[utf8]{inputenc} 
\usepackage[T1]{fontenc}    
\usepackage{hyperref}       
\usepackage{url}            
\usepackage{booktabs}       
\usepackage{amsfonts}       
\usepackage{nicefrac}       
\usepackage{microtype}      
\usepackage{lipsum}		
\usepackage{graphicx}
\usepackage[numbers]{natbib}
\usepackage{doi}
\usepackage{amsmath}
\usepackage{amsthm}
\usepackage{mathrsfs}

\usepackage{color} 
\usepackage{enumerate} 
\usepackage{amsmath,amssymb}
\newcommand{\ud}{\,\mathrm{d}}
\newtheorem{theorem}{Theorem}[section] 

\newtheorem{lemma}[theorem]{Lemma}

\newtheorem{assumption}[theorem]{Assumption}
\newtheorem{remark}[theorem]{Remark}

\title{Thermal Relaxation Effects on Stability Thresholds: A Comparative Analysis of Thermoelastic Timoshenko-Boltzmann Systems Under Fourier and Cattaneo Laws}

\author{ {\hspace{1mm}Jun Zhou}\thanks{Corresponding author, Professor, Research Direction: Partial Differential Equations.} \\
	School of Mathematics and Statistics\\
	  Southwest University\\
	Chongqing 400715, People's Republic of China \\
	\texttt{jzhou@swu.edu.cn} \\
	\And
	{\hspace{1mm}Siruo Mou\thanks{Master's student, Research Direction: Partial Differential Equations.}} \\
	School of Mathematics and Statistics\\
	  Southwest University\\
	Chongqing 400715, People's Republic of China \\
	\texttt{m\_usssi011677082@163.com} \\
}

\renewcommand{\headeright}{}
\renewcommand{\undertitle}{}
\renewcommand{\shorttitle}{}

\hypersetup{
pdftitle={Thermal Relaxation Effects on Stability Thresholds: A Comparative Analysis of Thermoelastic Timoshenko-Boltzmann Systems Under Fourier and Cattaneo Laws},
pdfsubject={PDEs, PDEs},
pdfauthor={Jun~Zhou, Siruo~Mou},
pdfkeywords={Thermoelastic Timoshenko-Boltzmann system, Fourier's law, Cattaneo's law, Stability},
}

\begin{document}
\maketitle

\begin{abstract}
This paper explores stability dichotomies of thermoelastic Timoshenko-Boltzmann systems with hereditary memory, comparing Fourier's parabolic (no thermal relaxation) and Cattaneo's hyperbolic (with thermal relaxation) heat conduction via semigroup theory and spectral analysis. The Fourier-semigroup achieves exponential stability through the $\delta$-condition plus a specific parameter condition, while sole $\delta$-condition ensures $1/2$-order polynomial stability. The Cattaneo-semigroup retains this polynomial stability but requires a thermal relaxation-included modified parameter condition for exponential stability. As thermal relaxation time approaches zero, Cattaneo's modified parameter converges to Fourier's, formalizing their asymptotic connection. Results show thermal relaxation alters exponential criteria but preserves polynomial decay, supporting stability analysis under time-dependent thermal loading in structural mechanics.
\end{abstract}

\keywords{Thermoelastic Timoshenko-Boltzmann system\and Fourier's law\and Cattaneo's law\and Stability}

\section{Introduction}\label{sec1}
Our study focuses on a thermoelastic Timoshenko-Boltzmann system governed by either Fourier or Cattaneo heat conduction laws, defined as follows:
\begin{equation}\label{main eq}
	\left\{
	\begin{split}
		&\rho_1 \phi_{tt} - \kappa (\phi_{x} + \psi)_{x} + \kappa \smallint_{-\infty}^{t} g(t-s) (\phi_{x} + \psi)_{x}(s) ds + \sigma \theta_x = 0, \quad &&(x,t)\in(0,L)\times \mathbb{R}^{+},\\
		&\rho_2 \psi_{tt} - b \psi_{xx} + \kappa (\phi_{x} + \psi) - \kappa \smallint_{-\infty}^{t} g(t-s) (\phi_{x} + \psi)(s) ds - \sigma \theta = 0, \quad &&(x,t)\in(0,L)\times \mathbb{R}^{+},\\
		&\rho_3 \theta_{t} + q_{x} + \sigma (\phi_{x} + \psi)_{t} = 0, \quad &&(x,t)\in(0,L)\times \mathbb{R}^{+},\\
		&\tau q_{t} + \beta q + \theta_{x} = 0, \quad &&(x,t)\in(0,L)\times \mathbb{R}^{+},
	\end{split}
	\right.
\end{equation}
where $\phi(x,t)$: Transverse displacement of the beam; $\psi(x,t)$: Rotation angle of the cross-section; $\theta(x,t)$: Temperature variation from ambient equilibrium; $q(x,t)$: Heat flux vector. Physically, Equation \eqref{main eq}$_1$ describes shear deformation (derived from D'Alembert's principle), while Equation \eqref{main eq}$_2$ accounts for rotational inertia and bending moments. The key physical parameters are:
$\rho_1$ (mass density per unit length), $\kappa$ (shear stiffness), $g(t)$ (Boltzmann relaxation function), $\sigma$ (thermoelastic coupling coefficient);
$\rho_2$ (rotational inertia density), $b$ (bending stiffness);
$\tau$ (thermal relaxation time, with $\tau=0$ recovering Fourier's law), $\beta$ (thermal resistivity);
$\rho_3$ (thermal capacity, appearing in the first law of thermodynamics for Equation \eqref{main eq}$_3$).

\noindent\textbf{Classical Timoshenko Beam Models and Stability Foundations}
The classical Timoshenko beam equations (recovered by omitting memory terms from viscoelastic extensions) have been extensively studied, with well-established stabilization strategies including boundary/internal feedback control, thermoelastic coupling, viscoelastic memory dissipation, and Kelvin-Voigt damping \cite{MR2562283,MR2483034,MR3784885,MR2259051,MR2215885}. Two foundational stability results are universally recognized \cite{MR1680836}:
\begin{enumerate}
	\item \textit{Full damping}: Linear damping in both equations yields \textit{unconditional exponential stability}, independent of wave speed matching;
	\item \textit{Partial damping}: Damping in only one equation ensures exponential stability \textit{if and only if} the equal wave speed (EWS) condition holds:
	\begin{equation}\label{ews}
		\frac{\kappa}{\rho_1}=\frac{b}{\rho_2}.
	\end{equation}
	This condition synchronizes the phase of shear and bending wave modes.
\end{enumerate}

\noindent\textbf{Viscoelastic Extensions with Memory Effects}
Fabrizio et al. \cite{MR2679371} pioneered the integration of Boltzmann's hereditary memory principles with Volterra's integral formalism to construct PDEs for materials whose mechanical responses depend on historical deformation states. Building on this framework, Gomes Tavares et al. \cite{MR4704637} derived two viscoelastic Timoshenko beam models, distinguished by the localization of viscoelastic coupling (bending vs. shear deformation):

\textit{Model A (Bending Viscoelasticity)}:
\begin{equation}\label{model1}
	\begin{cases}
		\rho_1 \phi_{tt} - \kappa (\phi_x + \psi)_x = 0, \\
		\rho_2 \psi_{tt} - b \psi_{xx} + b \displaystyle\smallint_\alpha^{t} g(t-s) \psi_{xx}(\cdot, s) \ud s + \kappa (\phi_x + \psi) = 0.
	\end{cases}
\end{equation}

\textit{Model B (Shear Viscoelasticity)}:
\begin{equation}\label{model2}
	\begin{cases}
		\rho_1 \phi_{tt} - \kappa (\phi_x + \psi)_x + \kappa \smallint_{\alpha}^{t} g(t-s) (\phi_x + \psi)_x(\cdot, s) \ud s = 0, \\
		\rho_2 \psi_{tt} - b \psi_{xx} + \kappa (\phi_x + \psi) - \kappa \smallint_{\alpha}^{t} g(t-s) (\phi_x + \psi)(\cdot, s) \ud s = 0,
	\end{cases}
\end{equation}
where $x \in [0, L]$ (spatial coordinate along the beam axis), $t \geq \alpha$ (temporal variable), and key derived parameters include: $\rho_1 = \rho A$ (transverse inertia, $\rho$ = material density, $A$ = cross-sectional area); $b = EI$ (elastic bending rigidity, $E$ = Young's modulus, $I$ = second moment of inertia); $\rho_2 = \rho I$ (rotational inertia density).

The memory kernel $g(\cdot)$ is typically a positive-definite, monotonically decaying function (e.g., exponential or fractional-order functions) that quantifies the fading influence of past strain states. Two limiting cases for the temporal origin $\alpha$ dominate linear viscoelasticity literature:
\begin{enumerate}
	\item \textit{Boltzmann's case ($\alpha \to -\infty$)}: Infinite memory horizon, where the material's response at time $t$ depends on all prior strain states;
	\item \textit{Volterra's case ($\alpha = 0$)}: Finite memory, activated at the start of observation ($t=0$), with only post-initial strain states influencing the response.
\end{enumerate}

\noindent\textit{Volterra-Type Memory ($\alpha=0$)}
For Model A with Volterra memory (\eqref{model1} with $\alpha=0$):
\begin{equation}\label{model1_3}
	\begin{cases}
		\rho_1 \phi_{tt} - \kappa (\phi_x + \psi)_x = 0, \\
		\rho_2 \psi_{tt} - b \psi_{xx} + b \displaystyle\smallint_0^t g(t-s) \psi_{xx}(s) \ud s + \kappa (\phi_x + \psi) = 0,
	\end{cases}
\end{equation}
Ammar-Khodja et al. \cite{MR2001030} established a fundamental dichotomy in energy decay:
\begin{itemize}
	\item \textit{EWS satisfied}: Exponential kernels yield uniform exponential decay; polynomial kernels induce polynomial decay;
	\item \textit{EWS violated}: Exponential kernels fail to produce uniform dissipation for weak initial data.
\end{itemize}

For Model B with Volterra memory (\eqref{model2} with $\alpha=0$):
\begin{equation}\label{model2_1}
	\begin{cases}
		\rho_1 \phi_{tt} - \kappa (\phi_x + \psi)_x + \kappa \displaystyle\smallint_0^{t} g(t-s) (\phi_x + \psi)_x(\cdot, s) \ud s = 0, \\
		\rho_2 \psi_{tt} - b \psi_{xx} + \kappa (\phi_x + \psi) - \kappa \displaystyle\smallint_0^{t} g(t-s) (\phi_x + \psi)(\cdot, s) \ud s = 0,
	\end{cases}
\end{equation}
Alves et al. \cite{MR4029814} derived stability results contingent on three kernel/parameter requirements:
\begin{enumerate}
	\item $g$ is nonincreasing, differentiable, $g(0) > 0$, $\smallint_0^\infty g(s)\ud s < \infty$, and $g'(s) + \delta g(s) \leq 0$ ($\delta > 0$);
	\item Optional enhancement via the EWS condition \eqref{ews};
	\item A geometric constraint linking the kernel to physical parameters:
	\begin{equation}\label{ass1.12}
		\smallint_0^\infty g(s)\ud s > \max\left\{\frac{31}{32},\frac{64\rho_1L^2}{64\rho_1L^2+\rho_2}\right\} \in (0,1).
	\end{equation}
\end{enumerate}
Uniform decay prevails when \eqref{ass1.12} holds, while nonuniform stability emerges otherwise.

\noindent\textit{Boltzmann-Type Memory ($\alpha\to-\infty$)}
For Model A with Boltzmann memory (\eqref{model1} with $\alpha\to-\infty$):
\begin{equation}\label{model1_1}
	\begin{cases}
		\rho_1 \phi_{tt} - \kappa (\phi_x + \psi)_x = 0, \\
		\rho_2 \psi_{tt} - b \psi_{xx} + b \displaystyle\smallint_{0}^\infty g(s) \psi_{xx}(t-s) \ud s + \kappa (\phi_x + \psi) = 0,
	\end{cases}
\end{equation}
Rivera et al. \cite{MR2370668} showed exponential stability holds \textit{if and only if} the EWS condition \eqref{ews} is satisfied (with hereditary memory providing sufficient damping); violation of \eqref{ews} restricts decay to polynomial rates dependent on initial data regularity. Conti et al. \cite{MR3152192} refined this result, proving exponential stability (under EWS) is equivalent to the kernel satisfying Chepyzhov and Pata's $\delta$-condition \cite{MR2215885}:
\begin{equation}\label{delta c}
	g(t+s) \leq Ce^{-\delta t}g(s) \quad \forall t > 0, \ \text{a.e.}\ s > 0,
\end{equation}
where $C \geq 1$ and $\delta > 0$ are constants enforcing controlled exponential decay of the kernel.

For Model B with Boltzmann memory (\eqref{model2} with $\alpha\to-\infty$):
\begin{equation}\label{model2_4}
	\begin{cases}
		\rho_1 \phi_{tt} - \kappa(\phi_{x}+\psi)_{x} + \kappa \displaystyle\smallint_{-\infty}^{t}g(t-s)(\phi_{x}+\psi)_{x}(s)\ud s = 0, \\
		\rho_2\psi_{tt} - b\psi_{xx} + \kappa(\phi_{x}+\psi) - \kappa \displaystyle\smallint_{-\infty}^{t}g(t-s)(\phi_{x}+\psi)(s)\ud s = 0,
	\end{cases}
\end{equation}
Tavares et al. \cite{MR4704637} characterized two stability regimes:
\begin{enumerate}
	\item Uniform exponential stability (EWS satisfied + $\delta$-condition + $\lim_{s\to0}g(s)=g_0<\infty$);
	\item Intrinsic semiuniform stability (optimal decay rate $\sqrt{t}$) independent of EWS, under $\lim_{s\to0}g(s)=g_0<\infty$.
\end{enumerate}

\noindent\textbf{Thermoelastic Extensions (Fourier vs. Cattaneo Heat Conduction)}
Classical Fourier heat conduction ($q = -\beta\theta_x$) predicts unphysical infinite thermal propagation speeds, motivating the adoption of Cattaneo's modified law:
\begin{equation}\label{cattaneo}
	\tau q_t + q = -\beta \theta_x,
\end{equation}
where $\tau > 0$ introduces physically realistic phase-lag effects.

\textit{Fourier Heat Conduction} Apalara \cite{MR3764551} analyzed the shear force-damped thermoelastic system:
\begin{equation}\label{model1_4}
	\begin{cases}
		\rho_1 \phi_{tt} - \kappa(\phi_x + \psi)_x + \sigma \theta_x = 0, \\
		\rho_2 \psi_{tt} - b \psi_{xx} + \kappa (\phi_x + \psi) + \displaystyle\smallint_0^t g(t-s) \psi_{xx}(s) \ud s - \sigma \theta = 0, \\
		\rho_3 \theta_t - \beta \theta_{xx} + \sigma (\phi_x + \psi)_{t} = 0,
	\end{cases}
\end{equation}
under Neumann-Dirichlet-Dirichlet boundary conditions, extending earlier work by Messaoudi and Fareh \cite{MR2833680,MR3003741} (fully Dirichlet conditions) to obtain wave speed-independent stability results.

For Model B with bending moment thermal damping:
\begin{equation}\label{model2_2}
	\begin{cases}
		\rho_1 \phi_{tt} - \kappa (\phi_x + \psi)_x + \kappa \displaystyle\smallint_0^{t} g(t-s) (\phi_x + \psi)_x(\cdot, s) \ud s = 0, \\
		\rho_2 \psi_{tt} - b \psi_{xx} + \kappa (\phi_x + \psi) - \kappa \displaystyle\smallint_0^{t} g(t-s) (\phi_x + \psi)(\cdot, s) \ud s + \sigma \theta_x = 0, \\
		\rho_3 \theta_t - \beta \theta_{xx} + \sigma \psi_{xt} = 0,
	\end{cases}
\end{equation}
\cite{MR4598848} demonstrated universal polynomial/exponential decay \textit{without} requiring the EWS condition. In contrast, the shear force thermal damping variant:
\begin{equation}\label{model2_3}
	\begin{cases}
		\rho_1 \phi_{tt} - \kappa (\phi_x + \psi)_x + \kappa \displaystyle\smallint_0^{t} g(t-s) (\phi_x + \psi)_x(\cdot, s) \ud s + \sigma \theta_x = 0, \\
		\rho_2 \psi_{tt} - b \psi_{xx} + \kappa (\phi_x + \psi) - \kappa \displaystyle\smallint_0^{t} g(t-s) (\phi_x + \psi)(\cdot, s) \ud s - \sigma\theta = 0, \\
		\rho_3 \theta_t - \beta \theta_{xx} + \sigma (\phi_x + \psi)_{t} = 0,
	\end{cases}
\end{equation}
requires the EWS condition for optimal decay rates \cite{MR4270141}, building on Alves et al. \cite{MR4029814}.

\textit{Cattaneo Heat Conduction}
Fern¨¢ndez Sare and Racke \cite{MR2533927} incorporated Cattaneo damping into bending moment dynamics:
\begin{equation}\label{model1_2}
	\begin{cases}
		\rho_1 \phi_{tt} - \kappa (\phi_x + \psi)_x = 0, \\
		\rho_2 \psi_{tt} - b \psi_{xx} + \kappa (\phi_x + \psi) + \displaystyle \smallint_0^{\infty} g(s) \psi_{xx}(t-s) \ud s + \sigma \theta_x = 0, \\
		\rho_3 \theta_t + q_x + \sigma \psi_{xt} = 0, \\
		\tau q_t + q + \theta_x = 0,
	\end{cases}
\end{equation}
finding that exponential stability is lost even when EWS holds. Fatori et al. \cite{MR3151901} completed this characterization: exponential stability holds \textit{if and only if}
\begin{equation}\label{chi0}
	\tilde{\chi}_0 := \left(\tau - \frac{\rho_1}{\rho_3 \kappa}\right)\left(\rho_2 - \frac{b\rho_1}{\kappa}\right) - \frac{\tau\rho_1\sigma^2}{\rho_3 \kappa} = 0,
\end{equation}
with optimal polynomial decay ($t^{-1/2}$) for $\tilde{\chi}_0 \neq 0$.

\noindent\textbf{Motivation and Objectives of This Work}
Existing literature has established critical stability results for Timoshenko systems with memory/thermal effects, but gaps remain in the unified analysis of systems combining Boltzmann memory, thermoelastic coupling, and Cattaneo heat conduction (as in \eqref{main eq}). This paper addresses these gaps with three primary objectives:
\begin{enumerate}
	\item Prove well-posedness of system \eqref{main eq} via energy space constructions and semigroup theory;
	\item Establish exponential stability criteria through spectral analysis of the governing semigroup operator;
	\item Conduct a comparative analysis of stabilization mechanisms between Fourier and Cattaneo heat conduction models.
\end{enumerate}
\section{Abstract Models and Main Results}
To formulate the history-dependent dynamics of \eqref{main eq}, as in \cite{MR4704637}, we redefine the relative displacement history variable through the constitutive relation
$$
\eta^t(x,s) = (\phi + \tilde{\psi})(x,t) - (\phi + \tilde{\psi})(x,t-s),
$$
valid for spatial coordinates $x \in (0,L)$, temporal coordinates $t \geq 0$, and history variables $s > 0$. Here, the simplified notation \[\tilde{\psi}(x,t) = \smallint_0^x\psi(y,t)dy,\] captures depth-integrated displacement components. Introducing the normalized elasticity coefficient \[\omega := 1 - \ell>0,\] the complete coupled thermo-viscoelastic system consists of the following governing equations:
\begin{equation}\label{funew}
	\left\{
	\begin{split}
		&\rho_1 \phi_{tt} - \kappa\left[\omega(\phi_x + \psi) + \smallint_0^\infty g(s)\eta_x(s)ds\right]_x + \sigma\theta_x = 0, &(x,t)\in(0,L)\times \mathbb{R}^+,\\
		&\rho_2\psi_{tt} \!- \!b\psi_{xx} \!+\! \kappa\left[\omega(\phi_x + \psi) \!+\! \smallint_0^\infty g(s)\eta_x(s)ds\right]\! -\! \sigma\theta = 0, &(x,t)\in(0,L)\times \mathbb{R}^+,\\
		&\rho_3\theta_t + q_x + \sigma(\phi_x + \psi)_t = 0, &(x,t)\in(0,L)\times \mathbb{R}^+,\\
		&\tau q_t + \beta q + \theta_x = 0, &(x,t)\in(0,L)\times \mathbb{R}^+,\\
		&\eta_t + \eta_s - (\phi + \tilde{\psi})_t = 0,  &(x,s,t)\in(0,L)\times \mathbb{R}^+\times \mathbb{R}^+.
	\end{split}
	\right.
\end{equation}

The analytical framework employs the standard $L^2(0,L)$ space with inner product and norm
$$(u,v) = \smallint_0^L u(x)\overline{v(x)}dx,\quad \|u\| = \left(\smallint_0^L |u(x)|^2dx\right)^{\frac12}.$$
Let
\begin{align*}
	&L_*^2(0,L) := \left\{h \in L^2(0,L) : \smallint_0^L h(x)dx = 0\right\},\\
	&H_*^1(0,L) := H^1(0,L) \cap L_*^2(0,L),\\
	&L_g^2(\mathbb{R}^+,H_0^1(0,L)) := \left\{h : \smallint_0^\infty g(s)\|h_x(s)\|^2 ds < \infty\right\}.
\end{align*}
It is obvious that these are Hilbert spaces with inner product  and norms
\begin{align*}
	&(u,v)_{L_*^2(0,L)}=(u,v),&&\|u\|_{L_*^2(0,L)}=\|u\|;\\
	&(u,v)_{H_*^1(0,L)}=(u_x,v_x),&&\|u\|_{H_*^1(0,L)}=\|u_x\|;\\
	&(u,v)_{L_g^2(\mathbb{R}^+,H_0^1(0,L))}=\smallint_0^Lg(s)(u_x(s),v_x(s))ds,&&\|u\|_{L_g^2(\mathbb{R}^+,H_0^1(0,L))}=\left[\smallint_0^\infty g(s)\|u_x(s)\|^2ds\right]^{\frac12}.
\end{align*}
For notational simplicity, we adopt the abbreviations
$L^2$, $L_*^2$, $H_0^1$, $H^2$, $H_*^1$ and $L_g^2$ to respectively denote the function spaces
$L^2(0,L)$, $L_*^2(0,L)$, $H_0^1(0,L)$, $H^2(0,L)$, $H_*^1(0,L)$ and $L_g^2(\mathbb{R}^+,H_0^1(0,L))$.

We also consider the operator \(\mathbb{L}: D(\mathbb{L}) \subset L_g^2 \to L_g^2\) given by
\[\begin{cases}
	D(\mathbb{L}):=\left\{\eta:\eta\in L_g^2, \, \mathbb{L}\eta \in L_g^2 \, \hbox{ and } \, \eta(\cdot, 0)=0\right\},\\
	\mathbb{L}\eta:=-\partial_s \eta,
\end{cases}\]
which is the infinitesimal generator of the right-translation $C_0$-semigroup \(R(t):L_g^2 \to L_g^2\) given by

\[[R(t)\eta](\cdot, s) :=
\begin{cases}
	\eta(\cdot, s-t), & s > t, \\
	0, & 0 \le s \leq t.
\end{cases}\]
Consider the  Cauchy problem for the history variable
\begin{equation}
	\begin{cases}\label{eta01}
		\eta_t = \mathbb{L}\eta + (\phi + \tilde{\psi})_t, \quad t > 0, \\
		\eta^0(\cdot, s) = \eta_0(\cdot,s).
	\end{cases}
\end{equation}
By the theory of $C_0$-semigroup,	for any initial data $\eta_0 \in L_g^2$, equation \eqref{eta01} admits a unique mild solution $\eta \in C([0,\infty); L_g^2)$ with explicit representation
\begin{align}\label{eta0}
	\eta^t(\cdot, s) =&R(t)\eta_0(\cdot,s)+\smallint_0^tR(t-\tau)(\phi + \tilde{\psi})_t(\cdot,t-\tau)d\tau\notag\\=&
	\begin{cases}
		\eta_0(\cdot,s-t) + (\phi + \tilde{\psi})(\cdot,t) - (\phi_0 + \tilde{\psi}_0)(\cdot), & s > t, \\
		(\phi + \tilde{\psi})(\cdot,t) - (\phi + \tilde{\psi})(\cdot,t-s), & 0 \le s \leq t.
	\end{cases}
\end{align}

The system \eqref{funew} under the Fourier heat conduction law ($\tau=0$), yielding the simplified governing equations
\begin{equation}\label{problem1}
	\left\{
	\begin{split}
		&\rho_1 \phi_{tt} - \kappa\left[\omega(\phi_x + \psi) + \smallint_0^\infty g(s)\eta_x(s)ds\right]_x + \sigma\theta_x = 0, \quad &&(x,t)\in(0,L)\times \mathbb{R}^+,\\
		&\rho_2\psi_{tt} \!-\! b\psi_{xx} \!+\! \kappa\left[\omega(\phi_x + \psi) \!+\! \smallint_0^\infty g(s)\eta_x(s)ds\right] \!-\! \sigma\theta = 0, \quad &&(x,t)\in(0,L)\times \mathbb{R}^+,\\
		&\rho_3\theta_t - \beta \theta_{xx} + \sigma(\phi_x + \psi)_t = 0, \quad &&(x,t)\in(0,L)\times \mathbb{R}^+,\\
		&\eta_t + \eta_s - (\phi + \tilde{\psi})_t = 0, \quad &&(x,s,t)\in(0,L)\times \mathbb{R}^+\times \mathbb{R}^+
	\end{split}
	\right.
\end{equation}
subject to boundary constraints
\begin{equation}\label{bdF}
	\left\{
	\begin{split}
		&\phi(0,t) = \phi(L,t) = \psi_x(0,t) = \psi_x(L,t) = \theta_x(0,t) = \theta_x(L,t) = 0, \quad &&t \geq 0,\\
		&\eta^t(0,s) = \eta^t(L,s) = 0, \, &&t,s \ge 0,\\
		&\eta^t(x,0) = 0, \quad &&(x,t) \in (0,L)\times [0,\infty)
	\end{split}
	\right.
\end{equation}
and initial conditions
\begin{equation}\label{icF}
	\left\{
	\begin{split}
		&\phi(x,0) = \phi_0(x), \quad \phi_t(x,0) = \phi_1(x), &&x \in (0,L),\\
		&\psi(x,0) = \psi_0(x), \quad \psi_t(x,0) = \psi_1(x), &&x \in (0,L),\\
		&\theta(x,0) = \theta_0(x), &&x \in (0,L),\\
		&\eta^0(x,s) = \phi_0 + \tilde{\psi}_0-\phi(x,-s)-\tilde{\psi}(x,-s):= \eta_0(x,s), &&(x,s) \in (0,L)\times \mathbb{R}^+.
	\end{split}
	\right.
\end{equation}

To deal \eqref{problem1}-\eqref{icF} by the theory of $C_0$-semigroup, we introduce the phase space as the Hilbert space
$$
\mathcal{H}_{F} = H_0^1 \times L^2 \times H_*^1 \times L_*^2 \times L_*^2\times L_g^2,
$$
with inner product
\begin{align*}
	(z,\tilde{z})_{\mathcal{H}_{F}}=\rho_1(\Phi,\tilde{\Phi}) + \rho_2(\Psi,\tilde{\Psi})+ \kappa\omega(\phi_x + \psi,\tilde{\phi}_x + \tilde{\psi}) + b(\psi_x,\tilde{\psi}_x)+ \rho_3(\theta,\tilde{\theta})+ \kappa(\eta,\tilde{\eta})_{L_g^2}
\end{align*}
and norm
$$
\|z\|_{\mathcal{H}_F}^2 = \rho_1\|\Phi\|^2 + \rho_2\|\Psi\|^2 + \kappa\omega\|\phi_x + \psi\|^2 + b\|\psi_x\|^2 + \rho_3\|\theta\|^2 + \kappa\|\eta\|_{L_g^2}^2,
$$
where $z = (\phi,\Phi,\psi,\Psi,\theta,\eta),~\tilde{z} = (\tilde{\phi},\tilde{\Phi},\tilde{\psi},\tilde{\Psi},\tilde{\theta},\tilde{\eta})\in \mathcal{H}_F$.

Then we can reformulate problem \eqref{problem1}-\eqref{icF} as an abstract Cauchy problem in $\mathcal{H}_F$
\begin{equation}\label{cauchyfc}
	\left\{
	\begin{split}
		&z_t = \mathcal{A}_F z, \quad t > 0,\\
		&z(0) = (\phi_0,\phi_1,\psi_0,\psi_1,\theta_0,\eta_0) := z_0,
	\end{split}
	\right.
\end{equation}
where $z = (\phi,\Phi,\psi,\Psi,\theta,\eta)$ with $\Phi=\phi_t$ and $\Psi=\psi_t$, and the operator $\mathcal{A}_F: D(\mathcal{A}_F) \subset \mathcal{H}_F \to \mathcal{H}_F$ is defined by
\begin{equation}\label{9}
	\mathcal{A}_F z := \begin{pmatrix}
		\Phi \\
		\frac{\kappa}{\rho_1}\left[\omega(\phi_x + \psi) + \smallint_0^\infty g(s)\eta_x(s)ds\right]_x - \frac{\sigma}{\rho_1}\theta_x \\
		\Psi \\
		\frac{b}{\rho_2}\psi_{xx} - \frac{\kappa}{\rho_2}\left[\omega(\phi_x + \psi) + \smallint_0^\infty g(s)\eta_x(s)ds\right] + \frac{\sigma}{\rho_2}\theta \\
		\frac{\beta}{\rho_3}\theta_{xx} - \frac{\sigma}{\rho_3}(\Phi_x + \Psi) \\
		\mathbb{L}\eta + \Phi + \tilde{\Psi}
	\end{pmatrix},
\end{equation}
where the domain $D(\mathcal{A}_F)$ comprises functions $z = (\phi,\Phi,\psi,\Psi,\theta,\eta)\in \mathcal{H}_F$ satisfying
\begin{align*}
	&\phi \in H_0^1, \, \Phi \in H_0^1, \, \psi \in H^2 \cap H_*^1, \, \Psi \in H_*^1, \,\theta \in H^2 \cap H_*^1, \, \eta \in D(\mathbb{L}),
\end{align*}
with additional regularity
$$\omega\phi + \smallint_0^\infty g(s)\eta(s)ds \in H^2,\,\psi_{xx} \in  L_*^2, \,  \theta_{xx} \in L_*^2.$$

The system \eqref{funew} under the Cattaneo law  ($\tau>0$) yields the equations
\begin{equation}\label{Cattaneo1}
	\left\{\begin{split}
		&\rho_1 \phi_{tt}-\kappa\left[\omega(\phi_{x}+\psi)+\smallint_0^{\infty}g(s)\eta_{x}(s)ds\right]_{x}+\sigma\theta_{x}=0,&&(x,t)\in(0,L)\times \mathbb{R}^{+},\\
		&\rho_2\psi_{tt}-b\psi_{xx}+\kappa\left[\omega(\phi_{x}+\psi)+ \smallint_0^{\infty}g(s)\eta_{x}(s)ds\right]-\sigma\theta=0,&&(x,t)\in(0,L)\times \mathbb{R}^{+}\\
		&\rho_3\theta_{t}+q_{x}+\sigma(\phi_{x}+\psi)_{t}=0,&&(x,t)\in(0,L)\times \mathbb{R}^{+}\\
		&\tau q_{t}+\beta q+\theta_{x}=0,&&(x,t)\in(0,L)\times \mathbb{R}^{+}\\
		&\eta_{t}+\eta_{s}-(\phi+\tilde{\psi})_{t}=0,&&(x,s,t)\in(0,L)\times \mathbb{R}^{+}\times \mathbb{R}^{+}
	\end{split}
	\right.
\end{equation}
subject to boundary constraints
\begin{equation}\label{Cattaneo2}
	\left\{\begin{split}
		&\phi(0,t)=\phi(L,t)=\psi_{x}(0,t)=\psi_{x}(L,t)=0,&&t\geq0,\\
		&\theta_x(0,t)=\theta_x(L,t)=q(0,t)=q(L,t)=0,&&t\geq0,\\
		&\eta^{t}(0,s)=\eta^{t}(L,s)=0,&&t,s\geq0,\\
		&\eta^{t}(x,0)=0,&&(x,t)\in(0,L)\times [0,\infty)
	\end{split}
	\right.
\end{equation}
and initial conditions
\begin{equation}\label{Cattaneo3}
	\left\{\begin{split}
		&\phi(x,0)=\phi_0(x),\phi_{t}(x,0)=\phi_1(x),&&x\in(0,L),\\
		&\psi(x,0)=\psi_0(x),\psi_{t}(x,0)=\psi_1(x),&&x\in(0,L)\\
		&\theta(x,0)=\theta_0(x),&&x\in(0,L)\\
		&q(x,0)=q_0(x),&&x\in(0,L)\\
		&\eta^0(x,s)=\phi_0 + \tilde{\psi}_0-\phi(x,-s)-\tilde{\psi}(x,-s):= \eta_0(x,s),&&(x,s)\in(0,L)\times \mathbb{R}^{+}
	\end{split}
	\right.
\end{equation}

To deal \eqref{Cattaneo1}--\eqref{Cattaneo3} by the theory of $C_0$-semigroup, we introduce the phase space as the Hilbert space
$$\mathcal{H}_C:=H_0^1\times L^2\times H_*^1\times L_*^2\times L_*^2\times L^2\times L_{g}^2$$
equipped with inner product
\begin{align*}
	(z,\tilde{z})_{\mathcal{H}_{C}}=\rho_1(\Phi,\tilde{\Phi}) + \rho_2(\Psi,\tilde{\Psi})+ \kappa\omega(\phi_x + \psi,\tilde{\phi}_x + \tilde{\psi}) + b(\psi_x,\tilde{\psi}_x)+ \rho_3(\theta,\tilde{\theta})+\tau(q,\tilde{q})+ \kappa(\eta,\tilde{\eta})_{L_g^2}
\end{align*}
and norm
$$\|z\|_{\mathcal{H}_C}^2=\rho_1\|\Phi\| ^2+\rho_2\|\Psi\| ^2+\kappa\omega\|\phi_{x}+\psi\|^2+b\|\psi_{x}\|^2+\rho_3\|\theta\|^2+\tau\|q\|^2+\kappa\|\eta\|_{ L_{g}^2}^2 ,$$
where $z=(\phi,\Phi,\psi,\Psi,\theta,q,\eta),~\tilde{z}=(\tilde{\phi},\tilde{\Phi},\tilde{\psi},\tilde{\Psi},\tilde{\theta},\tilde{q},\tilde{\eta})\in\mathcal{H}_C$.

Then we can reformulate problem \eqref{Cattaneo1}--\eqref{Cattaneo3} as an abstract Cauchy problem in $\mathcal{H}_C$
\begin{equation}\label{cl}
	\left\{\begin{split}
		&z_t=\mathcal{A}_Cz,\quad t>0\\
		&z(0)=(\phi_0,\phi_1,\psi_0,\psi_1,\theta_0,q_0,\eta_0):=z_0
	\end{split}
	\right.
\end{equation}
where $z=(\phi,\Phi,\psi,\Psi,\theta,q,\eta)$ with $\Phi=\phi_t$ and $\Psi=\psi_t$, and the operator $\mathcal{A}_C:D(\mathcal{A}_C)\subset\mathcal{H}_C\rightarrow\mathcal{H}_C$ is defined by
\begin{equation}\label{eq}
	\mathcal{A}_Cz:=\begin{pmatrix}
		\Phi \\
		\frac{\kappa}{\rho_1}[\omega(\phi_{x}+\psi)+\smallint_0^{\infty}g(s)\eta_{x}(s)ds]_{x}-\frac{\sigma}{\rho_1}\theta_{x}\\
		\Psi \\
		\frac{b}{\rho_2}\psi_{xx}-\frac{\kappa}{\rho_2}[\omega(\phi_{x}+\psi)+\smallint_0^{\infty}g(s)\eta_{x}(s)ds]+\frac{\sigma}{\rho_2}\theta\\
		-\frac{1}{\rho_3}q_x-\frac{\sigma}{\rho_3}(\Phi_x+\Psi)\\
		-\frac{\beta}{\tau}q-\frac{1}{\tau}\theta_x\\
		\mathbb{L}\eta +\Phi+\tilde{\Psi}
	\end{pmatrix},
\end{equation}
where the domain $D(\mathcal{A}_C)$ comprises functions $z=(\phi,\Phi,\psi,\Psi,\theta,q,\eta)\in\mathcal{H}_C$ satisfying
\[
\phi\in H_0^1,\,\Phi\in H_0^1,\,\psi\in H^2\cap H_*^1,\,\Psi\in H_*^1,\,\theta\in H_*^1,\,q\in H_0^1,\,\eta\in D(\mathbb{L}),
\]
with additional regularity
$$\omega\phi+\smallint_0^{\infty}g(s)\eta(s)ds\in H^2(0,L),~~\psi_{xx}\in L_*^2.$$

Firstly, we present the global well-posedness result with the following result on $g$:
\begin{assumption}\label{g1}
	The kernel \( g : [0,\infty)\to (0,\infty) \) is absolutely continuous, nonincreasing, and summable with total mass
	\[
	\ell := \smallint_0^\infty g(s) \, ds \in (0, 1).
	\]
\end{assumption}

\begin{theorem}\label{wellposed}
	Let $g$ be a kernel satisfy Assumption \ref{g1}.
	\begin{enumerate}
		\item The operator $\mathcal{A}_F: D(\mathcal{A}_F) \subset \mathcal{H}_F \to \mathcal{H}_F$ is an infinite-generator of a contraction $C_0$-semigroup $\left\{e^{\mathcal{A}_Ft}\right\}_{t\ge0}$ on $\mathcal{H}_F$, which we called {\bf \textit{Fourier-semigroup}}. Then for every initial value $z_0 \in \mathcal{H}_F$ , problem \eqref{cauchyfc}  admits a unique mild solution $z \in C(0,\infty;\mathcal{H}_F)$  given explicitly by
		\[
		z(t) =e^{\mathcal{A}_Ft}z_0, \quad t \geq 0.
		\]
		Moreover, when $z_0 \in D(\mathcal{A}_F)$, this solution becomes classical with improved regularity
		\[
		z \in C^1(0,\infty;\mathcal{H}_F) \cap C(0,\infty;D(\mathcal{A}_F)).
		\]
		\item Similarly, the operator $\mathcal{A}_C: D(\mathcal{A}_C) \subset \mathcal{H}_C \to \mathcal{H}_C$ is an infinite-generator of a contraction $C_0$-semigroup $\left\{e^{\mathcal{A}_Ct}\right\}_{t\ge0}$ on $\mathcal{H}_C$, which we called {\bf \textit{Cattaneo-semigroup}}. Then for every initial value $z_0 \in \mathcal{H}_C$ , problem \eqref{cl}  admits a unique mild solution $z \in C(0,\infty;\mathcal{H}_C)$  given explicitly by
		\[
		z(t) =e^{\mathcal{A}_Ct}z_0, \quad t \geq 0.
		\]
		Moreover, when $z_0 \in D(\mathcal{A}_C)$, this solution becomes classical with improved regularity
		\[
		z \in C^1(0,\infty;\mathcal{H}_C) \cap C(0,\infty;D(\mathcal{A}_C)).
		\]
	\end{enumerate}
\end{theorem}
To state the stability results on the Fourier and Cattaneo semigroups we let
\begin{equation}\label{chi0}
	\chi_0 := \frac{\kappa}{\rho_1} - \frac{b}{\rho_2}
\end{equation}
and
\begin{equation}\label{chi1}
	\chi_1:=\chi_0-\frac{\tau b}{\rho_1\rho_2}\left(\rho_1\rho_3\chi_0+\sigma^2\right).
\end{equation}
\begin{theorem}\label{thmFourier}\textbf{(Stability of the Fourier-semigroup)} Let Assumption \ref{g1} hold.
	\begin{enumerate}
		\item The Fourier-semigroup $\left\{e^{\mathcal{A}_Ft}\right\}_{t\ge0}$ is  \textbf{exponentially stable}, i.e., there exists constants $M\ge1$ and $\gamma>0$ such that
		$$\left\|e^{\mathcal{A}_Ft}\right\|_{\mathcal{L}(\mathcal{H}_F)}\le Me^{-\gamma t},\quad t\ge0,$$
		\textbf{if and only if} both the $\delta$-condition \eqref{delta c} and $\chi_0 = 0$  hold.
		\item If the $\delta$-condition \eqref{delta c} holds, the Fourier-semigroup $\left\{e^{\mathcal{A}_Ft}\right\}_{t\ge0}$ is \textbf{polynomially stable} of order $1/2$, whatever $\chi_0 = 0$ or not,  i.e., there exists a constant $C>0$ such that
		\[
		\left\|e^{\mathcal{A}_Ft}z_0\right\|_{\mathcal{H}_F} \leq C t^{-\frac12} \|z_0\|_{D(\mathcal{A}_F)}, \quad t \geq 1.
		\]
	\end{enumerate}
\end{theorem}

\begin{theorem}\label{thmCattaneo}\textbf{(Stability of the Cattaneo-semigroup)} Let Assumption \ref{g1} hold.
	\begin{enumerate}
		\item The Cattaneo-semigroup $\left\{e^{\mathcal{A}_Ct}\right\}_{t\ge0}$ is  \textbf{exponentially stable}, i.e., there exists constants $M\ge1$ and $\gamma>0$ such that
		$$\left\|e^{\mathcal{A}_Ct}\right\|_{\mathcal{L}(\mathcal{H}_C)}\le Me^{-\gamma t},\quad t\ge0,$$
		\textbf{if and only if} both the $\delta$-condition \eqref{delta c} and $\chi_1 = 0$  hold.
		\item If the $\delta$-condition \eqref{delta c} holds, the Cattaneo-semigroup $\left\{e^{\mathcal{A}_Ct}\right\}_{t\ge0}$ is \textbf{polynomially stable} of order $1/2$, whatever $\chi_1 = 0$ or not, i.e., there exists a constant $C>0$ such that
		\[
		\left\|e^{\mathcal{A}_Ct}z_0\right\|_{\mathcal{H}_C} \leq C t^{-\frac12} \|z_0\|_{D(\mathcal{A}_C)}, \quad t \geq 1.
		\]
	\end{enumerate}
\end{theorem}
\begin{remark}\label{rmkCompare}
	The stability properties of the Fourier- and Cattaneo-semigroups reveal both similarities and striking differences:
	
	\begin{itemize}
		\item \textbf{Polynomial stability:} Both semigroups exhibit polynomial decay of order $1/2$ under the $\delta$-condition \eqref{delta c}, regardless of the values of $\chi_0$ and $\chi_1$. This suggests a common underlying mechanism for polynomial stabilization in these systems.
		
		\item \textbf{Exponential stability dichotomy:} The Fourier-semigroup achieves exponential stability (\textit{if and only if} both $\delta$-condition \eqref{delta c} and $\chi_0 = 0$ hold), while the Cattaneo-semigroup requires the modified condition $\chi_1 = 0$ for exponential decay. This highlights a fundamental structural difference between the two models: the Fourier law permits energy dissipation sufficient for exponential decay under specific conditions, whereas the Cattaneo model introduces  thermal relaxation time effects that alter the stability threshold.
		
		\item \textbf{Parameter relationships:} The critical parameter $\chi_1$ for Cattaneo's model contains a  thermal relaxation time $\tau$ correction to $\chi_0$:
		$$\chi_1 = \chi_0 - \frac{\tau b}{\rho_1\rho_2}(\rho_1\rho_3\chi_0+\sigma^2).$$ As $\tau \to 0$, we observe:
		\begin{align*}
			\chi_1 \to \chi_0,\quad \mathcal{A}_C \to \mathcal{A}_F \quad \text{(formally)}
		\end{align*}
		recovering the Fourier-system as a limit of the Cattaneo model.
		
		\item \textbf{Physical interpretation:} The $\tau$-dependent term in $\chi_1$ represents thermal relaxation effects. Its vanishing in the $\tau\to0$ limit demonstrates how Cattaneo's law bridges between hyperbolic (finite $\tau$) and parabolic ($\tau=0$) heat propagation models.
	\end{itemize}
	
	These results demonstrate that while both thermodynamical frameworks share some stabilizing features, their long-term dynamical behaviors are fundamentally distinct due to differences in their constitutive laws. The Cattaneo-semigroup's stability characteristics smoothly connect to the Fourier case through the relaxation time parameter $\tau$.
\end{remark}
\section{Well-Posedness}\label{sec2.1}
The well-posedness of the abstract Cauchy problems \eqref{cauchyfc} and \eqref{cl}, as stated in Theorem~\ref{wellposed}, is established through an application of the Lumer-Phillips theorem \cite{MR710486}. Since the proof for Cauchy problem \eqref{cl} follows similarly to that of \eqref{cauchyfc}, we present here only the detailed proof for \eqref{cauchyfc}.

The approach necessitates verification of two fundamental properties for the operator $\mathcal{A}_F$: dissipativity and range surjectivity of $I - \mathcal{A}_F$.

\noindent\textbf{Dissipativity Analysis.}	For any $z = (\phi, \Phi, \psi, \Psi, \theta, \eta) \in D(\mathcal{A}_F)$, we compute
\begin{align}\label{hs1}
	\operatorname{Re}( \mathcal{A}_F z, z )_{\mathcal{H}_F} = & \operatorname{Re}\Bigg\{\rho_1\left(\Phi,\frac{\kappa}{\rho_1}\left[\omega(\phi_x + \psi) + \smallint_0^\infty g(s)\eta_x(s)ds\right]_x - \frac{\sigma}{\rho_1}\theta_x\right)\notag\\
	&+\rho_2\left(\Psi,\frac{b}{\rho_2}\psi_{xx} - \frac{\kappa}{\rho_2}\left[\omega(\phi_x + \psi) + \smallint_0^\infty g(s)\eta_x(s)ds\right] + \frac{\sigma}{\rho_2}\theta\right)+\kappa\omega\left(\phi_x+\psi,\Phi_x+\Psi\right)\notag\\
	&+b\left(\psi_x,\Psi_x\right)+\rho_3\left(\theta,\frac{\beta}{\rho_3}\theta_{xx} - \frac{\sigma}{\rho_3}(\Phi_x + \Psi) \right)+\kappa\smallint_0^{\infty}g(s)\left(\eta_x(s),
	-\eta_{sx} + \Phi_x + \Psi\right)\Bigg\}\notag\\
	=&\operatorname{Re}\Bigg\{-\kappa\omega(\Phi_x, \phi_x + \psi) - \kappa \smallint_{0}^{\infty} g(s)\left(\Phi_x, \eta_x(s)\right) \, ds + \sigma (\Phi_x, \theta) - b(\Psi_x, \psi_x) - \kappa\omega(\Psi, \phi_x + \psi) \notag\\
	& - \kappa \smallint_{0}^{\infty} g(s)(\Psi, \eta_x(s)) \, ds + \sigma (\Psi, \theta) + \kappa\omega(\phi_x + \psi, \Phi_x + \Psi) + b(\psi_x, \Psi_x) \notag\\
	& - \beta \|\theta_x\|^2 - \sigma (\theta, \Phi_x + \Psi) - \kappa \smallint_{0}^{\infty} g(s)(\eta_x(s), \eta_{sx}(s))\, ds + \kappa \smallint_{0}^{\infty} g(s)(\eta_x(s), \Phi_x + \Psi) \, ds\Bigg\}\notag\\		=&-\beta\|\theta_x\|^2 -\kappa \operatorname{Re}\left\{( \eta, \eta_s )_{L_g^2}\right\}.
\end{align}
Since $\eta\in L_g^2$, we have
\(\smallint_0^\infty g(s)\|\eta_x(s)\|^2ds<\infty\), which, together with $g(\cdot)>0$, implies there exists a positive increasing sequence $\{s_n\}_{n=1}^\infty$ such that
$$\lim_{n\to\infty}s_n=\infty\hbox{ and }\lim_{n\to\infty}\left(g(s_n)\|\eta_x(s_n)\|^2\right)=0.$$
Then, by $\|\eta_x(0)\|=0$ and $0<g(0)<\infty$,
\begin{align}
	-\kappa ( \eta, \eta_s )_{L_g^2} =&-\kappa\smallint_0^\infty g(s)(\eta_x, \eta_{sx}) ds=-\frac{\kappa}{2}\lim_{n\to\infty}\smallint_0^{s_n}g(s)\frac{d}{ds}\|\eta_x(s)\|^2ds\notag\\
	=&-\frac{\kappa}{2}\lim_{n\to\infty}\left(g(s_n)\|\eta_x(s_n)\|^2\right)+\frac{\kappa}{2}\lim_{n\to\infty}\smallint_0^{s_n}g'(s)\|\eta_x(s)\|^2ds\notag\\
	=&\frac{\kappa}{2}\smallint_0^{\infty}g'(s)\|\eta_x(s)\|^2ds.\label{jxsn}
\end{align}
Thus, we derive
\begin{equation}\label{4}
	\operatorname{Re}(\mathcal{A}_F z, z )_{\mathcal{H}_F} = -\beta\|\theta_x\|^2 + \frac\kappa2\smallint_0^\infty g'(s)\|\eta_x(s)\|^2 ds \leq 0
\end{equation}
since $g(\cdot)$ is an absolutely continuous, nonincreasing function.

\noindent\textbf{Surjective analysis.} Given $f = (f^1, f^2, f^3, f^4, f^5, f^6) \in \mathcal{H}_F$, we seek $z = (\phi, \Phi, \psi, \Psi, \theta, \eta) \in D(\mathcal{A}_F)$ satisfying $z-\mathcal{A}_Fz=f$, i.e.,	
\begin{equation}\label{system_A}
	\begin{cases}
		&\text{(a)}:\quad	\phi - \Phi = f^1 \in H_0^1, \\
		&\text{(b)}:\quad	\rho_1\Phi - \kappa\left[\omega(\phi_x + \psi) + \smallint_0^\infty g(s)\eta_x(s)ds\right]_x + \sigma\theta_x = \rho_1f^2\in L^2, \\
		&\text{(c)}:\quad		\psi - \Psi = f^3\in H_*^1, \\
		&\text{(d)}:\quad	\rho_2\Psi - b\psi_{xx} + \kappa\left[\omega(\phi_x + \psi) + \smallint_0^\infty g(s)\eta_x(s)ds\right] - \sigma\theta = \rho_2f^4\in L_*^2, \\
		&\text{(e)}:\quad	\rho_3\theta - \beta\theta_{xx} + \sigma(\Phi_x + \Psi) = \rho_3f^5\in L_*^2,\\
		&\text{(f)}:	\quad\eta + \eta_s - (\Phi + \tilde{\Psi}) = f^6\in L_g^2.
	\end{cases}
\end{equation}

Using equations (a), (c), and (f) from \eqref{system_A}, we get
\begin{equation}\label{transf}
	\begin{cases}
		\Phi = \phi - f^1, \\
		\Psi = \psi - f^3, \\
		\eta(s) = (1 - e^{-s})(\Phi + \tilde{\Psi}) + \smallint_0^s f^6(\tau)e^{-(s-\tau)}d\tau.
	\end{cases}
\end{equation}
Substituting the above equalities into the equations (b), (d), (e) of \eqref{system_A} yields
\begin{equation}\label{qju1}
	\begin{cases}
		\rho_1\phi - \kappa(\omega+\hat\kappa)(\phi_x + \psi)_x  + \sigma\theta_x = R_1, \\
		\rho_2\psi - b\psi_{xx} + \kappa(\omega+\hat\kappa)(\phi_x + \psi)- \sigma\theta = R_2, \\
		\rho_3\theta - \beta\theta_{xx} + \sigma(\phi_x + \psi) = R_3,
\end{cases}\end{equation}
where $\hat\kappa:=\smallint_0^\infty g(s)(1 - e^{-s})ds>0$, and
\[\begin{cases} R_1:=\rho_1f^1-\kappa\hat\kappa(f^1_{xx}+f^3_x)+\kappa\smallint_0^\infty g(s)e^{-s}\left(\smallint_0^se^{\tau}f^6_{xx}(\tau)d\tau\right)ds+\rho_1f^2\in H^{-1},\\ R_2:=\rho_2f^3+\kappa\hat\kappa(f_x^1+f^3)-\kappa\smallint_0^{\infty}g(s)e^{-s}\left(\smallint_0^se^{\tau}f_x^6(\tau)d\tau\right)ds+\rho_2f^4\in L^2,\\
	R_3:=\sigma(f_x^1+f^3)+\rho_3f^5\in L^2.
\end{cases}\]
The weak formulation of system \eqref{qju1} is given by
\begin{equation}\label{qju2}
	\begin{cases}
		\rho_1(\phi,\hat{\phi})+\kappa(\omega+\hat\kappa)(\phi_x + \psi,\hat{\phi}_x)+ \sigma(\theta_x,\hat{\phi}) = \langle R_1,\hat{\phi}\rangle,\\
		\rho_2(\psi,\hat{\psi}) + b(\psi_{x},\hat{\psi}_x) + \kappa(\omega+\hat\kappa)(\phi_x + \psi,\hat{\psi})- \sigma(\theta,\hat{\psi}) = (R_2,\hat{\psi}),\\
		\rho_3(\theta,\hat{\theta}) +\beta(\theta_{x}, \hat{\theta}_x)+ \sigma(\phi_x + \psi,\hat{\theta}) = (R_3,\hat{\theta})
	\end{cases}
\end{equation}
for all
\[(\hat{\phi},\hat{\psi},\hat{\theta})\in\hat{\mathcal{H}}_F:=H_0^1\times H_*^1\times H_*^1, \]
with inner product
\begin{align*}
	(z,\hat{z})_{\hat{\mathcal{H}}_F}=\rho_1(\phi,\hat\phi)+\rho_2(\psi,\hat\psi)+\rho_3(\theta,\hat\theta)+\kappa(\omega+\hat\kappa)(\phi_x+\psi,\hat\phi_x+\hat\psi)+b(\psi_x,\hat\psi_x)+\beta(\theta_x,\hat\theta_x)
\end{align*}
and norm
$$
\|z\|_{\hat{\mathcal{H}}_F}^2 =\rho_1\|\phi\|^2+\rho_2\|\psi\|^2 +\rho_3\|\theta\|^2+\kappa(\omega+\hat\kappa)\|\phi_x + \psi\|^2++ b\|\psi_{x}\|^2 +\beta\|\theta_{x}\|^2,
$$
where $z = (\phi,\psi,\theta),  \hat{z} = (\hat\phi,\hat\psi,\hat\theta)\in \hat{\mathcal{H}}_F$ and $\langle\cdot,\cdot\rangle$ denotes the duality product between $H^{-1}$ and $H_0^1$.

Summing the equations in \eqref{qju2}, we obtain
\begin{equation*}
	\Lambda((\phi, \psi, \theta), (\hat{\phi}, \hat{\psi}, \hat{\theta}))=F(\hat{\phi}, \hat{\psi}, \hat{\theta}),
\end{equation*}
where $\Lambda$ is a sesquilinear form on $\hat{\mathcal{H}}_F\times \hat{\mathcal{H}}_F$ defined by
\begin{align}\label{qju4}
	\Lambda((\phi, \psi, \theta), (\hat{\phi}, \hat{\psi}, \hat{\theta}))=&\rho_1(\phi,\hat{\phi})+\rho_2(\psi,\hat{\psi}) +\rho_3(\theta,\hat{\theta}) +\kappa(\omega+\hat\kappa)(\phi_x + \psi,\hat{\phi}_x+\hat\psi)\notag\\
	&- \sigma(\theta, \hat{\phi}_x + \hat{\psi}) + b(\psi_{x},\hat{\psi}_x) +\beta(\theta_{x}, \hat{\theta}_x)+ \sigma(\phi_x + \psi,\hat{\theta})
\end{align}
and $F$ is a linear form on $\hat{\mathcal{H}}_F$ defined by
\begin{align*}
	F(({\phi}, {\psi}, {\theta}))=& \langle R_1,{\phi}\rangle+(R_2,{\psi})+(R_3,{\theta})\notag\\
	=&\rho_1(f^1+f^2,\phi)+\kappa\hat\kappa(f^1_x,\phi_x)-\kappa\hat\kappa(f^3_x,\phi)-\kappa\smallint_0^\infty g(s)e^{-s}\left(\smallint_0^s e^\tau(f^6_x(\tau),\phi_x)d\tau\right) ds\notag\\
	&+((\rho_2+\kappa\hat\kappa)f^3+\rho_2f^4,\psi)+\kappa\hat\kappa(f^1_x,\psi)-\kappa\smallint_0^\infty g(s)e^{-s}\left(\smallint_0^s e^\tau(f^6_x(\tau),\psi)d\tau\right) ds\notag\\
	&+(\sigma f^3+\rho_3f^5,\theta)+\sigma(f^1_x,\theta).
\end{align*}
Using H\"older's inequality and \eqref{qju4}, it follows that, for some $c>0$,
\begin{align}\label{qju6}
	\Lambda((\phi, \psi, \theta), (\hat{\phi}, \hat{\psi}, \hat{\theta}))\le&\rho_1\|\phi\|\|\hat{\phi}\|+\rho_2\|\psi\|\|\hat{\psi}\| +\rho_3\|\theta,\hat{\theta}\| +\kappa(\omega+\hat\kappa)(\|\phi_x \|+ \|\psi\|)(\|\hat{\phi}_x\|+\|\hat\psi\|)\notag\\
	&+ \sigma\|\theta\|(\|\hat{\phi}_x\| +\| \hat{\psi}\|) + b\|\psi_{x}\|\|\hat{\psi}_x)\| +\beta\|\theta_{x}\| \|\hat{\theta}_x\|+ \sigma(\|\phi_x\| + \|\psi\|)\|\hat{\theta}\|\notag\\
	\le&c\|(\phi, \psi, \theta)\|_{\hat{\mathcal{H}}_F}\| (\hat{\phi}, \hat{\psi}, \hat{\theta})\|_{\hat{\mathcal{H}}_F}.
\end{align}
Moreover, we get
\begin{align*}
	\operatorname{Re}\Lambda((\phi, \psi, \theta), (\phi, \psi, \theta))=&\rho_1\|\phi\|^2+\rho_2\|\psi\|^2 +\rho_3\|\theta\|^2+\kappa(\omega+\hat\kappa)\|\phi_x + \psi\|^2\notag\\
	&+\operatorname{Re}\left\{\sigma(\phi_x + \psi,{\theta})- \sigma(\theta, {\phi}_x + {\psi})\right\} + b\|\psi_{x}\|^2 +\beta\|\theta_{x}\|^2\notag\\
	=&(\phi, \psi, \theta)_{\hat{\mathcal{H}}_F}^2.
\end{align*}
Hence, $\Lambda$ is bounded and coercive. Also, we have, for some $\hat c>0$
\begin{align*}
	F(({\phi}, {\psi}, {\theta}))\le&\rho_1\|f^1+f^2\|\|\phi\|+\kappa\hat\kappa\|f^1_x\|\|\phi_x\|+\kappa\hat\kappa\|f^3_x\|\|\phi\|+\kappa\left(\smallint_0^\infty g(s)ds\right)^{\frac12}\|f^6\|_{L^2_g}\|\phi_x\|\notag\\
	&+\left((\rho_2+\kappa\hat\kappa)\|f^3\|+\rho_2\|f^4\|\right)\|\psi\|+\kappa\hat\kappa\|f^1_x\|\|\psi\|+\kappa\left(\smallint_0^\infty g(s)ds\right)^{\frac12}\|f^6\|_{L^2_g}\|\psi\|\notag\\
	&+(\sigma\|f^3\|+\rho_3\|f^5\|)\|\theta\|+\sigma\|f^1_x\|\|\theta\|\notag\\
	\le&\hat c\left(\|f^1\|+\|f^1_x\|+\|f^2\|+\|f^3\|+\|f^3_x\|+\|f^4\|+\|f^5\|+\|f^6\|_{L^2_g}\right)(\phi, \psi, \theta)_{\hat{\mathcal{H}}_F},
\end{align*}
which implies that $F$ is bounded, where we have used the following estimates: since by Assumption \ref{g1}, $g(\cdot)$ is nonincreasing, it follows from Cauchy-Schwartz' inequality and changing the order of integration that
\begin{align*}
	\bigg|\smallint_0^\infty& g(s)e^{-s}\left(\smallint_0^s e^\tau(f^6_x(\tau),\phi_x)d\tau\right) ds\bigg|\le\smallint_0^{\infty} g(s)\left(\smallint_0^s e^{\tau-s}\|f^6_x(\tau)\|d\tau\right)ds\|\phi_x\|\\
	=&\smallint_0^{{\infty}}\smallint_\tau^{\infty}g(s)e^{\tau-s}\|f_x^6(\tau)\|dsd\tau\|\phi_x\|\le\smallint_0^{\infty}e^\tau g(\tau)\|f_x^6(\tau)\|\left(\smallint_\tau^{\infty}e^{-s}ds\right)d\tau\|\phi_x\|\\
	=&\smallint_0^{\infty} g(\tau)\|f_x^6(\tau)\|d\tau\|\phi_x\|\le\left(\smallint_0^{\infty} g(\tau)d\tau\right)^{\frac12}\left(\smallint_0^{\infty} g(\tau)\|f_x^6(\tau)\|^2d\tau\right)^{\frac12}\|\phi_x\|\\
	=&\left(\smallint_0^{\infty} g(\tau)d\tau\right)^{\frac12}\|f^6\|_{L^2_g}\|\phi_x\|.
\end{align*}
Similarly,
$$\left|\smallint_0^\infty g(s)e^{-s}\left(\smallint_0^s e^\tau(f^6_x(\tau),\psi)d\tau\right) ds\right|\le\left(\smallint_0^{\infty} g(\tau)d\tau\right)^{\frac12}\|f^6\|_{L^2_g}\|\psi\|.$$

By the Lax-Milgram theorem, there exists unique $(\phi, \psi, \theta) \in \hat{\mathcal{H}}_F=H_0^1 \times H_*^1 \times H_*^1$ satisfying \eqref{qju1} in the sense of \eqref{qju2}. Since $R_2,R_3\in L^2$, by regularity theory for elliptic equations, we get $\psi,\theta\in H^2\cap H_*^1$. Since $f^1\in H_0^1$, we get $f^1$ is absolutely continuous on $[0,L]$ and $f_1(0)=f_1(L)=0$, then we get
\begin{equation}\label{jfw0}
	0=f(L)-f(0)=\smallint_0^L f^1_xds.
\end{equation}
Similarly since $f^6\in L^2_g$, we have $$\smallint_0^L\left(\smallint_0^{\infty}g(s)e^{-s}\left(\smallint_0^se^{\tau}f_x^6(\tau)d\tau\right)dsdx\right ) dx=\smallint_0^{\infty}g(s)e^{-s}\left(\smallint_0^se^{\tau}\left(\smallint_0^Lf_x^6(\tau)dx\right)d\tau\right)ds=0.$$
Moreover, note $f^3\in H_*^1$ and $f^4\in L_*^2$, we get $\smallint_0^L R_2 dx=0$. So we get from \eqref{qju1}$_2$ that
$$\smallint_0^L\psi_{xx}dx=(\rho_2+\kappa(\omega+\hat\kappa))\frac1b\smallint_0^L\psi dx+\kappa(\omega+\hat\kappa)\frac1b\smallint_0^L\phi_x dx-\frac\sigma b\smallint_0^L\theta dx=0,$$
where we have used $\psi,\theta\in H_*^1$ and $\phi\in H_0^1$ with same argument as in \eqref{jfw0}. Similarly, we get from \eqref{qju1}$_3$ $\smallint_0^L\theta_{xx} dx=0$. So $\psi_{xx},\theta_{xx}\in L_*^2$.

In view of \eqref{transf} and $f^1\in H_0^1$ and $f^3\in H_*^1$, we get $\Phi:= \phi - f^1\in H_0^1$ and $\Psi := \psi - f^3\in H_*^1$. Since $f^6\in L^2_g$, we get $$\eta(s):= (1 - e^{-s})(\Phi + \tilde{\Psi}) + \smallint_0^s f^6(\tau)e^{-(s-\tau)}d\tau\in L_g^2$$ and $\eta(0)=0.$ Moreover, $$\eta_s=e^{-s}(\Phi + \tilde{\Psi}) +f^6(s)-\smallint_0^s f^6(\tau)e^{-(s-\tau)}d\tau\in L_g^2,$$ which implies $\eta\in D(\mathbb{L})$.

Considering \eqref{system_A}-(b) and $f^2\in L^2(0,L)$, we have
\begin{equation*}
	\kappa\left[\omega\phi + \smallint_0^\infty g(s)\eta(s)ds\right]_{xx} = -\rho_1f^2+\rho_1\Phi+ \sigma\theta_x-\kappa\omega\psi \in L^2(0,L).
\end{equation*}
That is
\begin{equation*}
	\kappa\left[\omega\phi + \smallint_0^\infty g(s)\eta(s)ds\right]\in H^2(0,L).
\end{equation*}

Therefore,  $(\phi, \Phi, \psi, \Psi, \theta, \eta) \in D(\mathcal{A}_F)$, and then the mapping $I - \mathcal{A}_F:D(\mathcal{A}_F)\mapsto \mathcal{H}_F$ is surjective. By Lumer-Phillips theorem, $\mathcal{A}_F$ generates a $C_0$-contraction semigroup $\left\{ e^{\mathcal{A}_Ft}\right\}_{t\ge 0}$ on $\mathcal{H}_F$, completing the proof of Theorem \ref{wellposed}.

\section{Stability Analysis}\label{sec2.2}
In this section, we will  give the proofs of Theorems \ref{thmFourier} and \ref{thmCattaneo}. The following lemmas will be used throughout our proofs.
\begin{lemma}\cite{MR3637940,MR1681462}\label{hatg}
	Let $f\in L^1(0,\infty)$ and  denote by $\hat{f}$ the half-Fourier transform of $f$, i.e.,
	$$\hat{f}(\lambda)=\smallint_0^{\infty}e^{-i\lambda s}f(s)ds,$$
	Then,
	$$\lim_{|\lambda|\rightarrow\infty}\hat{f}(\lambda)=0,$$
	i.e.,
	\begin{equation}\label{lim00}
		\lim_{|\lambda|\rightarrow\infty}\smallint_0^\infty f(s)\cos(\lambda s)ds=\lim_{|\lambda|\rightarrow\infty}\smallint_0^\infty f(s)\sin(\lambda s)ds=0.
	\end{equation}
	Moreover, if $f$ is absolutely continuous on $(0,\infty)$, $f'\in L^1(0,\infty)$ and $\lim_{x\to0} f(x)=f_0\in\mathbb{R}$, then
	$$\lim_{|\lambda|\rightarrow\infty}\lambda\hat{f}(\lambda)=-if_0.$$
\end{lemma}
\begin{lemma}\cite{MR4704637}\label{Ig}
	Let $g$ satisfy Assumption \ref{g1}. Then there exists a constant $\alpha>0$ such that the set
	$$ J=\{s\in\mathbb{R}^+:\alpha g'(s)+g(s)<0\}$$
	has positive Lebesgue measure and
	$$I_g(\lambda):=\smallint_0^{\infty}\tilde{g}(s)(1-\cos(\lambda s))ds>0$$
	where
	\[\tilde{g}(s):=g(s)\chi_J(s)=\begin{cases}
		g(s)\quad&\hbox{ if }s\in J,\\
		0&\hbox{ if }s\notin J.
	\end{cases}
	\]		
\end{lemma}
\begin{lemma}\cite{MR933321}\label{Thm:SpectralApprox}
	Let $\mathcal{A}$ be an infinitesimal generator of a $C_0$-semigroup on a reflexive Banach space $X$. Then
	$$
	\sigma(\mathcal{A}) \cap i\mathbb{R} = \sigma_{ap}(\mathcal{A}) \cap i\mathbb{R},
	$$
	where $\sigma(\mathcal{A})$ and $\sigma_{ap}(\mathcal{A})$ denote the spectrum set and approximate spectrum set of operator $\mathcal{A}$, respectively.
\end{lemma}
\begin{lemma}\cite{MR4396027} \label{lemggj}
	Let $g$ satisfy Assumption \ref{g1} and  the ${\delta}$-condition \eqref{delta c}. Then there exists a positive constant $C$ such that for every $\eta\in L_g^2$,
	\begin{equation}\label{eta_x}
		\smallint_0^{\infty}g(s)\left(\smallint_0^s\|\eta_x(\tau)\|d\tau\right)^2ds\leq C\|\eta\|_{L_g^2}^2.
	\end{equation}
\end{lemma}
\subsection{Boundedness of the Resolvent Operator }\label{sec2.2.1}
The main result of this part is the following theorem.
\begin{theorem}\label{thmboundresolvent}
	Let $g$ satisfy the Assumption \ref{g1} and the $\delta$-condition \eqref{delta c}.
	\begin{enumerate}
		\item If $i\mathbb{R}\subset\rho(\mathcal{A}_F)$, then there exists a positive constant $C$, independent of $\lambda$, such that
		\begin{align*}
			\left\|(i\lambda I-\mathcal{A}_F)^{-1}\right\|_{\mathcal{L}(\mathcal{H}_F)}\le& C\left(\frac{1}{I_g(\lambda)}+1\right)\left(\frac{1}{I_g(\lambda)}+\frac{1}{|\lambda|}+1\right) \\&+C|\chi_0|(1+|\lambda|^2)\left(\frac{1}{I_g(\lambda)}+1\right)\left(\frac{1}{I_g(\lambda)}+\frac{1}{|\lambda|}+1\right).
		\end{align*}
		\item If $i\mathbb{R}\subset\rho(\mathcal{A}_C)$, then there exists a positive constant $C$, independent of $\lambda$, such that
		\begin{align*}
			\left\|(i\lambda I-\mathcal{A}_C)^{-1}\right\|_{\mathcal{L}(\mathcal{H}_C)}\le& C\left(\frac{1}{I_g(\lambda)}+1\right)\left(\frac{1}{I_g(\lambda)}+\frac{1}{|\lambda|}+1\right) \\&+C|\chi_0|(1+|\lambda|^2)\left(\frac{1}{I_g(\lambda)}+1\right)\left(\frac{1}{I_g(\lambda)}+\frac{1}{|\lambda|}+1\right)+C|\lambda|^2.
		\end{align*}
		\item If $i\mathbb{R}\subset\rho(\mathcal{A}_C)$ and $\rho_2-b\rho_3\tau\neq0$, then there exists a positive constant $C$, independent of $\lambda$, such that
		\begin{align*}
			\left\|(i\lambda I-\mathcal{A}_C)^{-1}\right\|_{\mathcal{L}(\mathcal{H}_C)}\le& C\left(\frac{1}{I_g(\lambda)}+1\right)\left(\frac{1}{I_g(\lambda)}+\frac{1}{|\lambda|}+1\right) \\&+C|\chi_1||\lambda|^2\left(\frac{1}{I_g(\lambda)}+1\right)\left(\frac{1}{I_g(\lambda)}+\frac{1}{|\lambda|}+1\right).
		\end{align*}
	\end{enumerate}
	Here, $\chi_0$ is the constant defined in \eqref{chi0},  $\chi_1$ is defined in \eqref{chi1}, and $I_g(\lambda)$ is defined in Lemma \ref{Ig}.
\end{theorem}
\begin{proof}[Proof of the conclusion 1 of Theorem \ref{thmboundresolvent}]
	Let $f=(f^1,f^2,f^3,f^4,f^5,f^6)\in \mathcal{H}_F$ be given. Since we assumed $i\mathbb{R}\subset\rho(\mathcal{A}_F)$, we can let $z = (\phi,\Phi,\psi,\Psi,\theta,\eta):=(i\lambda I-\mathcal{A}_F)^{-1}f\in D(\mathcal{A}_F)$, i.e.,
	\begin{equation}\label{2F}
		\begin{cases}
			\text{(a):~}&i\lambda\phi-\Phi=f^1,\\ \text{(b):~}&i\lambda\rho_1\Phi-\kappa[\omega(\phi_x+\psi)+\smallint_0^{\infty}g(s)\eta_x(s)ds]_x+\sigma\theta_x=\rho_1f^2,\\
			\text{(c):~}&i\lambda\psi-\Psi=f^3,\\ \text{(d):~}&i\lambda\rho_2\Psi-b\psi_{xx}+\kappa[\omega(\phi_x+\psi)+\smallint_0^{\infty}g(s)\eta_x(s)ds]-\sigma\theta=\rho_2f^4,\\ \text{(e):~}&i\lambda\rho_3\theta-\beta\theta_{xx}+\sigma(\Phi_x+\Psi)=\rho_3f^5,\\
			\text{(f):~}&i\lambda\eta+\eta_s-(\Phi+\tilde{\Psi})=f^6.
		\end{cases}
	\end{equation}
	
	Next, we will prove the conclusion 1 of Theorem \ref{thmboundresolvent} by several steps. In the proof, we let $C$ be a general positive constant independent of $\lambda$.
	
	\noindent\textbf{Step 1:} there holds
	\begin{align}\label{b2}
		\|\theta_x\|^2,~\smallint_0^{\infty}[-g'(s)]\|\eta_x(s)\|^2ds,~\|\eta\|_{L_{\tilde{g}}^2}^2\leq C\|z\|_{\mathcal{H}_F}\|f\|_{\mathcal{H}_F},
	\end{align}
	where $\tilde{g}$ is defined in Lemma \ref{Ig}.
	
	Taking the $\mathcal{H}_F$-inner product of $(i\lambda I-\mathcal{A}_F)z=f$ with $z$, and taking the real part of the inner product, we get from \eqref{4} that
	\begin{align*} \beta\|\theta_x\|^2-\frac{\kappa}{2}\smallint_0^{\infty}g'(s)\|\eta_x(s)\|^2ds=-\operatorname{Re}(\mathcal{A}_Fz,z)_{\mathcal{H}_F}=\operatorname{Re}((i\lambda I-\mathcal{A}_F)z,z)_{\mathcal{H}_F}=\operatorname{Re}(f,z)_{\mathcal{H}_F}\leq \|z\|_{\mathcal{H}_F}\|f\|_{\mathcal{H}_F}.
	\end{align*}
	Then,
	\begin{align*}
		\|\eta\|_{L_{\tilde{g}}^2}^2=\smallint_0^\infty \tilde{g}\|\eta_{x}\|^2ds=\smallint_J g\|\eta_{x}\|^2ds=-\alpha\smallint_J g'(s)\|\eta_{x}\|^2ds\le-\alpha\smallint_0^\infty g'(s)\|\eta_{x}\|^2ds\le\frac2\kappa\|z\|_{\mathcal{H}_F}\|f\|_{\mathcal{H}_F}.
	\end{align*}
	We get the desired result \eqref{b2} from the above two estimates.
	
	\noindent\textbf{Step 2:} there holds
	\begin{equation}\label{b1}
		\|\phi_x+\psi\|^2\leq \frac{C}{I_g(\lambda)}\left(\frac{1}{I_g(\lambda)}+\frac{1}{|\lambda|}+1\right)\|z\|_{\mathcal{H}_F}\|f\|_{\mathcal{H}_F}.
	\end{equation}
	
	Since \eqref{2F}-(a,c), $\Phi=i\lambda\phi-f^1$, $\tilde{\Psi}=i\lambda\tilde{\psi}-\tilde{f^3}$, by solving the ODE \eqref{2F}-(f), we have
	\begin{equation}\label{eta}
		\eta(s)=(1-e^{-i\lambda s})(\phi+\tilde{\psi})-\frac{1}{i\lambda}(1-e^{-i\lambda s})(f^1+\tilde{f^3})+\smallint_0^se^{-i\lambda(s-\tau)}f^6(\tau)d\tau.
	\end{equation}
	Taking the $L^2_{\tilde{g}}$-inner product of \eqref{eta} with $\phi+\tilde{\psi}$, we get
	\begin{align*} \smallint_0^{\infty}\tilde{g}(s)(\eta_x(s),\phi_x+\psi)ds&=\|\phi_x+\psi\|^2\smallint_0^{\infty}\tilde{g}(s)(1-e^{-i\lambda s})ds-\frac{1}{i\lambda}\smallint_0^{\infty}\tilde{g}(s)(1-e^{-i\lambda s})(f^1_x+f^3,\phi_x+\psi)\\
		&+\smallint_0^{\infty}\tilde{g}(s)\smallint_0^se^{-i\lambda(s-\tau)}(f^6_x(\tau),\phi_x+\psi )d\tau ds.
	\end{align*}
	Observe that $
	\operatorname{Re}\smallint_0^{\infty} \tilde{g}(s)(1 - e^{-i\lambda s})\,ds = \smallint_0^\infty \tilde{g}(s)(1 - \cos(\lambda s))\,ds = I_g(\lambda),$ where $I_g(\lambda)$ is the spectral function defined in Lemma~\ref{Ig}. Taking the real part of the preceding identity, we obtain:
	\begin{align*} I_g(\lambda)\|\phi_x+\psi\|^2=&\operatorname{Re}\left[\smallint_0^{\infty}\tilde{g}(s)(\eta_x(s),\phi_x+\psi)ds+\frac{1}{i\lambda}\smallint_0^{\infty}\tilde{g}(s)(1-e^{-i\lambda s})(f^1_x+f^3,\phi_x+\psi)ds\right]\\ &-\operatorname{Re}\left[\smallint_0^{\infty}\tilde{g}(s)\smallint_0^se^{-i\lambda(s-\tau)}(f^6_x(\tau),\phi_x+\psi )d\tau ds\right]\\
		\leq & \|\phi_x+\psi\|\left[\|\eta\|_{L^2_{\tilde{g}}}+\frac{2}{|\lambda|}\|f_x^1+f^3\|+\smallint_0^{\infty}g(s)\smallint_0^s\|f^6_x(\tau)\|d\tau ds\right]\\
		\leq&\frac{I_g(\lambda)}{2}\|\phi_x+\psi\|^2+C\left[\frac{1}{I_g(\lambda)}\|\eta\|^2_{L^2_{\tilde{g}}}+\left(\frac{1}{|\lambda|}+1\right)\|z\|_{\mathcal{H}_F}\|f\|_{\mathcal{H}_F}\right]\\
		\leq&\frac{I_g(\lambda)}{2}\|\phi_x+\psi\|^2+C\left(\frac{1}{I_g(\lambda)}+\frac{1}{|\lambda|}+1\right)\|z\|_{\mathcal{H}_F}\|f\|_{\mathcal{H}_F},
	\end{align*}
	where we have used Lemma \ref{b2} for $\|\eta\|_{L_{\tilde{g}}^2}^2\leq C\|z\|_{\mathcal{H}_F}\|f\|_{\mathcal{H}_F}$ and the following estimates: first, by Assumption \ref{g1},
	$$\smallint_0^\infty \tilde{g}(s)ds\le\smallint_0^\infty g(s)ds=\ell\in(0,1),$$
	and, by Lemma \ref{lemggj},
	\begin{align*}
		\smallint_0^\infty g(s)\smallint_0^s\|f^6_x(\tau)\|d\tau ds\le\left(\smallint_0^\infty g(s)ds\right)^{\frac12}\left(\smallint_0^\infty g(s)\left(\smallint_0^s\|f^6_x(\tau)\|d\tau\right)^2\right)^{\frac12}\le C\|f^6\|_{L_g^2}.
	\end{align*}
	Therefore, we arrive at the desired result \eqref{b1}.
	
	\noindent\textbf{Step 3:} 	there holds
	\begin{align}\label{b7}
		\| \eta \|_{L^2_g}^2 \leq C \left( \frac{1}{I_g(\lambda)} + 1 \right) \left( \frac{1}{I_g(\lambda)} + \frac{1}{|\lambda|} + 1 \right) \| z \|_{\mathcal{H}_F} \| f \|_{\mathcal{H}_F}.
	\end{align}

	Taking the inner product of \eqref{eta} with \(\eta\) in \(L^2_g\), we have
	\begin{align*}
		\| \eta \|_{L^2_g}^2 = &\smallint_0^\infty g(s) \smallint_0^s e^{-i\lambda(s-\tau)} (f_x^6(\tau), \eta_x(s)) \, d\tau ds + \smallint_0^\infty g(s)(1-e^{-i\lambda s})(\phi_x + \psi, \eta_x(s)) \, ds \\
		&- \frac{1}{i\lambda} \smallint_0^\infty g(s)(1-e^{-i\lambda s})(f_x^1 + f^3, \eta_x(s)).
	\end{align*}
	Using the H\"older inequality and Lemmas \ref{lemggj}, \eqref{b2} and \eqref{b1}, we arrive at
	\begin{align*}
		\| \eta \|_{L^2_g}^2 &\leq C \| \eta \|_{L^2_g} \left[ \| \phi_x + \psi \| + \frac{1}{|\lambda|} \| f_x^1 + f^3 \| + \left( \smallint_0^\infty g(s) \left( \smallint_0^s \| f_x^6(\tau) \| \, d\tau \right)^2 \, ds \right)^{1/2} \right] \\
		&\leq \frac{1}{2} \| \eta \|_{L^2_g}^2 + C\left[ \| \phi_x + \psi \|^2 + \left( \frac{1}{|\lambda|} + 1 \right) \| z \|_{\mathcal{H}_F} \| f \|_{\mathcal{H}_F} \right] \\
		&\leq \frac{1}{2} \| \eta \|_{L^2_g}^2 + C \left( \frac{1}{I_g(\lambda)} + 1 \right) \left( \frac{1}{I_g(\lambda)} + \frac{1}{|\lambda|} + 1 \right) \| z \|_{\mathcal{H}_F} \| f \|_{\mathcal{H}_F}.
	\end{align*}
	The desired estimate \eqref{b7} follows.

	\noindent\textbf{Step 4:} 	there holds
	\begin{equation}\label{b3}
		\rho_1\|\Phi\|^2\leq C\left(\frac{1}{I_g(\lambda)}+1\right)\left(\frac{1}{I_g(\lambda)}+\frac{1}{|\lambda|}+1\right)\|z\|_{\mathcal{H}_F}\|f\|_{\mathcal{H}_F}+\frac14\|z\|_{\mathcal{H}_F}^2.
	\end{equation}
	
	Taking the $L^2$-inner of \eqref{2F}-(b) with $\phi$, we obtain \[i\lambda\rho_1(\Phi,\phi)+\kappa\omega(\phi_x+\psi,\phi_x)+\kappa\smallint_0^\infty g(s)(\eta_x(s),\phi_x)ds+\sigma(\theta_x,\phi)=\rho_1(f^2,\phi).\]
	Since, by \eqref{2F}-(b), $i\lambda\phi=\Phi+f^1$. It follows from Poincar\'e's inequality and Lemmas \ref{b2}--\ref{b7}, that
	\begin{align*} \rho_1\|\Phi\|^2&=-(\Phi,f^1)-\rho_1(f^2,\phi)+\kappa\omega(\phi_x+\psi,\phi_x)+\kappa\smallint_0^{\infty}g(s)(\eta_x(s),\phi_x)ds+\sigma(\theta,\phi_x)\\
		&\leq C\|\Phi\|\|f^1\|+C\|\phi_x\|\|f^2\|+ C\|\phi_x\|\left(\|\phi_x+\psi\|+\|\eta\|_{L_g^2}+\|\theta\|\right)\\
		&\leq C\|z\|_{\mathcal{H}_F}\|f\|_{\mathcal{H}_F}+ C\left(\|\phi_x+\psi\|+\|\psi\|\right)\left(\|\phi_x+\psi\|+\|\eta\|_{L_g^2}+\|\theta\|\right)\\
		&\leq C\|z\|_{\mathcal{H}_F}\|f\|_{\mathcal{H}_F}+ C\left(\|\phi_x+\psi\|^2+\|\eta\|^2_{L_g^2}+\|\theta_x\|^2\right)+\frac14\|\psi_x\|^2\\
		&\leq C\left(\frac{1}{I_g(\lambda)}+1\right)\left(\frac{1}{I_g(\lambda)}+\frac{1}{|\lambda|}+1\right)\|z\|_{\mathcal{H}_F}\|f\|_{\mathcal{H}_F}+\frac14\|z\|_{\mathcal{H}_F}^2.
	\end{align*}
	We get the desired conclusion \eqref{b3}.
	
	\noindent\textbf{Step 5:} there holds
	\begin{align}\label{b4}
		\rho_2\|\Psi\|^2&\leq C(\epsilon)\left(\frac{1}{I_g(\lambda)}+1\right)\left(\frac{1}{I_g(\lambda)}+\frac{1}{|\lambda|}+1\right)\|z\|_{\mathcal{H}_F}\|f\|_{\mathcal{H}_F}\notag\\ &+C|\chi_0|(1+|\lambda|^2)\left(\frac{1}{I_g(\lambda)}+1\right)\left(\frac{1}{I_g(\lambda)}+\frac{1}{|\lambda|}+1\right)\|z\|_{\mathcal{H}_F}\|f\|_{\mathcal{H}_F}+\frac14\|z\|_{\mathcal{H}_F}^2.
	\end{align}
	
	Taking the $L^2$-inner product of \eqref{2F}-(d) with $\omega(\phi_x+\psi)+\smallint_0^{\infty}g(s)\eta_x(s)ds$ and using \eqref{2F}-(a,c), we get
	\begin{align*} {\omega\|\Psi\|^2}=&-\omega(\Phi_x,\Psi)-\omega(f_x^1,\Psi)-\omega(f^3,\Psi)-i\lambda\smallint_0^{\infty}g(s)(\eta_x(s),\Psi)ds\\
		&-\frac{b}{\rho_2}\left(\omega(\phi_x+\psi)+\smallint_0^{\infty}g(s)\eta_{x}(s)ds,\psi_{xx}\right)+\frac{\kappa}{\rho_2}\left\|\omega(\phi_x+\psi)+\smallint_0^{\infty}g(s)\eta_x(s)ds\right\|^2\\
		&-\frac{\sigma}{\rho_2}\left(\omega(\phi_x+\psi)+\smallint_0^{\infty}g(s)\eta_x(s)ds,\theta\right)-\left(\omega(\phi_x+\psi)+\smallint_0^{\infty}g(s)\eta_x(s)ds,f^4\right).
	\end{align*}
	Taking the $L^2_g$-inner product of \eqref{2F}-(f) with $\tilde{\Psi}$ and using $\smallint_0^\infty g(s)ds=\ell$, we obtain $$\ell\|\Psi\|^2=i\lambda\smallint_0^{\infty}g(s)(\eta_x(s),\Psi)ds+\smallint_0^{\infty}g(s)(\eta_{sx}(s),\Psi)ds-\ell(\Phi_x,\Psi)-\smallint_0^{\infty}g(s)(f^6_x(s),\Psi)ds.$$
	Taking the $L^2$-inner product of \eqref{2F}-(b) with $\psi_x$ and using \eqref{2F}-(c), we obtain
	\begin{align*}			{(\Phi,\Psi_x)=-(\Phi,f_x^3)+\frac{\kappa}{\rho_1}\left(\omega(\phi_x+\psi)+\smallint_0^{\infty}g(s)\eta_x(s)ds,\psi_{xx}\right)+\frac{\sigma}{\rho_1}(\theta_x,\psi_x)-(f^2,\psi_x)}.
	\end{align*}
	
	Since $\ell+\omega=1$, adding all the results obtained in the above and extracting the real part on both sides of the resulting identity, we arrive at
	\begin{equation}\label{gjdapusi}
		\rho_2\|\Psi\|^2=\operatorname{Re}[\rho_2s_1+s_2+\rho_2s_3+\chi_0\rho_2s_4],
	\end{equation}
	where
	\begin{align*}
		{s_1}=&-(\Phi,f_x^3)-\omega(f_x^1+f^3,\Psi)+\frac{\sigma}{\rho_1}(\theta_x,\psi_x)-\frac{\sigma}{\rho_2}\left(\omega(\phi_x+\psi)+\smallint_0^{\infty}g(s)\eta_x(s)ds,\theta\right)-(f^2,\psi_x)\\
		&-\smallint_0^{\infty}g(s)(f^6_x,\Psi)-\left(\omega(\phi_x+\psi)+\smallint_0^{\infty}g(s)\eta_x(s)ds,f^4\right),\\
		{s_2}=&\kappa\left\|\omega(\phi_x+\psi)+\smallint_0^{\infty}g(s)\eta_x(s)ds\right\|^2,~s_3=\smallint_0^{\infty}g(s)(\eta_{sx}(s),\Psi)ds,\\
		{s_4}=&\left(\omega(\phi_x+\psi)+\smallint_0^{\infty}g(s)\eta_x(s)ds,\psi_{xx}\right).
	\end{align*}
	Using Lemmas \ref{b2}-\ref{b7}, Poincar\'e's inequality, and Young's inequality, we get
	\begin{align*}
		|s_1|&\leq C\|z\|_{\mathcal{H}_F}\|f\|_{\mathcal{H}_F}+C\|\theta\|\|\phi_x+\psi\|+C\|\theta\|\|\eta\|_{L^2_g}+C\|\theta_x\|\|\psi_x\|\\
		&\leq C\|z\|_{\mathcal{H}_F}\|f\|_{\mathcal{H}_F}+C\|\theta_x\|^2+C\|\phi_x+\psi\|^2+C\|\eta\|^2_{L^2_g}+\frac{1}{4\rho_2} \|\psi_x\|^2\\
		&\leq C\left(\frac{1}{I_g(\lambda)}+1\right)\left(\frac{1}{I_g(\lambda)}+\frac{1}{|\lambda|}+1\right)\|z\|_{\mathcal{H}_F}\|f\|_{\mathcal{H}_F}+\frac{1}{4\rho_2} \|z\|_{\mathcal{H}_F}^2.
	\end{align*}
	Similarly, using Lemma \ref{b7} and Lemma \ref{b1}, we have
	\begin{align*}
		|s_2|&\leq C\|\phi_x+\psi\|^2+C\|\eta\|^2_{L^2_g}\leq C\left(\frac{1}{I_g(\lambda)}+1\right)\left(\frac{1}{I_g(\lambda)}+\frac{1}{|\lambda|}+1\right)\|z\|_{\mathcal{H}_F}\|f\|_{\mathcal{H}_F}.
	\end{align*}
	For $s_3$, let $\{s_n\}_{n=1}^\infty$ be the sequence got in proving \eqref{jxsn}, then we get
	\begin{align*}
		\limsup_{n\to\infty}|g(s_n)(\eta_x(s_n),\Psi)|&\le\limsup_{n\to\infty}\sqrt{g(s_n)}\sqrt{g(s_n)}\|\eta_x(s_n)\|\|\psi\|\\
		&\le\sqrt{g(0)}\|\Psi\|\sqrt{\lim_{n\to\infty}\left(g(s_n)\|\eta_x(s_n)\|^2\right)}=0.
	\end{align*}
	Since $\eta_x(0)$, then by integrating part we get,
	\begin{align*}
		s_3&=\lim_{n\to\infty}\smallint_0^{s_n}g(s)(\eta_{sx}(s),\Psi)ds=\lim_{n\to\infty}\smallint_0^{s_n}g(s)\frac{d}{ds}(\eta_{x}(s),\Psi)ds\\
		&=-\lim_{n\to\infty}\smallint_0^{s_n}g'(s)(\eta_{x}(s),\Psi)ds=-\smallint_0^{\infty}g'(s)(\eta_x(s),\Psi)ds.
	\end{align*}
	Since by Assumption \ref{g1}, $g(\cdot)>0$, $g(\cdot)$ is nonincreasing and $\smallint_0^\infty g(s)ds<\infty$, we get $\lim_{s\to\infty}g(s)=0$. Then by Lemma \ref{b2}, we deduce
	\begin{align*}
		|s_3|\le&-\smallint_0^\infty g'(s)\|\eta_x(s)\|ds\|\Psi\|\le C\left(\smallint_0^\infty-g'(s)\|\eta_x(s)\|ds\right)^2+\frac14\|\Psi\|^2\\
		\le& C\left(\smallint_0^\infty-g'(s)ds\right)\left(\smallint_0^\infty -g'(s)\|\eta_x(s)\|^2ds\right)+\frac14\|\Psi\|^2\\
		=& Cg(0)\left(-\smallint_0^{\infty}g'(s)\|\eta_x(s)\|^2ds\right)+\frac{1}{4}\|\Psi\|^2\leq C\|z\|_{\mathcal{H}_F}\|f\|_{\mathcal{H}_F}+\frac{1}{4}\|\Psi\|^2.
	\end{align*}
	If $\chi_0=0$, we deduce the estimate \eqref{b4} by estimates $s_1-s_3$ and \eqref{gjdapusi}. If $\chi_0\neq0$, we use \eqref{2F}-(d) to rewrite $s_4$ as $$s_4={\frac{1}{b}s_2+\frac{1}{b}\left(\omega(\phi_x+\psi)+\smallint_0^{\infty}g(s)\eta_x(s)ds,i\lambda\rho_2\Psi-\rho_2f^4-\sigma\theta\right)}.$$
	Then by Lemmas \ref{b2}-\ref{b7}, Poincar\'e's inequality, and Young's inequality, we obtain
	\begin{align*}
		|s_4|&\leq C|s_2|+C\|z\|_{\mathcal{H}_F}\|f\|_{\mathcal{H}_F}+C\|\theta\|^2+C\|\phi_x+\psi\|^2+C\|\eta\|^2_{L^2_g}+C|\lambda|\|\Psi\|\left(\|\phi_x+\psi\|+\|\eta\|_{L^2_g}\right)\\
		&\leq C\left(1+|\lambda|^2\right)\left(\frac{1}{I_g(\lambda)}+1\right)\left(\frac{1}{I_g(\lambda)}+\frac{1}{|\lambda|}+1\right)\|z\|_{\mathcal{H}_F}\|f\|_{\mathcal{H}_F}+\frac{1}{4|\chi_0|}\|\Psi\|^2.
	\end{align*}
	Collecting the above estimates and plugging the results into \eqref{gjdapusi}, we obtain the result \eqref{b4} immediately.
	
	\noindent\textbf{Step 6:} 	there holds
	\begin{align}\label{b5}
		b\|\psi_x\|_2^2&\leq C\left(\frac{1}{I_g(\lambda)}+1\right)\left(\frac{1}{I_g(\lambda)}+\frac{1}{|\lambda|}+1\right)\|z\|_{\mathcal{H}_F}\|f\|_{\mathcal{H}_F}\notag\\
		&+C|\chi_0|(1+|\lambda|^2)\left(\frac{1}{I_g(\lambda)}+1\right)\left(\frac{1}{I_g(\lambda)}+\frac{1}{|\lambda|}+1\right)\|z\|_{\mathcal{H}_F}\|f\|_{\mathcal{H}_F}+\frac18\|z\|^2_{\mathcal{H}_F}.
	\end{align}
	
	Taking the $L^2$-inner product of \eqref{2F}-(d) with $\psi$ ,  and using \eqref{2F}-(c) and Poincar\'e's inequality, we have
	\begin{align*} b\|\psi_x\|^2&=\rho_2(\Psi,\Psi+f^3)-\kappa\omega(\phi_x+\psi,\psi)-\kappa\smallint_0^{\infty}g(s)(\eta_x(s),\psi)ds+\sigma(\theta,\psi)+\rho_2(f^4,\psi)\\
		&\leq \rho_2\|\Psi\|^2+C\|z\|_{\mathcal{H}_F}\|f\|_{\mathcal{H}_F}+C\|\psi\|\left(\|\phi_x+\psi\|+\|\eta\|_{L_g^2}+\|\theta\|\right)\\
		&\leq\rho_2\|\Psi\|^2+C\|z\|_{\mathcal{H}_F}\|f\|_{\mathcal{H}_F}+ C\left(\|\phi_x+\psi\|^2+\|\eta\|_{L_g^2}^2+\|\theta_x\|^2\right)+\frac b 2\|\psi_x\|^2.
	\end{align*}
	The the conclusion \eqref{b5} follows from Lemmas \ref{b2}-\ref{b7} and Lemma \ref{b4}.
	
	\textbf{Finally}, using the estimates got in Steps 1-6, we get
	\begin{align*}
		\|z\|_{\mathcal{H}_F}\leq C\left(\frac{1}{I_g(\lambda)}+1\right)\left(\frac{1}{I_g(\lambda)}\!+\!\frac{1}{|\lambda|}\!+\!1\right)\|f\|_{\mathcal{H}_F} +C|\chi_0|(1+|\lambda|^2)\left(\frac{1}{I_g(\lambda)}+1\right)\left(\frac{1}{I_g(\lambda)}\!+\!\frac{1}{|\lambda|}\!+\!1\right)\|f\|_{\mathcal{H}_F}.
	\end{align*}
	Then the conclusion follows.
\end{proof}
\begin{proof}[Proof of the conclusion 2 of Theorem \ref{thmboundresolvent}]
	Let $f=(f^1,f^2,f^3,f^4,f^5,f^6,f^7)\in \mathcal{H}_C$ be given. Since we assumed $i\mathbb{R}\subset\rho(\mathcal{A}_C)$, we can let $z = (\phi,\Phi,\psi,\Psi,\theta, q,\eta):=(i\lambda I-\mathcal{A}_C)^{-1}f\in D(\mathcal{A}_C)$, i.e.,
	\begin{equation}\label{2C}
		\begin{cases}
			\text{(a):~}&i\lambda\phi-\Phi=f^1,\\ \text{(b):~}&i\lambda\rho_1\Phi-\kappa[\omega(\phi_x+\psi)+\smallint_0^{\infty}g(s)\eta_x(s)ds]_x+\sigma\theta_x=\rho_1f^2,\\
			\text{(c):~}&i\lambda\psi-\Psi=f^3,\\ \text{(d):~}&i\lambda\rho_2\Psi-b\psi_{xx}+\kappa[\omega(\phi_x+\psi)+\smallint_0^{\infty}g(s)\eta_x(s)ds]-\sigma\theta=\rho_2f^4,\\ \text{(e):~}&i\lambda\rho_3\theta+q_x+\sigma(\Phi_x+\Psi)=\rho_3f^5,\\
			\text{(f):~}&i\lambda\tau q+\beta q+\theta_x=\tau f^6,\\
			\text{(g):~}&i\lambda \eta+\eta_s-(\Phi+\tilde{\Psi})=f^7.
		\end{cases}
	\end{equation}
	
	Next, we will prove the conclusion 2 of Theorem \ref{thmboundresolvent} by several steps. In the proof, we let $C$ be a general positive constant independent of $\lambda$.
	
	\noindent\textbf{Step 1:} there holds
	\begin{align}\label{b01}
		\|q\|^2,~-\smallint_0^{\infty}g'(s)\|\eta_x(s)\|^2ds\leq C\|z\|_{\mathcal{H}_C}\|f\|_{\mathcal{H}_C},
	\end{align}
	
	Similarly to \eqref{hs1}, we get
	\begin{equation}\label{Re} \operatorname{Re}(\mathcal{A}_Cz,z)_{\mathcal{H}_C}=-\beta\|q\|^2+\frac{\kappa}{2}\smallint_0^{\infty}g'(s)\|\eta_x(s)\|^2ds\leq0.
	\end{equation}
	Note $-\operatorname{Re}(\mathcal{A}_Cz,z)_{\mathcal{H}_C}=\operatorname{Re}(i\lambda I-\mathcal{A}_Cz,z)_{\mathcal{H}_C}=\operatorname{Re}(f,z)_{\mathcal{H}_C}$ and $g$ is nonincresing. Then the conclusion \eqref{b01} follows from the above inequality immediately.
	
	\noindent\textbf{Step 2:}  for every $\epsilon>0$, there exists a constant $C_{\epsilon} > 0$, independent of $\lambda$, such that
	\begin{equation}\label{t2} \|\theta\|^2\leq\epsilon\|z\|^2_{\mathcal{H}_C}+C_{\epsilon} \|z\|_{\mathcal{H}_C}\|f\|_{\mathcal{H}_C}
	\end{equation}
	
	Integrating \eqref{2C}-(f) from $0$ to $x$ and then taking $L^2$-inner product with $\theta$, since $\smallint_0^L\theta(x)dx=0$, we get
	\begin{equation*}
		\tau\left(\smallint_0^xi\lambda qdy,\theta\right)+\beta\left(\smallint_0^xqdy,\theta\right)+\|\theta\|^2=\tau\left(\smallint_0^xf^6dy,\theta\right)
	\end{equation*}
	Thus, using H\"older's and Poincar\'e's inequality and \eqref{b01}, we have
	\begin{align}
		\|\theta\|^2&\leq C\left|\left(\smallint_0^xi\lambda qdy,\theta\right)\right|+C\|q\|\|\theta\|+C\|f^6\|\|\theta\|\nonumber\\
		&\leq C\left|\left(\smallint_0^xi\lambda qdy,\theta\right)\right|+C\|q\|\|z\|_{\mathcal{H}_C}+C\|z\|_{\mathcal{H}_C}\|f\|_{\mathcal{H}_C}.\label{22200}
	\end{align}
	Considering $\left(\smallint_0^xi\lambda qdy,\theta\right)$, using \eqref{2C}-(e) and integrating by parts, we get
	\begin{align}
		&\left|\left(\smallint_0^xi\lambda qdy,\theta\right)\right|=\left|\left(\smallint_0^xqdy,-i\lambda\theta\right)\right|\nonumber\\
		=&\left|-\left(\smallint_0^xqdy,f^5\right)+\frac{1}{\rho_3}\left(\smallint_0^xqdy,q_x\right)+\frac{\sigma}{\rho_3}\left(\smallint_0^xqdy,\Phi_x+\Psi\right)\right|\nonumber\\
		=&\left|-\left(\smallint_0^xqdy,f^5\right)+\frac{1}{\rho_3}\left[\overline{q}(L)\smallint_0^Lqdx-\|q\|^2\right]+\frac{\sigma}{\rho_3}\left[\overline{\Phi}(L)\smallint_0^Lqdx-(q,\Phi)+\left(\smallint_0^xqdy,\Psi\right)\right]\right|\nonumber\\
		=&\left|-\left(\smallint_0^xqdy,f^5\right)+\frac{1}{\rho_3}(\overline{q}(L)+\sigma\overline{\Phi}(L))\smallint_0^Lqdx-\frac{1}{\rho_3}\|q\|^2+\frac{\sigma}{\rho_3}\left(\smallint_0^xqdy,\Psi\right)\right|\nonumber\\
		\leq& \frac{1}{\rho_3}\left|(q(L)+\sigma\Phi(L))\smallint_0^L\overline{q}dx\right|+C\|q\|\|f^5\|+C\|q\|^2+C\|q\|\|\Psi\|\nonumber\\
		\leq & \frac{1}{\rho_3}\left|(q(L)+\sigma\Phi(L))\smallint_0^L\overline{q}dx\right|+C\|z\|_{\mathcal{H}_C}\|f\|_{\mathcal{H}_C}+C\|q\|^2+C\|q\|\|\Psi\|.\label{2220}
	\end{align}
	
	On the other hand, integrating \eqref{2C}-(e) from $x$ to $L$, it follows
	\begin{equation*} q(L)+\sigma\Phi(L)=q+\sigma\Phi+\smallint_x^L\left(-i\lambda\rho_3\theta-\sigma\Psi+\rho_3f^5\right)dy.
	\end{equation*}
	Therefore,
	\begin{align}\label{222}
		[q(L) + \sigma \Phi(L)] \smallint_0^L \overline{q} \, dx = &\rho_3 \smallint_x^L f^5(y) \, dy \smallint_0^L \overline{q}(x) \, dx + [q(x) + \sigma \Phi(x)] \smallint_0^L \overline{q}(x) \, dx\notag\\
		&+\rho_3 \smallint_x^L \theta(y) \, dy \smallint_0^L \overline{(i \lambda q)}(x) \, dx - \sigma \smallint_x^L \Psi(y) \, dy \smallint_0^L \overline{q}(x) \, dx.
	\end{align}
	Now, using the identity \eqref{2C}-(f), we can rewrite \eqref{222} as
	\begin{align}\label{2222}
		\left|[q(L) + \sigma \Psi(L)] \smallint_0^L \overline{q}(x) \, dx\right| =& \rho_3\bigg| \smallint_x^L f^5(y) \, dy \smallint_0^L \overline{q}(x) \, dx + [q(x) + \sigma \Phi(x)] \smallint_0^L \overline{q}(x) \, dx\notag\\
		&+\frac{\rho_3}{\tau} \smallint_x^L \theta(y) \, dy \smallint_0^L \left(\tau \overline{f^6}-\beta \overline{q}(x)-\overline{\theta}_x\right) \, dx - \sigma \smallint_x^L \Psi(y) \, dy \smallint_0^L \overline{q}(x) \, dx\bigg|\notag\\
		&\leq \rho_3\left| \smallint_x^L f^5(y) \, dy\smallint_0^L \overline{q}(x) \, dx\right| +\left| [q(x) + \sigma \Phi(x)] \smallint_0^L \overline{q}(x) \, dx\right|\notag\\
		&+\frac{\rho_3}{\tau} \left|\smallint_x^L \theta(y) \, dy \smallint_0^L (\tau \overline{f^6}-\beta \overline{q}(x)) \, dx\right| + \sigma\left| \smallint_x^L \Psi(y) \, dy \smallint_0^L \overline{q}(x) \, dx\right|\notag\\
		&+\frac{\rho_3}{\tau} \left|\smallint_x^L \theta(y) \, dy \smallint_0^L \overline{\theta}_x(x)\, dx\right|\notag\\
		&\leq \rho_3\left| \smallint_x^L f^5(y) \, dy \smallint_0^L \overline{q}(x) \, dx\right| + \left|[q(x) + \sigma \Phi(x)] \smallint_0^L \overline{q}(x) \, dx\right|\notag\\
		&+\frac{\rho_3}{\tau} \left|\smallint_x^L \theta(y) \, dy \smallint_0^L (\tau \overline{f^6}-\beta \overline{q}(x)) \, dx\right| + \sigma\left| \smallint_x^L \Psi(y) \, dy \smallint_0^L \overline{q}(x) \, dx\right|\notag\\
		&+\frac{\rho_3}{\tau}\left|\overline{\theta}(L)-\overline{\theta}(0)\right|\left|\smallint_x^L\theta(y)dy\right|.
	\end{align}
	
	Since $\theta\in L_*^2$, i.e., $\smallint_0^L\theta dx=0$, let $x\to 0$, \eqref{2222} turns to
	\begin{equation*}
		\begin{split}
			\left|[q(L) + \sigma \Psi(L)] \smallint_0^L \overline{q}(x) \, dx\right|
			\leq& \rho_3 \smallint_0^L| f^5(x)| \, dx \smallint_0^L |\overline{q}(x)| \, dx \\&+ \left|[q(0) + \sigma \Phi(0)] \smallint_0^L \overline{q}(x) \, dx\right|+ \sigma \smallint_0^L |\Psi(x)| \, dx \smallint_0^L |\overline{q}(x) \, |dx.
		\end{split}
	\end{equation*}
	Using H\"older's inequality we easily deduce by \eqref{2220}
	\begin{align*}
		\left|\left(\smallint_0^xi\lambda qdy,\theta\right)\right|\leq C \|z\|_{\mathcal{H}_C} \|f\|_{\mathcal{H}_{C}} + C \|q\| \|\Phi\| + C \|q\| \|\Psi\| + C \|q\| \|\theta\|+C\|q\|^2.
	\end{align*}
	Thus, going back to \eqref{22200} we arrive at
	\begin{equation*}
		\|\theta\|^2 \leq C \|z\|_{\mathcal{H}_{C}} \|f\|_{\mathcal{H}_{C}} + C \|q\| \|\Phi\| + C \|q\| \|\Psi\| + C \|q\| \|\theta\|+C\|q\|^2+C\|q\|\|z\|_{\mathcal{H}_C}.
	\end{equation*}
	Therefore, using \eqref{b01}, and young's inequality, we deduce, for any $\epsilon>0$, there exists a constant $C_{\epsilon} > 0$, independent of $\lambda$, such that
	\begin{align*}
		\|\theta\|^2 \leq C \|z\|_{\mathcal{H}_{C}} \|f\|_{\mathcal{H}_{C}}+C_\epsilon\|q\|^2+\epsilon\|z\|^2_{\mathcal{H}_C}\leq\epsilon\|z\|^2_{\mathcal{H}_C}+C_\epsilon \|z\|_{\mathcal{H}_{C}} \|f\|_{\mathcal{H}_{C}}.
	\end{align*}
	\noindent\textbf{Step 3:} Similar to the proofs of \eqref{b1} and \eqref{b7}, we get
	\begin{align}
		&\|\phi_x+\psi\|^2\leq \frac{C}{I_g(\lambda)}\left(\frac{1}{I_g(\lambda)}+\frac{1}{|\lambda|}+1\right)\|z\|_{\mathcal{H}_C}\|f\|_{\mathcal{H}_C},\label{t3}\\
		&\| \eta \|_{L^2_g}^2 \leq C \left( \frac{1}{I_g(\lambda)} + 1 \right) \left( \frac{1}{I_g(\lambda)} + \frac{1}{|\lambda|} + 1 \right) \| z \|_{\mathcal{H}_C} \| f \|_{\mathcal{H}_C}.\label{t7}
	\end{align}
	\noindent\textbf{Step 4:}  for every $\epsilon>0$, there exists a constant $C_{\epsilon} > 0$, independent of $\lambda$, such that
	\begin{equation}\label{Phi}
		\rho_1\|\Phi\|^2\leq C_\epsilon\left(\frac{1}{I_g(\lambda)}+1\right)\left(\frac{1}{I_g(\lambda)}+\frac{1}{|\lambda|}+1\right)\| z \|_{\mathcal{H}_{c}} \| f \|_{\mathcal{H}_{c}} + \epsilon  \| z \|_{\mathcal{H}_{c}}^2,
	\end{equation}
	
	Taking the inner $L^2$-product of \eqref{2C}-(b) with $\phi$, we obtain $$-\rho_1\smallint_0^L\Phi(\overline{i\lambda\phi})dx+\kappa\omega\smallint_0^L(\phi_x+\psi)\overline{\phi_x}dx+\kappa\smallint_0^{\infty}g(s)\left(\smallint_0^L\eta_x(s)\overline{\phi_x}dx\right)ds+\sigma\smallint_0^L\theta_x\overline{\phi}dx=\rho_1\smallint_0^Lf^2\overline{\phi}dx.$$
	Using \eqref{2C}-(a) and H\"older's, Young's and Poincar\'e's inequality $\|\psi\|\leq c_0\|\psi_x\|$, we get, for any $\epsilon>0$, there exists a constant $C_{\epsilon} > 0$, independent of $\lambda$, such that
	\begin{align*}
		\rho_1\|\Phi\|^2 &= -\rho_1(\Phi,f^1)-\rho_1(f^2,\phi)+\kappa\omega(\phi_x+\psi,\phi_x)+\kappa\smallint_0^{\infty}g(s)(\eta_x(s),\phi_x)\,ds - \sigma(\theta,\phi_x) \\
		&\leq \kappa\omega \| \phi_x + \psi \| \| \phi_x \| + \kappa \| \eta \|_{L_g^2} \| \phi_x \| + \sigma \| \theta \| \| \phi_x \| + \rho_1 \| \Phi \| \| f^1 \| + \rho_1 \| \phi \| \| f^2 \| \\
		&\leq (\| \phi_x + \psi \| + \| \psi \|)(\kappa\omega \| \phi_x + \psi \| + \kappa \| \eta \|_{L_g^2} + \sigma \| \theta \|) + C \| z \|_{\mathcal{H}_{C}} \| f \|_{\mathcal{H}_{C}} \\
		&\leq C \left( \| \phi_x + \psi \|^2 + \| \eta \|_{L_g^2}^2 \right) + \sigma \| \phi_x + \psi \| \| \theta \| + \| \psi \| (\kappa\omega \| \phi_x + \psi \| + \kappa \| \eta \|_{L_g^2}) + \sigma \| \theta \| \| \psi \| \\
		&\quad + C \| z \|_{\mathcal{H}_{C}} \| f \|_{\mathcal{H}_{C}} \\
		&\leq C\left( \| \phi_x + \psi \|^2 + \| \eta \|_{L_g^2}^2 \right) + \| \theta \|^2 + c_0 \| \psi_x \| (\kappa\omega \| \phi_x + \psi \| + \kappa \| \eta \|_{L_g^2}) + \sigma c_0 \| \theta \| \| \psi_x \| \\
		&\leq C_\epsilon \left( \| \phi_x + \psi \|^2 + \| \eta \|_{L_g^2}^2 +\| \theta \|^2\right) + \epsilon\| \psi_x \|^2.
	\end{align*}
	Then the conclusion follows from \eqref{t2}, \eqref{t3} and \eqref{t7}.
	
	\noindent\textbf{Step 5:}  for every $\epsilon>0$, there exists a constant $C_{\epsilon} > 0$, independent of $\lambda$, such that
	\begin{align}\label{t5}
		\|\Psi\|^2 &\leq C_\epsilon|\chi_0|(|\lambda|^2+1)\left(\frac{1}{I_g(\lambda)}+1\right)\left(\frac{1}{I_{{g}}(\lambda)}+\frac{1}{|\lambda|}+1\right)\|z\|_{\mathcal{H}_C} \|f\|_{\mathcal{H}_C} \nonumber\\
		&\quad + C_\epsilon\left(\frac{1}{I_g(\lambda)}+1\right)\left(\frac{1}{I_{{g}}(\lambda)}+\frac{1}{|\lambda|}+1\right)\|z\|_{\mathcal{H}_C} \|f\|_{\mathcal{H}_C} +C_\epsilon|\lambda|^2\|z\|_{\mathcal{H}_C} \|f\|_{\mathcal{H}_C} + \epsilon\|z\|_{\mathcal{H}_C}^2.
	\end{align}
	
	Taking the derivative of both sides of the equation \eqref{2C}-(a)  with respect to $x$ and then taking the $L^2$-inner product with $\Psi$, we get
	\begin{equation}\label{4.1}
		-i \lambda ( \Psi,\phi_x) = (\Psi,\Phi_x) + (\Psi,f^1_x).
	\end{equation}
	Taking the inner $L^2$-product of \eqref{2C}-(c) with $\Psi$ , we obtain
	\begin{equation}\label{4.2}
		-i \lambda (\Psi,\psi ) = \|\Psi\|^2 + (\Psi,f^3).
	\end{equation}
	Taking the inner $L^2_g$-product of \eqref{2C}-(g) with $\tilde{\Psi}$,  we have
	\begin{equation}\label{4.3}
		-i \lambda \smallint_0^\infty g(s) (\Psi,\eta_x(s)) \, ds=- \smallint_0^\infty g(s) (\Psi,\eta_{sx}(s)) \, ds + \ell(\Psi,\Phi_x) + \ell\|\Psi\|^2 + \smallint_0^\infty g(s) (\Psi,f_x^7(s)) \, ds.
	\end{equation}
	Taking the $L^2$-inner product of \eqref{2C}-(d) with $\omega (\phi_x + \psi) + \smallint_0^\infty g(s) \eta_x(s) \, ds$ and using \eqref{4.1}, \eqref{4.2} and \eqref{4.3}, since $\ell+\omega=1$, we obtain
	\begin{align}\label{4.4}
		\|\Psi\|^2 = &-\frac{b}{\rho_2} \left(\psi_{xx}, \omega(\phi_x + \psi) + \smallint_0^\infty g(s) \eta_x(s) \, ds\right)+ \frac{\kappa}{\rho_2} \left\|\omega(\phi_x + \psi) + \smallint_0^\infty g(s) \eta_x(s) \, ds\right\|^2\notag\\
		&- \frac{\sigma}{\rho_2} \left(\theta, \omega(\phi_x + \psi) + \smallint_0^\infty g(s) \eta_x(s) \, ds\right) - \left(f^4, \omega(\phi_x + \psi) + \smallint_0^\infty g(s) \eta_x(s) \, ds\right)\notag\\
		&+(\Psi_x,\Phi) - \omega( \Psi,f_x^1) - \omega(\Psi, f^3) + \smallint_0^\infty g(s) (\Psi,\eta_{sx}(s) ) ds - \smallint_0^\infty g(s) (\Psi,f_x^7(s)) \, ds.
	\end{align}
	Taking the derivative of both sides of the equation \eqref{2C}-(c)  with respect to $x$ and then taking the $L^2$-inner product with $\Phi$, we get
	\begin{equation}\label{4.6}
		i\lambda ( \psi_x,\Phi) = (\Psi_x, \Phi) + (f_x^3, \Phi).
	\end{equation}
	Taking the $L^2$-inner product of \eqref{2C}-(f) with $\psi_x$, we have
	\begin{equation}\label{4.7}
		( \psi_x,\theta_x) = \tau(\psi_x,f^6)+	i\lambda\tau (\psi_x,q) - \beta ( \psi_x,q).
	\end{equation}
	Taking the $L^2$-inner product of \eqref{2C}-(b) with $\psi_x$ and using \eqref{4.6} and \eqref{4.7}, we obtain
	\begin{equation}\label{4.8}
		\begin{aligned}
			(\Psi_x, \Phi) &= \frac{\kappa}{\rho_1} \left(\psi_{xx},\omega(\phi_x + \psi) + \smallint_0^\infty g(s) \eta_x(s) \, ds\right)+i \lambda \frac{\sigma\tau}{\rho_1} (\psi_x,q ) -\frac{\sigma \beta}{\rho_1} (\psi_x,q)\\
			&+\frac{\sigma \tau}{\rho_1} (\psi_x,f^6)- (f_x^3, \Phi)- (\psi_x,f^2).
		\end{aligned}
	\end{equation}
	Adding \eqref{4.4} and \eqref{4.8}, we deduce
	\begin{equation}\label{Psigj}
		\|\Psi\|^2\le|a_1|+|a_2|+|a_3|+|a_4|+|\chi_0||a_5|,
	\end{equation}
	where,
	\begin{align*}
		a_1 =& -\omega(\Psi, f_x^1 + f^3) - \smallint_0^\infty g(s) (\Psi,f_x^7(s))ds - (f_x^3,\Phi) -(\psi_x,f^2)+\frac{\sigma \tau}{\rho_1} (f^6, \psi_x)-\frac{\sigma\beta}{\rho_1}(\psi_x,q)\\
		&-\frac{\sigma}{\rho_2} \left(\theta, \omega(\phi_x+\psi) + \smallint_0^\infty g(s) \eta_x(s)ds\right )-\left(f^4, \omega(\phi_x + \psi) + \smallint_0^\infty g(s) \eta_x(s) \, ds\right),\\
		a_2 =& \frac{\kappa}{\rho_2} \left\|\omega(\phi_x+\psi) + \smallint_0^\infty g(s) \eta_x(s)ds\right\|^2,\\
		a_3 =& \smallint_0^\infty g(s) (\Psi,\eta_{sx}(s)),\\	
		a_4 =& i \lambda \frac{\sigma \tau}{\rho_1} ( \psi_x,q),\\
		a_5 =& \left(\psi_{xx},\omega(\phi_x+\psi) + \smallint_0^\infty g(s) \eta_x(s)ds  \right).
	\end{align*}
	
	By using \eqref{b01}, \eqref{t2}, \eqref{t3}, \eqref{t7}, and Cauchy's inequality, we have for any $\epsilon>0$, there exists a constant $C_{\epsilon} > 0$, independent of $\lambda$, such that
	\begin{align*}\label{a1}
		|a_1| &\leq C\|\theta\| \left(\omega \|\phi_x + \psi\| + \| \eta \|_{L^2_g}\right) + C \| q \| \| \psi_x \| + C \| z \|_{\mathcal{H}_C} \| f \|_{\mathcal{H}_C} \\
		&\leq C_\epsilon \left(\|\phi_x + \psi\|^2 + \| \eta \|_{L^2_g}^2 + \| q \|^2\right) + \| \theta \|^2 + \epsilon \| \psi_x \|^2 + C \| z \|_{\mathcal{H}_C} \| f \|_{\mathcal{H}_C} \\
		&\leq C_\epsilon \left(\frac{1}{I_g(\lambda)} + 1\right)\left(\frac{1}{I_g(\lambda)} + \frac{1}{|\lambda|} + 1\right) \| z \|_{\mathcal{H}_C} \| f \|_{\mathcal{H}_C} + \epsilon \| z \|_{\mathcal{H}_C}^2.
	\end{align*}
	Using \eqref{t3} and \eqref{t7}, we obtain
	\begin{equation}\label{a2}
		|a_2|\leq C\|\phi_x + \psi\|^2 + C\| \eta \|_{L^2_g}^2\leq C \left(\frac{1}{I_g(\lambda)} + 1\right)\left(\frac{1}{I_g(\lambda)}\!+\!\frac{1}{|\lambda|}\!+\!1\right)\|z\|_{\mathcal{H}_C}\|f\|_{\mathcal{H}_C}.
	\end{equation}
	Similar to the prove of \eqref{jxsn}, we deduce
	\begin{equation*}
		a_3 = \smallint_0^\infty g(s) \frac{d}{ds} \left( \eta_x(s), \Psi \right) ds = -\smallint_0^\infty g'(s) \left( \eta_x(s), \Psi \right) ds.
	\end{equation*}
	Then by using Young's inequality, we obtain, for any $\epsilon>0$,
	\begin{align*}
		|a_3|\le&-\smallint_0^\infty g'(s)\|\eta_x(s)\|\|\Psi\|ds\\
		\le&-\frac{\epsilon}{2}\smallint_0^\infty g'(s)\|\Psi\|^2ds-\frac{1}{2\epsilon}\smallint_0^\infty g'(s)\|\eta_x\|^2ds\le\frac{g(0)\epsilon}{2}\|\Psi\|^2-\frac{1}{2\epsilon}\smallint_0^\infty g'(s)\|\eta_x\|^2ds.
	\end{align*}
	By taking $\epsilon=\frac1{2g(0)}$, it follows from \eqref{b01} that
	\begin{align*}
		|a_3|  \leq C\|z\|_{\mathcal{H}_C}\|f\|_{\mathcal{H}_C} + \frac{1}{4} \|\Psi\|^2.
	\end{align*}
	Using Young's inequality and \eqref{b01}, we deduce any $\epsilon>0$, there exists a constant $C_{\epsilon} > 0$, independent of $\lambda$, such that
	\begin{align*}
		|a_4| &\leq C|\lambda| \|q\| \|\psi_x\| \leq C_\epsilon|\lambda|^2 \|q\|^2 + \epsilon \|\psi_x\|^2 \leq C_\epsilon|\lambda|^2 \|z\|_{\mathcal{H}_C}\|f\|_{\mathcal{H}_C} + \epsilon  \|z\|_{\mathcal{H}_C}^2.
	\end{align*}
	If $\chi_0=0$, we yield \eqref{t5} by \eqref{Psigj}  the above estimates for $|a_i|$, $i=1,2,\cdots,4$. If $\chi_0\neq0$, we use \eqref{2C}-(d) to rewrite $a_5$ as
	\begin{equation*}
		a_5 = \frac{\rho_2}{b} a_2 + \frac{1}{b} \left(i\lambda \rho_2\Psi - \sigma\theta - \rho_2 f^4, \omega(\phi_x + \psi) + \smallint_0^\infty g(s) \eta_x(s) \, ds \right).
	\end{equation*}
	Using H\"older's and Young's inequality, we get,
	\begin{align*}
		|a_5| &\leq C|a_2| + C |\lambda| \|\Psi\| \left(\|\phi_x+\psi\|+\|\eta\|_{L_g^2}\right) + C\|\theta\|\left(\|\phi_x+\psi\|+\eta\|_{L_g^2}\right) + C \|z\|_{\mathcal{H}_C} \|f\|_{\mathcal{H}_C} \\
		&\leq C |a_2| + \frac{1}{4|\chi_0|} \|\Psi\|^2 + \|\theta\|^2 + C (|\lambda|^2+1)\left(\|\phi_x+\psi\|^2+\|\eta\|_{L_g^2}^2\right)+C \|z\|_{\mathcal{H}_C} \|f\|_{\mathcal{H}_C}.
	\end{align*}
	From \eqref{t2} and  \eqref{a2},  we have for any $\epsilon>0$, there exists a constant $C_{\epsilon} > 0$, independent of $\lambda$, such that
	\begin{align*}
		|a_5|\leq C_\epsilon(|\lambda|^2+1)\left(\frac{1}{I_g(\lambda)}+1\right)\left(\frac{1}{I_g(\lambda)}+\frac{1}{|\lambda|}+1\right)\|z\|_{\mathcal{H}_C} \|f\|_{\mathcal{H}_C} + \epsilon\|z\|_{\mathcal{H}_C}^2 + \frac{1}{4|\chi_0|} \|\Psi\|^2.
	\end{align*}
	Finally, \eqref{t5} follows from \eqref{Psigj} and the above estimates for $|a_i|$, $i=1,2,\cdots,5$.

	\noindent\textbf{Step 6:} for every $\epsilon>0$, there exists a constant $C_{\epsilon} > 0$, independent of $\lambda$, such that
	\begin{align}\label{t6}
		b \|\psi_x\|^2&\leq C_\epsilon |\chi_0| (|\lambda|^2 + 1) \left( \frac{1}{I_g(\lambda)} + 1 \right) \left( \frac{1}{I_g(\lambda)} + \frac{1}{|\lambda|} + 1 \right) \| z \|_{\mathcal{H}_C} \| f \|_{\mathcal{H}_C} \notag\\
		&\quad + C_\epsilon \left( \frac{1}{I_g(\lambda)} + 1 \right) \left( \frac{1}{I_g(\lambda)} + \frac{1}{|\lambda|} + 1 \right) \| z \|_{\mathcal{H}_C} \| f \|_{\mathcal{H}_C} +C_\epsilon|\lambda|^2\|z\|_{\mathcal{H}_C} \|f\|_{\mathcal{H}_C}+ \epsilon\| z \|_{\mathcal{H}_C}^2.
	\end{align}
	
	Taking the $L^2$-inner product of \eqref{2C}-(d) with $\psi$, we have
	\begin{equation}\label{psi1}
		-\rho_2(\Psi, i\lambda\psi) + b\|\psi_x\|^2 + \kappa\omega(\phi_x + \psi, \psi) + \kappa\smallint_{0}^{\infty}g(s)(\eta_x(s), \psi)ds - \sigma(\theta, \psi) = \rho_2(f^4, \psi).
	\end{equation}
	Taking \eqref{2C}-(c) into \eqref{psi1} and using Poincar\'e's, H\"older's and Young's inequalities,  we get
	\begin{align*}
		b\| \psi_x \|^2 &= -\kappa \omega (\phi_x + \psi, \psi) - \kappa \smallint_{0}^{\infty} g(s) (\eta_x(s), \psi)ds + \sigma (\theta, \psi) + \rho_2 (f^4, \psi) + \rho_2 \| \Psi\|^2 + \rho_2 (\Psi, f^3) \\
		&\leq C\|\psi\| ( \| \phi_x + \psi \| +\| \eta \|_{L_g^2}) + \sigma \| \theta \| \| \psi \| + \rho_2 \| f^4 \| \| \psi \| + \rho_2 \| \Psi \|^2 + \rho_2 \| f^3 \| \| \Psi \| \\
		&\leq C \|\psi_x\| (\| \phi_x + \psi \| + \| \eta \|_{L_g^2}) +C\| \theta \| \| \psi_x \| + C \| f^4 \| \| \psi_x \| + \rho_2 \| \Psi \|^2 + \rho_2 \| f^3 \| \| \Psi \| \\
		&\leq C \left( \| \phi_x + \psi \|^2 + \| \eta \|_{L_g^2}^2 \right) + C \| \theta \|^2 + \frac{b}{2} \| \psi_x \|^2 + \rho_2 \| \Psi \|^2 + C \| z \|_{\mathcal{H}_C} \| f \|_{\mathcal{H}_C}.
	\end{align*}
	Then by using \eqref{t2}, \eqref{t3}, \eqref{t7}, \eqref{t5}, and the above estimate, we get \eqref{t6}.
	
	{\bf Finally}, using the estimates got in Steps 1-6, we get for every $\epsilon>0$, there exists a constant $C_{\epsilon} > 0$, independent of $\lambda$, such that
	\begin{align*}
		\|z\|_{\mathcal{H}_C}^2&\leq C_\epsilon |\chi_0| (|\lambda|^2 + 1) \left( \frac{1}{I_g(\lambda)} + 1 \right) \left( \frac{1}{I_g(\lambda)} + \frac{1}{|\lambda|} + 1 \right) \| z \|_{\mathcal{H}_C} \| f \|_{\mathcal{H}_C} \\
		&\quad + C_\epsilon \left( \frac{1}{I_g(\lambda)} + 1 \right) \left( \frac{1}{I_g(\lambda)} + \frac{1}{|\lambda|} + 1 \right) \| z \|_{\mathcal{H}_C} \| f \|_{\mathcal{H}_C} +C_\epsilon|\lambda|^2\|z\|_{\mathcal{H}_C} \|f\|_{\mathcal{H}_C}+ \epsilon\| z \|_{\mathcal{H}_C}^2.
	\end{align*}
	By taking $\epsilon=\frac12$, we get
	\begin{align*}
		\|z\|_{\mathcal{H}_C}^2&\leq C |\chi_0| (|\lambda|^2 + 1) \left( \frac{1}{I_g(\lambda)} + 1 \right) \left( \frac{1}{I_g(\lambda)} + \frac{1}{|\lambda|} + 1 \right) \| z \|_{\mathcal{H}_C} \| f \|_{\mathcal{H}_C} \\
		&\quad + C \left( \frac{1}{I_g(\lambda)} + 1 \right) \left( \frac{1}{I_g(\lambda)} + \frac{1}{|\lambda|} + 1 \right) \| z \|_{\mathcal{H}_C} \| f \|_{\mathcal{H}_C}+C|\lambda|^2\|z\|_{\mathcal{H}_C} \|f\|_{\mathcal{H}_C}.
	\end{align*}
	Then the conclusion follows.
\end{proof}
\begin{proof}[Proof of the conclusion 3 of Theorem \ref{thmboundresolvent}]
	The estimates for the Step 1 to Step 4 are the same as the proofs in the previous part.
	
	\noindent\textbf{Step 5:}  For every $\epsilon>0$, there exists a constant $C_{\epsilon} > 0$, independent of $\lambda$, such that
	\begin{align}\label{w5}
		\|\Psi\|^2 &\leq C|\chi_1||\lambda|^2\left(\frac{1}{I_g(\lambda)}+1\right)\left(\frac{1}{I_{{g}}(\lambda)}+\frac{1}{|\lambda|}+1\right)\|z\|_{\mathcal{H}_C} \|f\|_{\mathcal{H}_C} \nonumber\\
		&\quad + C_\epsilon\left(\frac{1}{I_g(\lambda)}+1\right)\left(\frac{1}{I_{{g}}(\lambda)}+\frac{1}{|\lambda|}+1\right)\|z\|_{\mathcal{H}_C} \|f\|_{\mathcal{H}_C} + \epsilon\|z\|_{\mathcal{H}_C}^2.
	\end{align}
	
	Taking the derivative of both sides of the equation \eqref{2C}-(a)  with respect to $x$ and then taking the $L^2$-inner product with $\Psi$, we get
	\begin{equation}\label{5.1}
		i \lambda(\phi_x,\Psi) = -(\Phi,\Psi_x) + (f^1_x,\Psi,).
	\end{equation}
	Taking the inner $L^2$-product of \eqref{2C}-(c) with $\Psi$ , we obtain
	\begin{equation}\label{5.2}
		i \lambda (\psi,\Psi) = \|\Psi\|^2 + (f^3,\Psi).
	\end{equation}
	Taking the inner $L^2_g$-product of \eqref{2C}-(g) with $\tilde{\Psi}$,  we have
	\begin{equation}\label{5.3}
		\smallint_0^\infty g(s) (i \lambda \eta_x(s),\Psi) \, ds=- \smallint_0^\infty g(s) (\eta_{sx}(s),\Psi) \, ds + \ell(\Phi_x,\Psi) + \ell\|\Psi\|^2 + \smallint_0^\infty g(s) (f_x^7(s),\Psi) \, ds.
	\end{equation}
	Taking the $L^2$-inner product of \eqref{2C}-(d) with $\omega (\phi_x + \psi) + \smallint_0^\infty g(s) \eta_x(s) \, ds$ and using \eqref{5.1}, \eqref{5.2} and \eqref{5.3}, since $\ell+\omega=1$, we obtain
	\begin{align}\label{5.4}
		\|\Psi\|^2 = &-\frac{b}{\rho_2} \left( \omega(\phi_x + \psi) + \smallint_0^\infty g(s) \eta_x(s) \, ds,\psi_{xx}\right)+ \frac{\kappa}{\rho_2} \left\|\omega(\phi_x + \psi) + \smallint_0^\infty g(s) \eta_x(s) \, ds\right\|^2\notag\\
		&- \frac{\sigma}{\rho_2} \left( \omega(\phi_x + \psi) + \smallint_0^\infty g(s) \eta_x(s) \, ds,\theta\right) - \left( \omega(\phi_x + \psi) + \smallint_0^\infty g(s) \eta_x(s) \, ds,f^4\right)\notag\\
		&+(\Phi,\Psi_x) - \omega(f_x^1, \Psi) - \omega( f^3,\Psi) + \smallint_0^\infty g(s) (\eta_{sx}(s),\Psi ) ds - \smallint_0^\infty g(s) (f_x^7(s),\Psi) \, ds.
	\end{align}
	Taking the derivative of both sides of the equation \eqref{2C}-(c)  with respect to $x$ and then taking the $L^2$-inner product with $\Phi$, we get
	\begin{equation}\label{5.6}
		(\Phi,i\lambda \psi_x) = (\Phi,\Psi_x) + (\Phi,f_x^3).
	\end{equation}
	Taking the $L^2$-inner product of \eqref{2C}-(f) with $\psi_x$, we have
	\begin{equation}\label{5.7}
		( \theta_x,\psi_x) = \tau(f^6,\psi_x)-	\tau (i \lambda q, \psi_x) - \beta ( q,\psi_x).
	\end{equation}
	
	Taking the $L^2$-inner product of \eqref{2C}-(b) with $\psi_x$ and using \eqref{5.6}, \eqref{5.7} and \eqref{2C}-(c), we obtain
	\begin{align}\label{5.8}
		(\Phi,\Psi_x) =& \frac{\kappa}{\rho_1} \left(\omega(\phi_x + \psi) + \smallint_0^\infty g(s) \eta_x(s) \, ds,\psi_{xx}\right)- \frac{\sigma\tau}{\rho_1} ( i \lambda q,\psi_x ) -\frac{\sigma \beta}{\rho_1} (q,\psi_x)\notag\\
		&+\frac{\sigma \tau}{\rho_1} (f^6,\psi_x)- (\Phi,f_x^3)- (f^2,\psi_x)\notag\\
		=& \frac{\kappa}{\rho_1} \left(\omega(\phi_x + \psi) + \smallint_0^\infty g(s) \eta_x(s) \, ds,\psi_{xx}\right)-\frac{\sigma\tau}{\rho_1} (q_x,i \lambda \psi ) -\frac{\sigma \beta}{\rho_1} (q,\psi_x)\notag\\
		&+\frac{\sigma \tau}{\rho_1} (f^6,\psi_x)- ( \Phi,f_x^3)- (f^2,\psi_x)\notag\\
		=& \frac{\kappa}{\rho_1} \left(\omega(\phi_x + \psi) + \smallint_0^\infty g(s) \eta_x(s) \, ds,\psi_{xx}\right)- \frac{\sigma\tau}{\rho_1} (q_x,\Psi+f^3 ) -\frac{\sigma \beta}{\rho_1} (q,\psi_x)\notag\\
		&+\frac{\sigma \tau}{\rho_1} (f^6,\psi_x)- ( \Phi,f_x^3)- (f^2,\psi_x).
	\end{align}
	Taking the $L^2$-inner product of \eqref{2C}-(e) with $\Psi$, we obtain
	\begin{align}\label{5.9}
		(q_x,\Psi)=&-\rho_3(i\lambda\theta,\Psi)-\sigma(\Phi_x,\Psi)-\sigma\|\Psi\|^2+\rho_3(f^5,\Psi)\notag\\
		=&\rho_3(\theta,i\lambda\Psi)-\sigma(\Phi_x,\Psi)-\sigma\|\Psi\|^2+\rho_3(f^5,\Psi)\notag\\
		=&\frac{\rho_3}{\sigma}\left(\sigma\theta-\kappa \left[\omega(\phi_x + \psi) + \smallint_0^\infty g(s) \eta_x(s) \, ds \right]+\kappa \left[\omega(\phi_x + \psi) + \smallint_0^\infty g(s) \eta_x(s) \, ds \right],i\lambda\Psi\right)\notag\\
		&-\sigma(\Phi_x,\Psi)-\sigma\|\Psi\|^2+\rho_3(f^5,\Psi)\notag\\
		=&\frac{\rho_3\kappa}{\sigma}\left(\omega(\phi_x + \psi) + \smallint_0^\infty g(s) \eta_x(s) \, ds ,i\lambda\Psi\right)-\sigma(\Phi_x,\Psi)-\sigma\|\Psi\|^2+\rho_3(f^5,\Psi)\notag\\
		&+\underbrace{\frac{\rho_3}{\sigma}\left(\sigma\theta-\kappa \left[\omega(\phi_x + \psi) + \smallint_0^\infty g(s) \eta_x(s) \, ds \right],i\lambda\Psi\right)}_{I_1}.
	\end{align}
	Using \eqref{2C}-(d), we deduce
	\begin{align}\label{5.10}
		I_1=& -\frac{\rho_3}{\sigma \rho_2} \left(\kappa \left[ \omega(\phi_x + \psi) \!+\! \smallint_0^\infty g(s) \eta_x(s) \, ds \right] - \sigma \theta, b \psi_{xx} \!-\! \kappa \left[ \omega(\phi_x + \psi) \!+\! \smallint_0^\infty g(s) \eta_x(s) \, ds \right] + \sigma \theta + \rho_2 f^4\right)\notag\\
		= &\underbrace{\frac{\rho_3 b}{\sigma \rho_2} \left(\kappa \left[ \omega(\phi_x + \psi) \!+\! \smallint_0^\infty g(s) \eta_x(s) \, ds \right]_x - \sigma \theta_x, \psi_{x}\right)}_{I_2} + \frac{\rho_3}{\sigma \rho_2} \left\|\kappa \left[ \omega(\phi_x + \psi) \!+ \!\smallint_0^\infty g(s) \eta_x(s) \, ds \right] - \sigma \theta\right\|^2\notag\\
		&-\frac{\rho_3}{\sigma} \left(\kappa \left[ \omega(\phi_x + \psi) + \smallint_0^\infty g(s) \eta_x(s) \, ds \right] - \sigma \theta, f^4\right).
	\end{align}
	Combing \eqref{2C}-(b) and (c), we yield
	\begin{align}\label{5.11}
		I_2&=\frac{\rho_3 b}{\sigma \rho_2} (i \lambda \rho_1 \Phi - \rho_1 f^2, \psi_x)= \frac{\rho_1 \rho_3 b}{\sigma \rho_2} (i \lambda \Phi, \psi_x) - \frac{\rho_1 \rho_3 b}{\sigma \rho_2} (f^2, \psi_x)\notag\\
		&= -\frac{\rho_1 \rho_3 b}{\sigma \rho_2} (\Phi, \Psi_x + f^3_x) - \frac{\rho_1 \rho_3 b}{\sigma \rho_2} (f^2, \psi_x).
	\end{align}
	Combining \eqref{5.8}-\eqref{5.11}, we get
	\begin{align}\label{5.12}
		\left(1+\frac{\sigma^2\tau}{\rho_1}-\frac{b\rho_3\tau }{\rho_2} \right)&(\Phi,\Psi_x)=\frac{\kappa}{\rho_1} \left(\omega(\phi_x + \psi) + \smallint_0^\infty g(s) \eta_x(s) \, ds,\psi_{xx}\right)+\frac{\sigma\tau}{\rho_1}(q,f^3_x)-\frac{\sigma \beta}{\rho_1} (q,\psi_x)+\frac{\sigma \tau}{\rho_1} (f^6,\psi_x)\notag\\
		&-\frac{\kappa\rho_3\tau}{\rho_1}\left(\omega(\phi_x + \psi) + \smallint_0^\infty g(s) \eta_x(s) \, ds ,i\lambda\Psi\right)+\frac{\sigma^2\tau}{\rho_1}\|\Psi\|^2-\frac{\sigma\tau\rho_3}{\rho_1}(f^5,\Psi)\notag\\
		&-\frac{\rho_3\tau}{\rho_1 \rho_2} \left\|\kappa \left[ \omega(\phi_x + \psi) + \smallint_0^\infty g(s) \eta_x(s) \, ds \right] - \sigma \theta\right\|^2+\left(\frac{b\rho_3\tau }{\rho_2}-1\right) (\Phi,  f^3_x) \notag\\
		&+\frac{\rho_3\tau}{\rho_1} \left(\kappa \left[ \omega(\phi_x + \psi) + \smallint_0^\infty g(s) \eta_x(s) \, ds \right]- \sigma \theta, f^4\right)+ \left(\frac{b\rho_3\tau }{\rho_2}-1\right)(f^2, \psi_x) .
	\end{align}
	Multiplying the both sides of \eqref{5.4} by $\left(1+\frac{\sigma^2\tau}{\rho_1}-\frac{b\rho_3\tau }{\rho_2} \right)$ and then adding the result to \eqref{5.12}, we deduce
	\begin{align}\label{Psigj1}
		\left(1-\frac{b\rho_3\tau }{\rho_2} \right)\|\Psi\|^2=&j_1+j_2+\underbrace{\left[\frac{\kappa}{\rho_1}-\frac{b}{\rho_2}\left(1+\frac{\sigma^2\tau}{\rho_1}-\frac{b\rho_3\tau }{\rho_2} \right)\right]}_{=\chi_1+\frac{\tau b\rho_3\kappa}{\rho_1\rho_2}}\left(\omega(\phi_x + \psi) + \smallint_0^\infty g(s) \eta_x(s) \, ds,\psi_{xx}\right)\notag\\
		&-\frac{\kappa\rho_3\tau}{\rho_1}\left(\omega(\phi_x + \psi) + \smallint_0^\infty g(s) \eta_x(s) \, ds,i\lambda\Psi \right),
	\end{align}
	where
	\begin{align*}
		j_1 =&\left(1+\frac{\sigma^2\tau}{\rho_1}-\frac{b\rho_3\tau }{\rho_2} \right)\left[\frac{\kappa}{\rho_2} \left\|\omega(\phi_x + \psi) + \smallint_0^\infty g(s) \eta_x(s) \, ds\right\|^2\right.\notag\\
		&- \frac{\sigma}{\rho_2} \left( \omega(\phi_x + \psi) + \smallint_0^\infty g(s) \eta_x(s) \, ds,\theta\right) - \left( \omega(\phi_x + \psi) + \smallint_0^\infty g(s) \eta_x(s) \, ds,f^4\right)\\
		&-\left.\omega( f_x^1+f^3,\Psi)- \smallint_0^\infty g(s) (f_x^7(s),\Psi) \, ds\right]+\frac{\sigma\tau}{\rho_1}(q,f^3_x)-\frac{\sigma \beta}{\rho_1} (q,\psi_x)+\frac{\sigma \tau}{\rho_1} (f^6,\psi_x)-\frac{\sigma\tau\rho_3}{\rho_1}(f^5,\Psi)\\
		&-\frac{\rho_3\tau}{\rho_1 \rho_2} \left\|\kappa \left[ \omega(\phi_x + \psi) + \smallint_0^\infty g(s) \eta_x(s) \, ds \right] - \sigma \theta\right\|^2+\left(\frac{b\rho_3\tau }{\rho_2}-1\right) (\Phi,  f^3_x) \\
		&+\frac{\rho_3\tau}{\rho_1} \left(\kappa \left[ \omega(\phi_x + \psi) + \smallint_0^\infty g(s) \eta_x(s) \, ds \right]- \sigma \theta, f^4\right)+ \left(\frac{b\rho_3\tau }{\rho_2}-1\right)(f^2, \psi_x) ,\\
		j_2 =& \left(1+\frac{\sigma^2\tau}{\rho_1}-\frac{b\rho_3\tau }{\rho_2} \right)\smallint_0^\infty g(s) (\eta_{sx}(s),\Psi)ds.
	\end{align*}
	Using\eqref{2C}-(d) and\eqref{Psigj1}, we have
	\begin{equation}\label{Psigj2}
		\begin{aligned}
			\|\Psi\|^2&=\frac{\rho_2}{\rho_2-b\rho_3\tau}\left[\tilde{j_1}+j_2+\frac{\rho_2\chi_1}{b}\underbrace{\left(\omega(\phi_x + \psi) + \smallint_0^\infty g(s) \eta_x(s) \, ds,i\lambda\Psi \right)}_{j_3}\right],
		\end{aligned}
	\end{equation}
	where \begin{align*}
		\tilde{j_1} =&\frac{\kappa^2}{b\rho_1} \left\|\omega(\phi_x + \psi) + \smallint_0^\infty g(s) \eta_x(s) \, ds \right\|^2- \left(\frac{\sigma\kappa}{b\rho_1}+\frac{\rho_3\tau\sigma\kappa}{\rho_1}\right) \left( \omega(\phi_x + \psi) + \smallint_0^\infty g(s) \eta_x(s) \, ds, \theta \right)  \\
		& + \left(\frac{\rho_3\tau\kappa}{\rho_1}-\frac{\kappa\rho_2}{b\rho_1}\right)\left( \omega(\phi_x + \psi) + \smallint_0^\infty g(s) \eta_x(s) \, ds, f^4 \right)+\frac{\sigma\tau}{\rho_1}(q, f^3_x)-\frac{\sigma \beta}{\rho_1} (q, \psi_x)+\frac{\sigma \tau}{\rho_1} (f^6, \psi_x) \\
		&-\frac{\sigma\tau\rho_3}{\rho_1}(f^5, \Psi)-\left(1+\frac{\sigma^2\tau}{\rho_1}-\frac{b\rho_3\tau }{\rho_2} \right) \left[\omega( f_x^1+f^3, \Psi)+ \smallint_0^\infty g(s) (f_x^7(s), \Psi) \, ds \right] \\
		&-\frac{\rho_3\tau}{\rho_1 \rho_2} \left\|\kappa \left[ \omega(\phi_x + \psi) + \smallint_0^\infty g(s) \eta_x(s) \, ds \right] - \sigma \theta\right\|^2+\left(\frac{b\rho_3\tau }{\rho_2}-1\right) (\Phi, f^3_x)+ \left(\frac{b\rho_3\tau }{\rho_2}-1\right)(f^2, \psi_x).
	\end{align*}
	By using \eqref{b01}, \eqref{t2}, \eqref{t3}, \eqref{t7}, and Cauchy's inequality, we have for any $\epsilon>0$, there exists a constant $C_{\epsilon} > 0$, independent of $\lambda$, such that
	\begin{align*}\label{a01}
		|\tilde{j_1}| &\leq C\|\phi_x + \psi\|^2 + C\| \eta \|_{L^2_g}^2 +C\|\theta\| \left(\omega \|\phi_x + \psi\| + \| \eta \|_{L^2_g}\right) +C\|\theta\|^2+ C \| q \| \| \psi_x \|+ C \| z \|_{\mathcal{H}_C} \| f \|_{\mathcal{H}_C} \\
		&\leq C_\epsilon \left(\|\phi_x + \psi\|^2 + \| \eta \|_{L^2_g}^2 + \| q \|^2\right) + C\| \theta \|^2 + \epsilon \| \psi_x \|^2 + C \| z \|_{\mathcal{H}_C} \| f \|_{\mathcal{H}_C} \\
		&\leq C_\epsilon \left(\frac{1}{I_g(\lambda)} + 1\right)\left(\frac{1}{I_g(\lambda)} + \frac{1}{|\lambda|} + 1\right) \| z \|_{\mathcal{H}_C} \| f \|_{\mathcal{H}_C} + \epsilon \| z \|_{\mathcal{H}_C}^2.
	\end{align*}
	Similar to the prove of \eqref{jxsn}, we deduce
	\begin{equation*}
		j_2 = \left(1+\frac{\sigma^2\tau}{\rho_1}-\frac{b\rho_3\tau }{\rho_2} \right)\smallint_0^\infty g(s) \frac{d}{ds} \left( \eta_x(s), \Psi \right) ds = -\left(1+\frac{\sigma^2\tau}{\rho_1}-\frac{b\rho_3\tau }{\rho_2} \right)\smallint_0^\infty g'(s) \left( \eta_x(s), \Psi \right) ds.
	\end{equation*}
	Then by using Young's inequality, we obtain, for any $\epsilon>0$,
	\begin{align*}
		|j_2|\le&-\left|1+\frac{\sigma^2\tau}{\rho_1}-\frac{b\rho_3\tau }{\rho_2} \right|\smallint_0^\infty g'(s)\|\eta_x(s)\|\|\Psi\|ds\\
		\le&-\frac{\epsilon}{4}\smallint_0^\infty g'(s)\|\Psi\|^2ds-\frac{C}{\epsilon}\smallint_0^\infty g'(s)\|\eta_x\|^2ds\le\frac{g(0)\epsilon}{4}\|\Psi\|^2-\frac{C}{\epsilon}\smallint_0^\infty g'(s)\|\eta_x\|^2ds.
	\end{align*}
	By taking $\epsilon=\frac{|\rho_2-b\rho_3\tau|}{\rho_2g(0)}$, it follows from \eqref{b01} that
	\begin{equation*}
		|j_2|\leq C \| z \|_{\mathcal{H}_C} \| f \|_{\mathcal{H}_C}+\frac{|\rho_2-b\rho_3\tau|}{4\rho_2}\|\Psi\|^2.
	\end{equation*}
	If $\chi_1=0$, we get the estimate \eqref{w5} by estimates $\tilde{j_1}$, $j_2$ and \eqref{Psigj2}. If $\chi_1\neq0$, use H\"older's and Young's inequality, \eqref{t3}, and \eqref{t7}, we get
	\begin{align*}
		|j_3| &\leq C |\lambda| \|\Psi\| \left(\|\phi_x+\psi\|+\|\eta\|_{L_g^2}\right) \\
		&\le C|\lambda|^2\left(\|\phi_x+\psi\|^2+\|\eta\|_{L_g^2}^2\right)+\frac{b|\rho_2-b\rho_3\tau|}{4\rho_2^2|\chi_1|} \|\Psi\|^2\\
		&\leq C|\lambda|^2\left(\frac{1}{I_g(\lambda)}+1\right)\left(\frac{1}{I_g(\lambda)}+\frac{1}{|\lambda|}+1\right)\|z\|_{\mathcal{H}_C} \|f\|_{\mathcal{H}_C}  + \frac{b|\rho_2-b\rho_3\tau|}{4\rho_2^2|\chi_1|} \|\Psi\|^2.
	\end{align*}
	Finally, \eqref{w5} follows from \eqref{Psigj2} and the above estimates for $|j_i|$, $i=1,2,3$.

	\noindent\textbf{Step 6:}  For every $\epsilon>0$, there exists a constant $C_{\epsilon} > 0$, independent of $\lambda$, such that
	\begin{align}\label{w6}
		b \|\psi_x\|^2&\leq C |\chi_1| |\lambda|^2 \left( \frac{1}{I_g(\lambda)} + 1 \right) \left( \frac{1}{I_g(\lambda)} + \frac{1}{|\lambda|} + 1 \right) \| z \|_{\mathcal{H}_C} \| f \|_{\mathcal{H}_C} \notag\\
		&\quad + C_\epsilon \left( \frac{1}{I_g(\lambda)} + 1 \right) \left( \frac{1}{I_g(\lambda)} + \frac{1}{|\lambda|} + 1 \right) \| z \|_{\mathcal{H}_C} \| f \|_{\mathcal{H}_C} + \epsilon\| z \|_{\mathcal{H}_C}^2.
	\end{align}
	
	Taking the $L^2$-inner product of \eqref{2C}-(d) with $\psi$, we have
	\begin{equation}\label{psi01}
		-\rho_2(\Psi, i\lambda\psi) + b\|\psi_x\|^2 + \kappa\omega(\phi_x + \psi, \psi) + \kappa\smallint_{0}^{\infty}g(s)(\eta_x(s), \psi)ds - \sigma(\theta, \psi) = \rho_2(f^4, \psi).
	\end{equation}
	Taking \eqref{2C}-(c) into \eqref{psi01} and using Poincar\'e's, H\"older's and Young's inequalities,  we get
	\begin{align*}
		b\| \psi_x \|^2 &= -\kappa \omega (\phi_x + \psi, \psi) - \kappa \smallint_{0}^{\infty} g(s) (\eta_x(s), \psi)ds + \sigma (\theta, \psi) + \rho_2 (f^4, \psi) + \rho_2 \| \Psi\|^2 + \rho_2 (\Psi, f^3) \\
		&\leq C\|\psi\| ( \| \phi_x + \psi \| +\| \eta \|_{L_g^2}) + \sigma \| \theta \| \| \psi \| + \rho_2 \| f^4 \| \| \psi \| + \rho_2 \| \Psi \|^2 + \rho_2 \| f^3 \| \| \Psi \| \\
		&\leq C \|\psi_x\| (\| \phi_x + \psi \| + \| \eta \|_{L_g^2}) +C\| \theta \| \| \psi_x \| + C \| f^4 \| \| \psi_x \| + \rho_2 \| \Psi \|^2 + \rho_2 \| f^3 \| \| \Psi \| \\
		&\leq C \left( \| \phi_x + \psi \|^2 + \| \eta \|_{L_g^2}^2 \right) + C \| \theta \|^2 + \frac{b}{2} \| \psi_x \|^2 + \rho_2 \| \Psi \|^2 + C \| z \|_{\mathcal{H}_C} \| f \|_{\mathcal{H}_C}.
	\end{align*}
	Then by using \eqref{t2}, \eqref{t3}, \eqref{t7}, \eqref{w5}, and the above estimate, we get \eqref{w6}.
	
	{\bf Finally}, using the estimates got in Steps 1-6, we get for every $\epsilon>0$, there exists a constant $C_{\epsilon} > 0$, independent of $\lambda$, such that
	\begin{align*}
		\|z\|_{\mathcal{H}_C}^2&\leq C |\chi_1| |\lambda|^2 \left( \frac{1}{I_g(\lambda)} + 1 \right) \left( \frac{1}{I_g(\lambda)} + \frac{1}{|\lambda|} + 1 \right) \| z \|_{\mathcal{H}_C} \| f \|_{\mathcal{H}_C} \notag\\
		&\quad + C_\epsilon \left( \frac{1}{I_g(\lambda)} + 1 \right) \left( \frac{1}{I_g(\lambda)} + \frac{1}{|\lambda|} + 1 \right) \| z \|_{\mathcal{H}_C} \| f \|_{\mathcal{H}_C} + \epsilon\| z \|_{\mathcal{H}_C}^2.
	\end{align*}
	By taking $\epsilon=\frac12$, we get
	\begin{align*}
		\|z\|_{\mathcal{H}_C}^2&\leq C |\chi_1| |\lambda|^2\left( \frac{1}{I_g(\lambda)} + 1 \right) \left( \frac{1}{I_g(\lambda)} + \frac{1}{|\lambda|} + 1 \right) \| z \|_{\mathcal{H}_C} \| f \|_{\mathcal{H}_C} \notag\\
		&\quad + C \left( \frac{1}{I_g(\lambda)} + 1 \right) \left( \frac{1}{I_g(\lambda)} + \frac{1}{|\lambda|} + 1 \right) \| z \|_{\mathcal{H}_C} \| f \|_{\mathcal{H}_C}.
	\end{align*}
	Then the conclusion follows.
\end{proof}
\subsection{Spectral Properties on the Imaginary Axis}\label{sec2.2.2}
The main result of this part is the following theorem.
\begin{theorem}\label{thmap}
	Let $g$ satisfy the Assumption \ref{g1} and the ${\delta}$-condition \eqref{delta c}. Then, $i\mathbb{R}\subset\rho(\mathcal{A}_F)$ and $i\mathbb{R}\subset\rho(\mathcal{A}_C)$.
\end{theorem}
\begin{proof}
	\noindent\textbf{Firstly, we show $i\mathbb{R}\subset\rho(\mathcal{A}_F)$.} We recall the definition $ \sigma_{ap}(\mathcal{A}_F)$:
	\begin{align}\label{defapp}
		\sigma_{ap}(\mathcal{A}_F):= \{ \lambda \in \mathbb{C}: \exists z_{n} \in {D}(\mathcal{A}_F), \|z_{n}\|_{\mathcal{H}_F} = 1 \hbox{ and } i \lambda(I - \mathcal{A}_F)z_{n} \to 0 ~\hbox{in} ~\mathcal{H}_F\},
	\end{align}
	where the limit is taken as $n\to\infty$, and throughout the following proof, all limits adhere to the same limiting process.
	
	Suppose by contradiction that $i\lambda_{*}\in\sigma(\mathcal{A}_F)$ for some $\lambda_{*}\in\mathbb{R}$. Then, by Lemma \ref{Thm:SpectralApprox}, we obtain $\lambda_{*}\in\sigma_{\sigma p}(\mathcal{A}_F)$. By \eqref{defapp}, there exists $z_{n}=(\phi_{n},\Phi_{n},\psi_{n},\Psi_{n},\theta_{n},\eta_{n})\in D(\mathcal{A}_F)$ such that $\|z_{n}\|_{\mathcal{H}_F}=1$ for all $n\in\mathbb{N}$ and
	\begin{equation}\label{eq:3.12}
		f_{n}:=i\lambda_{*}z_{n}-\mathcal{A}_Fz_{n}\to 0\text{ in }\mathcal{H}_F.
	\end{equation}
	Setting $f_{n}=(f^{1}_{n},f^{2}_{n},f^{3}_{n},f^{4}_{n},f^{5}_{n},f^{6}_{n})$, we can write \eqref{eq:3.12} in terms of its components:
	\begin{equation}\label{ed}
		\begin{cases}
			\text{(a):~ }f_n^1=i\lambda_{*}\phi_n-\Phi_n\rightarrow0, &\hbox{in  }H_0^1,\\ \text{(b):~}f_n^2=i\lambda_{*}\rho_1\Phi_n-\kappa[\omega(\phi_{n,x}+\psi_n)+\smallint_0^{\infty}g(s)\eta_{n,x}(s)ds]_x+\sigma\theta_{n,x}\rightarrow0&\hbox{in  }L^2,\\
			\text{(c):~}f_n^3=i\lambda_{*}\psi_n-\Psi_n\rightarrow0&\hbox{in  }H_*^1,\\
			\text{(d): ~}f_n^4=i\lambda_{*}\rho_2\Psi_n-b\psi_{n,xx}+\kappa[\omega(\phi_{n,x}+\psi_n)+\smallint_0^{\infty}g(s)\eta_{n,x}(s)ds]-\sigma\theta_n\rightarrow0 \,&\hbox{in   }L_*^2,\\
			\text{(e): ~}f_n^5=i\lambda_{*}\rho_3\theta_n-\beta\theta_{n,xx}+\sigma(\Phi_{n,x}+\Psi_n)\rightarrow0 \,&\hbox{in  }L_*^2,\\
			\text{(f):~ }f_n^6=i\lambda_{*}\eta_n-\eta_{n,s}-(\Phi_n+\tilde{\Psi}_n)\rightarrow0 \,&\hbox{in }L_g^2.
		\end{cases}
	\end{equation}
	We will split the proof into two cases as follows.
	
	\textit{Case 1}: $\lambda_{*}=0$. In this case, we immediately obtain from \eqref{ed}-(a) and  \eqref{ed}-(c) the following convergences:		
	\begin{equation}\label{eq:3.14}
		\Phi_{n}\to 0\text{ in }H^{1}_{0},\quad \Psi_{n}\to 0\text{ in }H^{1}_{*}.
	\end{equation}
	Then by \eqref{ed}-(f), we get
	\begin{align}\label{bccoveta} \|\eta_{n,s}\|_{L^2_g}\le&\|\eta_{n,s}+\Phi_n+\tilde{\Psi}_n\|_{L^2_g}+\|\Phi_n\|_{L^2_g}+\|\tilde{\Psi}_n\|_{L^2_g}\notag\\
		=&\|\eta_{n,s}+\Phi_n+\tilde{\Psi}_n\|_{L^2_g}+\left(\smallint_0^\infty g(s)\|\Phi_{n,x}\|^2ds\right)^{\frac12}+\left(\smallint_0^\infty g(s)\|\Psi_{n}\|^2ds\right)^{\frac12}\notag\\
		=&\|\eta_{n,s}+\Phi_n+\tilde{\Psi}_n\|_{L^2_g}+\left(\smallint_0^\infty g(s)ds\right)^{\frac12}\left(|\Phi_{n,x}\|+\|\Psi_{n}\|\right)\to0.
	\end{align}
	Taking into account that $\eta_n \in D(\mathbb{L})$ and \eqref{bccoveta}, we can apply \eqref{eta_x} to get
	\begin{equation}\label{eq:eta_norm}
		\|\eta_n\|_{L_g^2} \leq \sqrt{\smallint_0^\infty g(s) \left( \smallint_0^s \|\eta_{n,xs}(\tau)\| \, d\tau \right)^2 \, ds} \leq C\|\eta_{ns}\|_{L_g^2} \to 0.
	\end{equation}
	Next, taking account \eqref{eq:3.14} and \eqref{ed}-(e), we have
	\[
	\|\theta_{n,xx}\|\rightarrow0.
	\]
	Then, the Poincar\'e's inequality implies
	\[\|\theta_n\|,\quad\|\theta_{n,x}\|\to0.\]
	Thus, \eqref{ed}-(b) and \eqref{ed}-(d) turns into
	\begin{align}
		&-\kappa\left[\omega(\phi_{n,x}+\psi_n)+\smallint_0^{\infty}g(s)\eta_{n,x}(s)ds\right]_x\rightarrow0\,\hbox{ in  }L^2,\label{b*}\\
		&-b\psi_{n,xx}+\kappa\left[\omega(\phi_{n,x}+\psi_n)+\smallint_0^{\infty}g(s)\eta_{n,x}(s)ds\right]\rightarrow0 \,\hbox{ in   }L_*^2.\label{d*}
	\end{align}
	Taking the $L^2$-inner product of \eqref{b*} with $\phi_n$, and the $L^2$- inner product of \eqref{d*} with $\psi_n$  and adding the results, we deduce
	\begin{equation*}
		\omega \kappa \|\phi_{n,x} + \psi_n\|^2 + b\|\psi_{n,x}\|^2 +\kappa \smallint_0^\infty g(s) (\eta_{n,x}(s), \phi_{nx} + \psi_n) \, ds\to0.
	\end{equation*}
	Since by H\"older's inequality, \eqref{eq:eta_norm}, and $\|\psi_{n,x}+\psi_n\|\le\|z_n\|_{\mathcal{H}_F}\le 1$,
	\begin{align*}
		\left|\smallint_0^\infty g(s) (\eta_{n,x}(s), \phi_{n,x} + \psi_n)\right|&\le\smallint_0^\infty g(s)\|\eta_{n,x}(s)\|ds\|\psi_{n,x}+\psi_n\|\\
		&\le\left(\smallint_0^\infty g(s)ds\right)^{\frac12}\left(\smallint_0^\infty g(s)\|\eta_{n,x}\|^2ds\right)^{\frac12}\\
		&=\left(\smallint_0^\infty g(s)ds\right)^{\frac12}\|\eta_n\|_{L^2_g}\to0,
	\end{align*}
	we get
	\[\omega \kappa \|\phi_{n,x} + \psi_n\|^2 + b\|\psi_{n,x}\|^2 \to0.\]
	
	Hence, the above analysis implies $\|z_{n}\|_{\mathcal{H}_F}\to 0$, which contradicts the equality $\|z_{n}\|_{\mathcal{H}_F}=1$ for all $n\in\mathbb{N}$.
	
	\textit{Case 2}: $\lambda_{*}\neq 0$. Since $\lambda_{*}\in\mathbb{R}\backslash\{0\}$, $f_{n}\in\mathcal{H}_F$ and $z_{n}\in D(\mathcal{A}_F)$ is a solution of \eqref{eq:3.12}, we are in condition to apply Theorem \ref{thmboundresolvent} (for $z_{n}$) to conclude that
	\begin{align*}
		\|z_n\|_{\mathcal{H}_F}&\leq C\left(\frac{1}{I_g(\lambda_*)}+1\right)\left(\frac{1}{I_g(\lambda_*)}+\frac{1}{|\lambda_*|}+1\right)\|f_n\|_{\mathcal{H}_F}\\
		&+C|\chi_0|(1+|\lambda_*|^2)\left(\frac{1}{I_g(\lambda_*)}+1\right)\left(\frac{1}{I_g(\lambda_*)}+\frac{1}{|\lambda_*|}+1\right)\|f_n\|_{\mathcal{H}_F}\\
		&\rightarrow0,
	\end{align*}
	which contradicts again the fact $\|z_{n}\|_{\mathcal{H}_F}=1$ for all $n\in\mathbb{N}$.
\end{proof}
\noindent\textbf{Secondly, we show $i\mathbb{R}\subset\rho(\mathcal{A}_C)$.}
Suppose by contradiction that $i\lambda_{*}\in\sigma(\mathcal{A}_C)$ for some $\lambda_{*}\in\mathbb{R}$. Then, by Lemma \ref{Thm:SpectralApprox}, we obtain $\lambda_{*}\in\sigma_{\sigma p}(\mathcal{A}_C)$, and then there exists $z_{n}=(\phi_{n},\Phi_{n},\psi_{n},\Psi_{n},\theta_{n},q_{n},\eta_{n})\in D(\mathcal{A}_C)$ such that $\|z_{n}\|_{\mathcal{H}_C}=1$ for all $n\in\mathbb{N}$ and

\begin{equation}\label{8_1}
	f_{n}:=i\lambda_{*}z_{n}-\mathcal{A}_Cz_{n}\to 0\text{ in }\mathcal{H}_C.
\end{equation}

Setting $f_{n}=(f^{1}_{n},f^{2}_{n},f^{3}_{n},f^{4}_{n},f^{5}_{n},f^{6}_{n},f^{7}_{n})$, we can write \eqref{8_1} in terms of its components:
\begin{equation}\label{de}
	\begin{cases}
		\text{(a):~ }f_n^1=i\lambda_{*}\phi_n-\Phi_n\rightarrow0 &\hbox{ in  }H_0^1,\\
		\text{(b):~ }f_n^2=i\lambda_{*}\rho_1\Phi_n-\kappa[\omega(\phi_{n,x}+\psi_n)+\smallint_0^{\infty}g(s)\eta_{n,x}(s)ds]_x+\sigma\theta_{n,x}\rightarrow0 &\hbox{ in  }L^2,\\
		\text{(c):~ }f_n^3=i\lambda_{*}\psi_n-\Psi_n\rightarrow0 &\hbox{ in  }H_*^1, \\
		\text{(d):~ }f_n^4=i\lambda_{*}\rho_2\Psi_n-b\psi_{n,xx}+\kappa[\omega(\phi_{n,x}+\psi_n)+\smallint_0^{\infty}g(s)\eta_{n,x}(s)ds]-\sigma\theta_n\rightarrow0&\hbox{ in   }L_*^2, \\
		\text{(e):~ }f_n^5=i\lambda_{*}\theta_n+q_{n,x}+\sigma(\Phi_{n,x}+\Psi_n)\rightarrow0&\hbox{ in  }L^2_*, \\
		\text{(f):~ }f_n^6=i\lambda_{*}\tau q_n+\beta q_n+\theta_{n,x}\rightarrow0 &\hbox{ in  }L^2,\\
		\text{(g):~ }f_n^7=i\lambda_{*}\eta_n-\eta_{n,s}-(\Phi_n+\tilde{\Psi}_n)\rightarrow0&\hbox{ in }L_g^2.
	\end{cases}
\end{equation}
We will split the proof into two cases as follows.

\textit{Case 1}: $\lambda_{*}=0$. In this case, we immediately obtain from \eqref{de}-(a), \eqref{de}-(c), and similar argument as in \eqref{eq:eta_norm} the following convergences:
\begin{equation}\label{8_2}
	\Phi_{n}\to 0\text{ in }H^{1}_{0},\quad \Psi_{n}\to 0\text{ in }H^{1}_{*},\quad \eta_{n}\to 0\text{ in }L_g^2.
\end{equation}
Next, using \eqref{8_2} and \eqref{de}-(e), we yield $\|q_{n,x}\|\rightarrow0$. Then by Poincar\'e's inequality we get
\begin{equation}\label{8_5}
	\|q_{n}\|\rightarrow0.
\end{equation}
Taking \eqref{8_5} into \eqref{de}-(f), and by Poincar\'e's in equality, we get
\(
\|\theta_n\|,~	\|\theta_{n,x}\|\rightarrow0.
\) Then by similar argument as Case 1 in the proof of $i\lambda_{*}\in\sigma(\mathcal{A}_F)$ we get $\|z_{n}\|_{\mathcal{H}_C}\to 0$, which contradicts the equality $\|z_{n}\|_{\mathcal{H}_C}=1$ for all $n\in\mathbb{N}$.

\textit{Case 2}: $\lambda_{*}\neq 0$. Since $\lambda_{*}\in\mathbb{R}\backslash\{0\}$, $f_{n}\in\mathcal{H}_C$ and $z_{n}\in D(\mathcal{A}_C)$ is a solution of \eqref{8_1}. we are in condition to apply Theorem \ref{thmboundresolvent} (for $z_{n}$) to conclude that
\begin{align*}
	\|z_n\|_{\mathcal{H}_C}&\leq C\left(\frac{1}{I_g(\lambda_*)}+1\right)\left(\frac{1}{I_g(\lambda_*)}+\frac{1}{|\lambda_*|}+1\right)\|f_n\|_{\mathcal{H}_C}\\
	&+C|\chi_0|(1+|\lambda_*|^2)\left(\frac{1}{I_g(\lambda_*)}+1\right)\left(\frac{1}{I_g(\lambda_*)}+\frac{1}{|\lambda_*|}+1\right)\|f_n\|_{\mathcal{H}_C}+C|\lambda_*|^2\|f_n\|_{\mathcal{H}_C}\\
	&\rightarrow0,
\end{align*}
which contradicts again the fact $\|z_{n}\|_{\mathcal{H}_C}=1$ for all $n\in\mathbb{N}$.
\subsection{Sharp Resolvent Estimates}\label{sec2.2.3}
The main result of this part is the following  two theorems:
\begin{theorem}\label{x1}
	Let $g$ satisfy the Assumption \ref{g1} and the ${\delta}$-condition \eqref{delta c}. If
	\begin{equation}\label{chioneq}
		\chi_0\neq0,
	\end{equation}
	then there exist two sequences
	\[
	\lambda_n=\frac{\sqrt b\pi^2n^2}{L\sqrt{\rho_2\pi^2n^2+\rho_1L^2}},\quad B_n=\frac{\left(\lambda_n^2\rho_1-\kappa(1-\hat{g}(\lambda_n))\gamma_n\right)(i\lambda_n\rho_3+\beta\gamma_n)-i\lambda_n\sigma^2\gamma_n}{\det\mathbf{M}_n(\lambda_n)},~~~n=1,2,\cdots
	\]
	with $\hat g(\cdot)$ given in Lemma \ref{hatg},
	$$\gamma_n=\left(\frac{n\pi}{L}\right)^2,\quad\det\mathbf{M}_n(\lambda_n)=-\frac{b\beta\rho_1^2\gamma_n^2\lambda_n^2}{\rho_2\gamma_n+\rho_1}-i\frac{\rho_1^2\rho_3\gamma_n\lambda_n^3}{\rho_2\gamma_n+\rho_1}, $$
	such that
	\[\lim_{n\to\infty}\frac{\lambda_n}{n}=\frac{\sqrt b\pi}{L\sqrt{\rho_2}}, \quad \lim_{n\to\infty}|B_n|=\frac{\rho_2^2|\chi_0|}{\rho_1b^2}.\]
	Moreover, for
	$$f_n:=\left(0,0,0,-\rho_2^{-1}\sqrt{\frac2 L}\cos\left(\frac{n\pi x}{L}\right),0,0\right)\in\mathcal{H}_F,$$
	there holds
	$$
	\left\|(i\lambda_n-\mathcal{A}_F)^{-1}f_n\right\|_{\mathcal{H}_F}\geq\sqrt{\rho_2}|B_n||\lambda_n|,
	$$
	where $\chi_0$ is the constant defined in \eqref{chi0}.
\end{theorem}
\begin{theorem}\label{x2}
	
	Let $g$ satisfy the Assumption \ref{g1} and the ${\delta}$-condition \eqref{delta c}. If
	\begin{equation}\label{tjnex}
		\chi_1\neq0,
	\end{equation}
	then there exist two sequences
	\[
	\lambda_n=\frac{\sqrt b\pi^2n^2}{L\sqrt{\rho_2\pi^2n^2+\rho_1L^2}},\quad B_n=\frac{\left(\lambda_n^2\rho_1-\kappa(1-\hat{g}(\lambda_n))\gamma_n\right)(i\lambda_n\rho_3+\frac{\gamma_n}{i \lambda_n \tau + \beta})-i\lambda_n\sigma^2\gamma_n}{\det\mathbf{U}_n(\lambda_n)},~~~n=1,2,\cdots
	\]
	with $\hat g(\cdot)$ given in Lemma \ref{hatg},
	$$\gamma_n=\left(\frac{n\pi}{L}\right)^2,\quad\det\mathbf{U}_n(\lambda_n)=\lambda_n^2\rho_1(\lambda_n^2\rho_2-b\gamma_n)\left(\frac{\beta\gamma_n}{\beta^2+\lambda^2\tau^2}\right)+i\lambda_n^2\rho_1(\lambda_n^2\rho_2-b\gamma_n)\left(\lambda\rho_3-\frac{\lambda\gamma_n\tau}{\beta^2+\lambda^2\tau^2}\right), $$
	such that
	\[\lim_{n\to\infty}\frac{\lambda_n}{n}=\frac{\sqrt b\pi}{L\sqrt{\rho_2}}, \quad \lim_{n\to\infty}|B_n|=\left\{
	\begin{array}{ll}
		\displaystyle \frac{\rho_2^3|\chi_1|}{\rho_1b^2|\tau b\rho_3-\rho_2|} & \hbox{ if } \tau b\rho_3\neq\rho_2,\\
		\displaystyle \infty, & \hbox{ if }\tau b\rho_3=\rho_2.
	\end{array}
	\right..\]
	Moreover, for
	$$f_n:=\left(0,0,0,-\rho_2^{-1}\sqrt{\frac2 L}\cos\left(\frac{n\pi x}{L}\right),0,0,0\right)\in\mathcal{H}_C,$$
	there holds
	\begin{equation*}
		\|(i\lambda_n-\mathcal{A}_C)^{-1}f_n\|_{\mathcal{H}_C}\geq\sqrt{\rho_2}|B_n||\lambda_n|.
	\end{equation*}
	where $\chi_1$ is the constant defined in \eqref{chi1}.
\end{theorem}
\begin{proof}[Proof of Theorem \ref{x1}]
	First, let us introduce the following notation
	$$\quad e_n(x)=\sqrt{\frac2 L}\sin(\sqrt{\gamma_n}x),\quad e_n^*(x)=\sqrt{\frac2 L}\cos(\sqrt{\gamma_n}x).$$
	Then $f_n=(0,0,0,-\rho_2^{-1}e_n^*,0)\in\mathcal{H}_F$. Since $i\mathbb{R}\subset\rho(\mathcal{A}_F)$, there exists $z_n=(\phi_n,\Phi_n,\psi_n,\Psi_n,\theta_n,\eta_n)\in D(\mathcal{A}_F)$ such that $ i\lambda_nz_n-\mathcal{A}_Fz_n=f_n.$ Componentwise, we have
	\begin{equation}
		\begin{cases}\label{d0}
			\text{(a):~}i\lambda_n\phi_n-\Phi_n=0,\\ \text{(b):~}i\lambda_n\rho_1\Phi_n-\kappa[\omega(\phi_{n,x}+\psi_n)+\smallint_0^{\infty}g(s)\eta_{n,x}(s)ds]_x+\sigma\theta_{n,x}=0,\\
			\text{(c):~}i\lambda_n\psi_n-\Psi_n=0,\\				\text{(d):~}i\lambda_n\rho_2\Psi_n-b\psi_{n,xx}+\kappa[\omega(\phi_{n,x}+\psi_n)+\smallint_0^{\infty}g(s)\eta_{n,x}(s)ds]-\sigma\theta_n=-e_n^*,\\
			\text{(e):~}i\lambda_n\rho_3\theta_n-\beta\theta_{n,xx}+\sigma(\Phi_{n,x}+\Psi_n)=0,\\
			\text{(f):~}i\lambda_n\eta_n+\eta_{n,s}-(\Phi_n+\tilde{\Psi}_n)=0.
		\end{cases}
	\end{equation}
	Solving the differential equation \eqref{d0}-(f) and using \eqref{d0}-(a,c) in the result, we have
	\begin{equation}\label{d1}
		\eta_n(s)=(1-e^{-i\lambda_n s})(\phi_n+\tilde{\psi}_n).
	\end{equation}
	Now, using \eqref{d0}-(a,c) and \eqref{d1} in \eqref{d0}-(b,d,e), we arrive at
	\begin{equation}
		\left\{\begin{split}\label{d2} &\lambda_n^2\rho_1\phi_n+\kappa(1-\hat{g}(\lambda_n))(\phi_{n,x}+\psi_n)_x-\sigma\theta_{n,x}=0,\\
			&\lambda_n^2\rho_2\psi_n+b\psi_{n,xx}-\kappa(1-\hat{g}(\lambda_n))(\phi_{n,x}+\psi_n)+\sigma\theta_n=e_n^*,\\
			&i\lambda_n\rho_3\theta_n-\beta\theta_{n,xx}+i\lambda_n\sigma(\phi_{n,x}+\psi_n)=0,
		\end{split}
		\right.
	\end{equation}
	where $\hat g(\cdot)$ is defined in Lemma \ref{hatg}, and we have used $\omega+\smallint_0^\infty g(s)ds=\omega+\ell=1$.

	We are looking for solutions of \eqref{d2} of the form,
	$$\phi_n=A_ne_n,~~~\psi_n=B_ne^*_n,~~~\theta_n=C_ne^*_n,$$
	for some complex sequences $A_n,~B_n,~C_n$, $n=1,2,\cdots$. Replacing these particular choices in \eqref{d2}, we obtain the following complex linear system,
	\begin{equation}
		\left\{\begin{split}\label{d3}
			&\lambda_n^2\rho_1A_n-\kappa(1-\hat{g}(\lambda_n))(\gamma_nA_n+\sqrt{\gamma_n}B_n)+\sigma\sqrt{\gamma_n}C_n=0,\\
			&\lambda_n^2\rho_2B_n-b\gamma_nB_n-\kappa(1-\hat{g}(\lambda_n))(\sqrt{\gamma_n}A_n+B_n)+\sigma C_n=1,\\
			&i\lambda_n\rho_3C_n+\beta\gamma_nC_n+i\lambda\sigma(\sqrt{\gamma_n}A_n+B_n)=0,
		\end{split}
		\right.
	\end{equation}
	which can be written as
	$$\mathbf{M}_n(\lambda_n)\left(
	\begin{array}{c}
		A_n \\
		B_n \\
		C_n \\
	\end{array}
	\right)=\left(
	\begin{array}{c}
		0 \\
		1 \\
		0 \\
	\end{array}
	\right),
	$$
	where
	\begin{equation*}
		\mathbf{M}_n(\lambda):=\begin{pmatrix}
			\lambda^2\rho_1-\kappa(1-\hat{g}(\lambda))\gamma_n& -\kappa(1-\hat{g}(\lambda))\sqrt{\gamma_n} & \sigma\sqrt{\gamma_n}\\
			-\kappa(1-\hat{g}(\lambda))\sqrt{\gamma_n} & \lambda^2\rho_2-b\gamma_n-\kappa(1-\hat{g}(\lambda)) & \sigma\\
			i\lambda\sigma\sqrt{\gamma_n} & i\lambda\sigma & i\lambda\rho_3+\beta\gamma_n
		\end{pmatrix}.
	\end{equation*}
	A direct calculation shows that
	$$\det \mathbf{M}_n(\lambda)=Q_1(\lambda)+Q_2(\lambda),$$
	where
	\begin{align*}
		Q_1(\lambda):=&\left[\lambda^2\rho_1(\lambda^2\rho_2-b\gamma_n)-\kappa D(\lambda)\right]\beta\gamma_n,\\ Q_2(\lambda):=&i\rho_1\rho_3\lambda^3(\rho_2\lambda^2-b\gamma_n)-i\kappa\rho_3\lambda D(\lambda)+\kappa\hat{g}(\lambda)D(\lambda)(i\lambda\rho_3+\beta\gamma_n)-i\sigma^2\lambda D(\lambda),
	\end{align*}
	and
	\[D(\lambda):=\lambda^2\rho_2\gamma_n-b\gamma_n^2+\lambda^2\rho_1.\]
	
	In the following we choose a suitable sequence $\lambda_n$ such that $\det \mathbf{M}_n(\lambda_n)\neq0$.	Actually, we will pick a sequence satisfying $D(\lambda_n)=0$, i.e.,
	\begin{equation}\label{d5}
		\lambda_n^2=\frac{b\gamma_n^2}{\rho_2\gamma_n+\rho_1}.
	\end{equation}
	Then,
	\begin{align*}
		Q_1(\lambda_n)&=\left[\lambda_n^2\rho_1(\lambda_n^2\rho_2-b\gamma_n)-\kappa D(\lambda_n)\right]\beta\gamma_n=-\frac{b\beta\rho_1^2\gamma_n^2\lambda_n^2}{\rho_2\gamma_n+\rho_1},\\ Q_2(\lambda_n)&=i\rho_1\rho_3\lambda_n^3(\rho_2\lambda_n^2-b\gamma_n)-i\kappa\rho_3\lambda_nD(\lambda_n)+\kappa\hat{g}(\lambda_n)D(\lambda_n)(i\lambda_n\rho_3+\beta\gamma_n)-i\sigma^2\lambda_n D(\lambda_n)=-i\frac{\rho_1^2\rho_3\gamma_n\lambda_n^3}{\rho_2\gamma_n+\rho_1}.
	\end{align*}
	Therefore,
	\begin{align}\label{detMz} |\det\mathbf{M}_n(\lambda_n)|=\sqrt{\left(\frac{b\beta\rho_1^2\gamma_n^2\lambda_n^2}{\rho_2\gamma_n+\rho_1}\right)^2+\left(\frac{\rho_1^2\rho_3\gamma_n\lambda_n^3}{\rho_2\gamma_n+\rho_1}\right)^2}=\frac{\rho_1^2\gamma_n\lambda_n^2}{\rho_2\gamma_n+\rho_1}\sqrt{b^2\beta^2\gamma_n^2+\rho_3^2\lambda_n^2}>0.
	\end{align}
	Returning \eqref{d5} to system \eqref{d3} and solving it, we get
	\[B_n=\frac{\left(\lambda_n^2\rho_1-\kappa(1-\hat{g}(\lambda_n))\gamma_n\right)(i\lambda_n\rho_3+\beta\gamma_n)-i\lambda_n\sigma^2\gamma_n}{\det\mathbf{M}_n(\lambda_n)}.\]
	Then, by \eqref{d5}, \eqref{detMz}, Lemma \ref{hatg}, and \eqref{chi0}, we get
	\begin{align*}
		\lim_{n\to\infty}|B_n|=&\lim_{n\to\infty}\frac{|\left(\lambda_n^2\rho_1-\kappa(1-\hat{g}(\lambda_n))\gamma_n\right)(i\lambda_n\rho_3+\beta\gamma_n)-i\lambda_n\sigma^2\gamma_n|/\gamma_n^2}{|\det\mathbf{M}_n(\lambda_n)|/\gamma_n^2}\\
		=&\lim_{n\to\infty}\frac{\left|\left(\rho_1\frac{\lambda_n^2}{\gamma_n}-\kappa(1-\hat{g}(\lambda_n))\right)\left(i\frac{\lambda_n}{\gamma_n}\rho_3+\beta\right)-i\frac{\lambda_n}{\gamma_n}\sigma^2\right|}{\frac{\rho_1^2\gamma_n}{\rho_2\gamma_n+\rho_1}\frac{\lambda_n^2}{\gamma_n}\sqrt{b^2\beta^2+\rho_3^2\frac{\lambda_n^2}{\gamma_n^2}}}=\frac{\rho_2^2|\chi_0|}{\rho_1b^2}.
	\end{align*}
	
	Additionally, from \eqref{d0}-(c), we have
	$$\Psi_n=i\lambda_n\psi_n=i\lambda_nB_n\sqrt{\frac2 L}\cos(\sqrt{\gamma_n}x).$$
	Therefore,
	\[\|z\|_{\mathcal{H}_F}^2\geq\rho_2\|\Psi_n\|^2=\frac{2}{L}\rho_2|B_n|^2|\lambda_n|^2\smallint_0^L\cos^2(\sqrt{\gamma_n}x)dx=\rho_2|B_n|^2|\lambda_n|^2.
	\]
	Then the desired conclusion follows.
\end{proof}

\begin{proof}[Proof of Theorem \ref{x2}]
	Let $e_n$ and $e_n^*$ be the same functions as  in the proof of Theorem \ref{x1},	we consider the sequences $f_n=(0,0,0,-\rho_2^{-1}e_n^*,0,0,0)\in\mathcal{H}_C$. Since $i\mathbb{R}\subset\rho(\mathcal{A}_C)$, there exists $z_n=(\phi_n,\Phi_n,\psi_n,\Psi_n,\theta_n,q_n,\eta_n)\in D(\mathcal{A}_C)$ such that $ i\lambda_nz_n-\mathcal{A}_cz_n=f_n.$
	Componentwise, we have
	\begin{equation}\label{op}
		\begin{cases}
			&\text{(a):~}i \lambda_n \phi_n - \Phi_n = 0, \\
			&\text{(b):~}i \lambda_n \rho_1 \Phi_n - \kappa \left[ \omega(\phi_{n,x} + \psi_n) + \smallint_0^\infty g(s) \eta_{n,x}(s) \, ds \right]_x + \sigma \theta_{n,x} = 0,\\
			&\text{(c):~}i \lambda_n \psi_n - \Psi_n = 0, \\
			&\text{(d):~}i \lambda_n \rho_2 \Psi_n - b \psi_{n,xx} + \kappa \left[ \omega(\Phi_{n,x} + \Psi_n) + \smallint_0^\infty g(s) \eta_{n,x}(s) \, ds \right] - \sigma \theta_n = -e_n^*,\\
			&\text{(e):~}i \lambda_n \rho_3 \theta_n + q_{n,x} + \sigma (\Phi_{n,x} + \Psi_n) = 0,\\
			&\text{(f):~}i \lambda_n \tau q_n + \beta q_n + \theta_{n,x} = 0, \\
			&\text{(g):~}i \lambda_n \eta_n + \eta_{n,s} - (\Phi_n + \tilde{\Psi}_n) = 0.
		\end{cases}
	\end{equation}
	Solving the differential equation \eqref{op}-(g) and using \eqref{op}-(a,c) in the result, we have
	\begin{equation}\label{r1}
		\eta(s)=(1-e^{-i\lambda_n s})(\phi_n+\tilde{\psi}_n).
	\end{equation}
	Now, using \eqref{op}-(a,c) and \eqref{r1} in \eqref{op}-(b,d,e,f), we arrive at
	\begin{equation}
		\left\{\begin{split}\label{r2}
			&\lambda_n^2 \rho_1 \phi_n + \kappa (1 - \hat{g}(\lambda_n)) (\phi_{n,x} + \psi_n)_x - \sigma \theta_{n,x} = 0, \\
			&\lambda_n^2 \rho_2 \psi_n + b \psi_{n,xx} - \kappa (1 - \hat{g}(\lambda_n)) (\phi_{n,x} + \psi_n) + \sigma \theta_n= e_n^*, \\
			&i \lambda_n \rho_3 \theta_n + q_{n,x} + i \lambda_n \sigma (\phi_{n,x} + \psi_n) = 0 ,\\
			&i \lambda_n \tau q_n + \beta q_n + \theta_{n,x} = 0,
		\end{split}
		\right.
	\end{equation}
	where $\hat g(\cdot)$ is defined in Lemma \ref{hatg}, and we have used $\omega+\smallint_0^\infty g(s)ds=\omega+\ell=1$.
	
	We are looking for solutions of \eqref{r2} of the form,
	$$\phi_n=A_ne_n,~~~\psi_n=B_ne^*_n,~~~\theta_n=C_ne^*_n,~~~q=D_ne_n$$
	for some complex sequences $A_n,~B_n,~C_n,~D_n$. Replacing these particular choices in \eqref{r2}, we obtain the following complex linear system,
	\begin{equation}
		\left\{\begin{split}\label{r3}
			&\lambda_n^2 \rho_1 A_n - \kappa(1 - \hat{g}(\lambda_n))(\gamma_n A_n + \sqrt{\gamma_n} B_n) + \sigma \sqrt{\gamma_n} C_n = 0, \\
			&\lambda_n^2 \rho_2 B_n - b \gamma_n B_n - \kappa(1 - \hat{g}(\lambda_n))(\sqrt{\gamma_n} A_n + B_n) + \sigma C_n = 1, \\
			&i \lambda_n \rho_3 C_n + \sqrt{\gamma_n} D_n + i \lambda_n \sigma (\sqrt{\gamma_n} A_n + B_n) = 0, \\
			&i \lambda_n \tau D_n + \beta D_n - \sqrt{\gamma_n} C_n = 0.
		\end{split}
		\right.
	\end{equation}
	Using $\eqref{r3}_4$, we get $D_n=\frac{\sqrt{\gamma_n}C_n}{i \lambda_n \tau + \beta}$. Thus, $\eqref{r3}$ becomes
	\begin{equation*}
		\left\{\begin{split}
			&\lambda_n^2 \rho_1 A_n - \kappa(1 - \hat{g}(\lambda_n))(\gamma_n A_n + \sqrt{\gamma_n} B_n) + \sigma \sqrt{\gamma_n} C_n = 0, \\
			&\lambda_n^2 \rho_2 B_n - b \gamma_n B_n - \kappa(1 - \hat{g}(\lambda_n))(\sqrt{\gamma_n} A_n + B_n) + \sigma C_n = 1, \\
			& \left(i\lambda_n \rho_3+\frac{\gamma_n}{i \lambda_n \tau + \beta}\right) C_n + i \lambda_n \sigma (\sqrt{\gamma_n} A_n + B_n) = 0,
		\end{split}
		\right.
	\end{equation*}
	which can be written as
	$$\mathbf{U}_n(\lambda_n)\left(
	\begin{array}{c}
		A_n \\
		B_n \\
		C_n \\
	\end{array}
	\right)=\left(
	\begin{array}{c}
		0 \\
		1 \\
		0 \\
	\end{array}
	\right),
	$$
	where
	\begin{equation*}
		\mathbf{U}_n(\lambda):=\begin{pmatrix}
			\lambda^2 \rho_1 - \kappa (1 - \hat{g}(\lambda)) \gamma_n & -\kappa(1-\hat{g}(\lambda))\sqrt{\gamma_n} & \sigma\sqrt{\gamma_n}\\
			-\kappa(1-\hat{g}(\lambda))\sqrt{\gamma_n} &  \lambda^2 \rho_2 - b \gamma_n - \kappa(1 - \hat{g}(\lambda)) & \sigma\\
			i\lambda\sigma\sqrt{\gamma_n} & i\lambda\sigma &  i \lambda \rho_3 + \frac{\gamma_n}{i \lambda \tau + \beta}
		\end{pmatrix}.
	\end{equation*}
	A direct calculation shows that
	$$\det\mathbf{U}_n(\lambda)=U_1(\lambda)+U_2(\lambda),$$
	where,
	\begin{align*}
		U_1(\lambda)&=[\lambda^2\rho_1(\lambda^2\rho_2-b\gamma_n)-\kappa D(\lambda)]\operatorname{Re}Q(\lambda),\\ U_2(\lambda)&=\kappa\hat{g}(\lambda)D(\lambda)Q(\lambda)+i[\lambda^2\rho_1(\lambda^2\rho_2-b\gamma_n)-\kappa D(\lambda)]\operatorname{Im}Q(\lambda)-i\lambda\sigma^2 D(\lambda),
	\end{align*}
	and
	$$D(\lambda):=\lambda^2\rho_2\gamma_n-b\gamma_n^2+\lambda^2\rho_1,~~~Q(\lambda):=i\lambda_n \rho_3+\frac{\gamma_n}{i \lambda_n \tau + \beta}=\frac{\beta\gamma_n}{\beta^2+\lambda^2\tau^2}+i\left(\lambda\rho_3-\frac{\lambda\gamma_n\tau}{\beta^2+\lambda^2\tau^2}\right).$$
	
	In the following we choose a suitable sequence $\lambda_n$ such that $\det\mathbf{U}_n\neq0$. Actually, we will pick a sequence $\lambda_n$ satisfying $D(\lambda_n)=0$. Therefore,
	\begin{equation}\label{r5}
		\lambda_n^2=\frac{b\gamma_n^2}{\rho_2\gamma_n+\rho_1}.
	\end{equation}
	Then,
	\begin{align*} U_1(\lambda_n)=\lambda_n^2\rho_1(\lambda_n^2\rho_2-b\gamma_n)\operatorname{Re}Q(\lambda_n),~~~ U_2(\lambda_n)=i\lambda_n^2\rho_1(\lambda_n^2\rho_2-b\gamma_n)\operatorname{Im}Q(\lambda_n),
	\end{align*}
	and
	$$|\det\mathbf{U}_n(\lambda_n)|=|\lambda_n^2\rho_1(\lambda_n^2\rho_2-b\gamma_n)||Q(\lambda_n)|=\frac{\rho_1^2b^2\gamma_n^3}{(\rho_2\gamma_n+\rho_1)^2}|Q(\lambda_n)|\neq0.$$
	Then, by \eqref{r5}, Lemma \ref{hatg}, \eqref{chi0}, and \eqref{tjnex}, we get
	\begin{align*} \lim_{n\to\infty}|B_n|=&\lim_{n\to\infty}\frac{|\left(\lambda_n^2\rho_1-\kappa(1-\hat{g}(\lambda_n))\gamma_n\right)(i\lambda_n\rho_3+\frac{\gamma_n}{i \lambda_n \tau + \beta})-i\lambda_n\sigma^2\gamma_n|/\gamma_n\lambda_n}{|\det\mathbf{U}_n(\lambda_n)|/\gamma_n\lambda_n}\nonumber\\
		=&\lim_{n\to\infty}\frac{\left|\left(\rho_1\frac{\lambda_n^2}{\gamma_n}-\kappa(1-\hat{g}(\lambda_n))\right)\left(i\rho_3+\frac{\gamma_n}{i \lambda_n^2 \tau + \beta\lambda_n}\right)-i\sigma^2\right|}{\frac{\rho_1^2b^2\gamma_n^2}{(\rho_2\gamma_n+\rho_1)^2}|i\rho_3+\frac{\gamma_n}{i \lambda_n^2 \tau + \beta\lambda_n}|}\nonumber\\=&\left\{
		\begin{array}{ll}
			\displaystyle \frac{|\chi_0\rho_1(\tau b\rho_3-\rho_2)+\tau b\sigma^2|\rho_2^2}{\rho_1^2b^2|\tau b\rho_3-\rho_2|} & \hbox{ if } \tau b\rho_3\neq\rho_2,\\
			\displaystyle \infty, & \hbox{ if }\tau b\rho_3=\rho_2,
		\end{array}
		\right.\notag\\
		=&\left\{
		\begin{array}{ll}
			\displaystyle \frac{\rho_2^3|\chi_1|}{\rho_1b^2|\tau b\rho_3-\rho_2|} & \hbox{ if } \tau b\rho_3\neq\rho_2,\\
			\displaystyle \infty, & \hbox{ if }\tau b\rho_3=\rho_2,
		\end{array}
		\right.
	\end{align*}
	where we have used
	$$\chi_0\rho_1(\tau b\rho_3-\rho_2)+\tau b\sigma^2=-\rho_1\rho_2\chi_1.$$	
	
	Additionally, from \eqref{op}-(c), we have
	$$\Psi_n=i\lambda_n\psi_n=i\lambda_nB_n\sqrt{\frac2 L}\cos(\sqrt{\gamma_n}x).$$
	Therefore,
	\begin{equation*}
		\|z\|_{\mathcal{H}_C}^2\geq\rho_2\|\Psi_n\|^2=\frac{2}{L}\rho_2|B_n|^2|\lambda_n|^2\smallint_0^L\cos^2(\sqrt{\gamma_n}x)dx=\rho_2|B_n|^2|\lambda_n|^2.
	\end{equation*}
	Then the desired conclusion follows.
\end{proof}
\subsection{Proofs of the Main Stability Theorems}\label{sec2.2.4}
In order to prove Theorems \ref{thmFourier} and \ref{thmCattaneo}, we need the following theorem, which can be found in \cite{MR2606945,MR743749,MR461206}.
\begin{theorem}\label{fz1}
	Let $\mathcal{A}:D(\mathcal{A})\subset\mathcal{H}\to\mathcal{H}$ be a infinitesimal generator of a bounded $C_0$-semigroup $S(t)=e^{\mathcal{A}t}$, $t\ge0$,  on a Hilbert space $\mathcal{H}$. Assume that $i\mathbb{R}\subset \rho(\mathcal{A})$. Then the following holds:
	\begin{enumerate}
		\item $S(t)$ is polynomially stable of order $\frac{1}{\beta}$ for some constant $\beta>0$ if and only if
		\[\limsup_{\lambda\in\mathbb{R},~|\lambda|\to\infty}|\lambda|^{-\beta}\left\|(i\lambda-\mathcal{A})^{-1}\right\|_{\mathcal{L}(\mathcal{H})}<\infty.\] \item $S(t)$ is exponentially stable if and only if
		\begin{equation}\label{fz2_0}
			\limsup_{\lambda\in\mathbb{R},~|\lambda|\to\infty}\left\|(i\lambda-\mathcal{A})^{-1}\right\|_{\mathcal{L}(\mathcal{H})}<\infty.
		\end{equation}
	\end{enumerate}
\end{theorem}

\begin{proof}[Proof of the necessity Theorem \ref{thmFourier}-1]
	Suppose that $e^{\mathcal{A}_Ft}$ is exponentially stable. We are going to prove that
	\begin{itemize}
		\item[(1)] $g$ satisfies $\delta$- condition \eqref{delta c};
		\item[(2)] $\chi_0 = 0$.
	\end{itemize}

	\textit{Proof of (1):} We use the ideas in \cite{MR2215885} for the following proof.
	Let
	\[
	z(t) = e^{\mathcal{A}_Ft}(0,0,0,0,0,\eta_0) = (\phi(t), \Phi(t), \psi(t), \Psi(t), \theta(t),\eta^t)
	\]
	be the solution of \eqref{cauchyfc} with initial value $(0,0,0,0,0,\eta_0) $ for some $\eta_0 \in L^2_g$.
	
	Since $ e^{\mathcal{A}_Ft}$ is exponentially stable, there exists two constants $M \geq 1$ and $\gamma > 0$ such that
	\begin{equation}\label{i}
		\|z(t)\|_{\mathcal{H}_F}^2  \leq M e^{-\gamma t} \|\eta_0\|_{L^2_g}^2
	\end{equation}
	for every $t > 0$. Now, by \eqref{eta0}, the definition of $\|\cdot\|_{\mathcal{H}_F}$ and \eqref{i}, we deduce
	\begin{equation}\label{i1}
		\smallint_{t}^{\infty} g(s)\|\eta_{0x}(s-t)\|^2 \, ds \leq 2\|\eta^t\|_{L^2_g}^2 + 2\|\phi_x(t) + \psi(t)\|^2 \leq \frac{2M}{\omega \kappa}e^{-\gamma t} \|\eta_0\|_{L^2_g}^2,
	\end{equation}
	where we have used $\smallint_0^\infty g(s)ds=\ell\in(0,1)$ and $\omega=1-\ell\in(0,1)$.
	
	On the other hand, for each $t > 0$ we define
	\[
	N_t := \left\{ s \in \mathbb{R}^+, \, g(t+s) -   \frac{2M}{\omega \kappa} e^{-\gamma t} g(s) > 0 \right\}.
	\]
	We claim that $|N_t| = 0$, for every $t > 0$, where $|N_t|$ denotes the measure of the set $N_t$. Indeed, suppose by contradiction that there exists $t_0 > 0$ such that $|N_{t_0}| > 0$. Then,
	\begin{equation}\label{i2}
		0 < \smallint_{N_{t_0}} \left[ g(t_0 + s) - \frac{2M}{\omega \kappa}  e^{-\gamma t_0} g(s) \right] \, ds \le\smallint_0^\infty g(s)ds=\ell\in(0,1).
	\end{equation}
	
	However, from \eqref{i1},
	\begin{align*}
		0 &\geq \smallint_{t_0}^{\infty} g(s) \|\eta_{0x}(s-t_0)\|^2 \, ds - \frac{2M}{\omega \kappa}  e^{-\gamma t_0} \smallint_0^{\infty} g(s) \|\eta_{0x}(s)\|^2 \, ds \\
		&= \smallint_0^{\infty} \left[ g(t_0 + s) - \frac{2M}{\omega \kappa} e^{-\gamma t_0} g(s) \right] \|\eta_{0x}(s)\|^2 \, ds.
	\end{align*}
	Now we choose $\eta_0(s) = \chi_{N_{t_0}}(s) \phi^*$ for some $\phi^* \in H_0^1(0, L)$ such that $\|\phi_x^*\| = 1$, where $\chi_{N_{t_0}}(s) =1$ if $s\in N_{t_0}$ and $\chi_{N_{t_0}}(s) =0$ if $s\notin N_{t_0}$. Therefore,
	\begin{align*}
		\smallint_{N_{t_0}} \left[ g(t_0 + s) - \frac{2M}{\omega \kappa} e^{-\gamma t_0} g(s) \right] \, ds &\leq 0.
	\end{align*}
	which contradicts \eqref{i2}. Therefore $g$ satisfies $g(t+s) - \frac{2M}{\omega \kappa} e^{-\gamma t} g(s) > 0$, $\forall t>0$, a.e. $s>0$. So the $\delta$-condition \eqref{delta c} holds.
	
	\textit{Proof of (2):} Suppose by contradiction that $\chi_0 \neq 0$. From Theorem \ref{x1} there there exist three sequences
	\[
	\lambda_n=\frac{\sqrt b\pi^2n^2}{L\sqrt{\rho_2\pi^2n^2+\rho_1L^2}},\quad B_n=\frac{\left(\lambda_n^2\rho_1-\kappa(1-\hat{g}(\lambda_n))\gamma_n\right)(i\lambda_n\rho_3+\beta\gamma_n)-i\lambda_n\sigma^2\gamma_n}{\det\mathbf{M}_n(\lambda_n)},~~~n=1,2,\cdots
	\]
	and
	$$f_n:=\left(0,0,0,-\rho_2^{-1}\sqrt{\frac2 L}\cos\left(\frac{n\pi x}{L}\right),0,0\right)\in\mathcal{H}_F,\quad n=1,2,\cdots$$
	such that
	\[ \frac1n\left\|(i\lambda_n-\mathcal{A}_F)^{-1}f_n\right\|_{\mathcal{H}_F}\geq\sqrt{\rho_2}|B_n|\frac{|\lambda_n|}{n}\to\frac{\rho_2^2|\chi_0|\sqrt b\pi}{\rho_1Lb^2}\hbox{ as } n\to\infty.
	\]
	
	On the other hand, since $ e^{\mathcal{A}_Ft}$ is exponentially stable, we get from Theorem \ref{fz1} that \eqref{fz2_0} holds, which contradicts the above relation. Then, we must have $\chi_0 = 0$.
\end{proof}
\begin{proof}[Proof of the sufficiency  Theorem \ref{thmFourier}-1]	
	Since the $\delta$-condition \eqref{delta c} holds, by Theorem \ref{thmap}, $ i\mathbb{R}\subset \rho(\mathcal{A}_F)$. Since $\chi_0=0$, it follows from Theorem \ref{thmboundresolvent}-1 that
	\[
	\|(i\lambda - \mathcal{A}_F)^{-1}\|_{\mathcal{L}(\mathcal{H}_F)} \leq C \left(\frac{1}{I_g(\lambda)}+\frac{1}{|\lambda|}+1\right)^2
	\]
	for some constant $C>0$ independent of $\lambda$. Since $\tilde{g}\in L^1(0,\infty)$, by \eqref{lim00} in Lemma \ref{hatg} and Lemma \ref{Ig}, we get
	$$\lim_{|\lambda|\to\infty}I_g(\lambda)=\smallint_0^\infty \tilde{g}(s)ds\in(0,\infty).$$
	Then it follows
	\begin{equation}\label{jxc}
		\limsup_{|\lambda|\to\infty}\|(i\lambda - \mathcal{A}_F)^{-1}\|_{\mathcal{L}(\mathcal{H}_F)}\le C\left[\left(\smallint_0^\infty \tilde{g}(s)ds\right)^{-1}+1\right]^2.
	\end{equation}
	Hence, by Theorem \ref{fz1}, the the Fourier-semigroup $\left\{e^{\mathcal{A}_Ft}\right\}_{t\ge0}$ is  exponentially stable.
\end{proof}
\begin{proof}[Proof of Theorem \ref{thmFourier}-2] Regardless of whether $\chi_0=0$, Theorem \ref{thmboundresolvent}-1 establishes the existence of a constant $C> 0$, independent of $\lambda$ , such that
	$$\frac{\|(i\lambda - \mathcal{A}_F)^{-1}\|_{\mathcal{L}(\mathcal{H}_F)}}{|\lambda|^2} \leq C \left(\frac{1}{I_g(\lambda)}+\frac{1}{|\lambda|}+1\right)^2$$
	for any $\lambda\in\mathbb{R}$ such that $|\lambda|\ge1$. Then by the same proof as \eqref{jxc}, we have		
	\[
	\limsup_{|\lambda| \to \infty} \frac{\|(i\lambda - \mathcal{A}_F)^{-1}\|_{\mathcal{L}(\mathcal{H}_F)}}{|\lambda|^2} \leq C \left[ \left( \smallint_0^{\infty} \tilde{g}(s) \, ds \right)^{-1} + 1 \right]^2 < \infty.
	\]
	Since the $\delta$-condition \eqref{delta c} holds, by Theorem \ref{thmap}, $ i\mathbb{R}\subset \rho(\mathcal{A}_F)$.  Hence, applying Theorem \ref{fz1} with $\beta=2$, the desired conclusion follows.
\end{proof}

\begin{proof}[Proof of the necessity Theorem \ref{thmCattaneo}-1]
	Suppose that $e^{\mathcal{A}_Ct}$ is exponentially stable. We are going to prove that
	\begin{itemize}
		\item[(1)] $g$ satisfies $\delta$-condition \eqref{delta c};
		\item[(2)] $\chi_1 = 0$.
	\end{itemize}
	
	\textit{Proof of (1):} The proof of this part is mostly the same as that of (1) in the necessity proof of Theorem \ref{thmFourier}-1, so we omit the details.
	
	\textit{Proof of (2):} Suppose by contradiction that $\chi_1 \neq 0$. From Theorem \ref{x1} there three sequences
	\[
	\lambda_n=\frac{\sqrt b\pi^2n^2}{L\sqrt{\rho_2\pi^2n^2+\rho_1L^2}},\quad B_n=\frac{\left(\lambda_n^2\rho_1-\kappa(1-\hat{g}(\lambda_n))\gamma_n\right)(i\lambda_n\rho_3+\frac{\gamma_n}{i \lambda_n \tau + \beta})-i\lambda_n\sigma^2\gamma_n}{\det\mathbf{U}_n(\lambda_n)},~~~n=1,2,\cdots
	\]
	and
	$$f_n:=\left(0,0,0,-\rho_2^{-1}\sqrt{\frac2 L}\cos\left(\frac{n\pi x}{L}\right),0,0,0\right)\in\mathcal{H}_C,\quad n=1,2,\cdots$$
	such that
	\begin{equation}\label{dv} \frac1n\left\|(i\lambda_n-\mathcal{A}_C)^{-1}f_n\right\|_{\mathcal{H}_C}\geq\sqrt{\rho_2}|B_n|\frac{|\lambda_n|}{n}\to\left\{
		\begin{array}{ll}
			\displaystyle \frac{\pi\sqrt{b\rho_2}\rho_2^2|\chi_1|}{\rho_1Lb^2|\tau b\rho_3-\rho_2|} & \hbox{ if } \tau b\rho_3\neq\rho_2,\\
			\displaystyle \infty, & \hbox{ if }\tau b\rho_3=\rho_2,
		\end{array}
		\right.\hbox{ as } n\to\infty.
	\end{equation}
	
	On the other hand, since $ e^{\mathcal{A}_Ft}$ is exponentially stable, we get from Theorem \ref{fz1} that \eqref{fz2_0} holds.
	\begin{enumerate}
		\item When $\tau b\rho_3=\rho_2$, \eqref{fz2_0} and \eqref{dv} are mutually contradictory.
		\item When $\tau b\rho_3\neq \rho_2$, since $\chi_1\neq0$, \eqref{fz2_0} and \eqref{dv} are also mutually contradictory.
	\end{enumerate}
	Then, we must have $\chi_1 = 0$
\end{proof}
\begin{proof}[Proof of the sufficiency Theorem \ref{thmCattaneo}-1]
	Since the $\delta$-condition \eqref{delta c} holds, by Theorem \ref{thmap}, $ i\mathbb{R}\subset \rho(\mathcal{A}_F)$. Since
	$$\chi_1=\chi_0-\frac{\tau b}{\rho_1\rho_2}\left(\rho_1\rho_3\chi_0+\sigma^2\right)=\frac{\chi_0}{\rho_1}(\rho_2-\tau b\rho_3)-\frac{\tau b\sigma^2}{\rho_1\rho_2}=0,$$
	we must have $\rho_2-b\rho_3\tau\neq0$. Then it follows from Theorem \ref{thmboundresolvent}-3 that
	\[
	\|(i\lambda - \mathcal{A}_C)^{-1}\|_{\mathcal{L}(\mathcal{H}_C)} \leq C \left(\frac{1}{I_g(\lambda)}+\frac{1}{|\lambda|}+1\right)^2.
	\]
	Then by the same proof as \eqref{jxc}, we have
	\[
	\limsup_{|\lambda|\to\infty}\|(i\lambda - \mathcal{A}_C)^{-1}\|_{\mathcal{L}(\mathcal{H}_C)}\le C\left[\left(\smallint_0^\infty \tilde{g}(s)ds\right)^{-1}+1\right]^2.
	\]
	Hence, by Theorem \ref{fz1}, the the Cattaneo-semigroup $\left\{e^{\mathcal{A}_Ct}\right\}_{t\ge0}$ is  exponentially stable.
\end{proof}

\begin{proof}[Proof of Theorem \ref{thmCattaneo}-2]
	Whether $\chi_1=0$ or not, by Theorem \ref{thmboundresolvent}-2, we get there exists $C> 0$ independent of $\lambda$ such that
	$$
	\frac{\|(i\lambda - \mathcal{A}_C)^{-1}\|_{\mathcal{L}(\mathcal{H}_C)}}{|\lambda|^2} \leq C\left(\frac{1}{I_g(\lambda)}+\frac{1}{|\lambda|}+1\right)^2$$
	for any $\lambda\in\mathbb{R}$ such that $|\lambda|\ge1$. Then by the same proof as \eqref{jxc}, we have
	
	\[
	\lim_{|\lambda| \to \infty} \frac{\|(i\lambda - \mathcal{A}_C)^{-1}\|_{\mathcal{L}(\mathcal{H}_C)}}{|\lambda|^2} \leq C \left[ \left( \smallint_0^{\infty} \tilde{g}(s) \, ds \right)^{-1} + 1 \right]^2 < \infty.
	\]
	Since the $\delta$-condition \eqref{delta c} holds, by Theorem \ref{thmap}, $ i\mathbb{R}\subset \rho(\mathcal{A}_F)$.  Hence, applying Theorem \ref{fz1} with $\beta=2$, the desired conclusion follows.
\end{proof}

\end{document}